\documentclass[12pt]{article}
\usepackage[centertags]{amsmath}
\usepackage{amsfonts}
\usepackage{amssymb}
\usepackage{amsthm}
\usepackage{amsmath,mathrsfs}
\usepackage{amsmath,amscd}
\title{\textbf{The Kontsevich Integral in Book Notation}}
\author{Renaud Gauthier\\ \\Kansas State University}
\theoremstyle{definition}
\newtheorem*{acknowledgments}{Aknowledgments}

\newtheorem{DefZf}{Definition}[subsection]

\newtheorem{thekill}[DefZf]{Remark}
\newtheorem{caution}{Note}[subsection]

\newtheorem{xLemma1&2}[caution]{Remark}

\newtheorem{Zfmultitop}[caution]{Proposition}

\newtheorem{pZfeqZ}[caution]{Theorem}
\newtheorem{ZfindofM}[caution]{Proposition}

\newtheorem{cf2pm^2eq1}[ZDTeqDZT]{Lemma}

\newtheorem{Ltilde}{Definition}[subsubsection]
\newtheorem{Thm}{Theorem}[section]
\newtheorem{Iso1}{Proposition}[section]

\newtheorem{q-tangle}[word]{Definition}

\newtheorem{Zhumpvert}{Proposition}[subsection]
\newtheorem{rmkslice}[Zhumpvert]{Remark}

\newtheorem{humptos}[Zhumpvert]{Proposition}
\newtheorem{Zhumpget}[Zhumpvert]{Proposition}
\newtheorem{stohump}[Zhumpvert]{Proposition}
\newtheorem{humptohump}[Zhumpvert]{Proposition}
\newtheorem{Iso3}{Proposition}[subsection]
\newtheorem{Iso4}{Proposition}[subsection]
\newtheorem{Or}{Proposition}[section]
\newtheorem{ZfSLeqSZfL}{Proposition}[subsection]
\newtheorem{ZfhSLeqSZfhL}[ZfSLeqSZfL]{Theorem}
\newtheorem{ZftSLeqSZftL}[ZfSLeqSZfL]{Theorem}
\newtheorem{Zfta}{Proposition}[section]
\newtheorem{bareZfta}[Zfta]{Proposition}
\newtheorem{generalbareZfta}[Zfta]{Corollary}
\newtheorem{Xpm}[Zfta]{Lemma}
\newtheorem{!book}[Zfta]{Proposition}
\newtheorem{detect}[Zfta]{Lemma}
\newtheorem{Rmkdetect}[Zfta]{Remark}
\newtheorem{findlocmax}[Zfta]{Lemma}
\newtheorem{maxX}[Zfta]{Proposition}
\newtheorem{detectXpm}[Zfta]{Proposition}
\newtheorem{detectkof}[Zfta]{Lemma}

\newtheorem{exdetect}[Zfta]{Example}
\newtheorem{plat}[Zfta]{Proposition}
\newtheorem{findlocmin}[Zfta]{Lemma}
\newtheorem{Recovery}[Zfta]{Theorem}

\newtheorem{Rmkidz}[Zfta]{Remark}
\newtheorem{Goodbehavior}[DefZf]{Theorem}

\DeclareMathOperator*{\stck}{\times}

\newcommand{\beq}{\begin{equation}}
\newcommand{\eeq}{\end{equation}}
\newcommand{\eps}{\epsilon}
\newcommand{\varz}{\vartriangle \!\! z}
\newcommand{\twopii}{2 \pi i}
\newcommand{\barA}{\overline{\mathcal{A}}}
\newcommand{\hatA}{\hat{\mathcal{A}}}
\newcommand{\dlog}{\text{dlog}}

\begin{document}
\maketitle
\begin{abstract}
We introduce a matrix representation of a chord on a tangle which leads us to representing tangle chord diagrams as stacks of matrices that we call books. We show that band sum moves, Reidemeister moves as well as orientation changes, are implemented on $\widetilde{Z}_f(L)$ (\cite{RG}) by matrix congruences. We prove that being given the bare framed Kontsevich integral $Z_f(L)$ in book notation for some unknown link $L$, we can determine what the link $L$ is, as well as the projection of $Z_f(L)$ in the original completed algebra of chord diagrams $\overline{\mathcal{A}}(\amalg S^1)$.
\end{abstract}

\newpage

\section{Introduction}
In \cite{RG} we introduced an isotopy invariant $\widetilde{Z}_f$ that is well-behaved under band sum moves, a first step towards defining an invariant of 3-manifolds, also defined in the same paper. Practically, we are now seeking a way to write the Kontsevich integral in a more compact form, with the hope that deformations of link components, orientation changes as well as band sum moves can be easily implemented on $\widetilde{Z}_f(L)$ in this new form. We introduce a representation of tangle chord diagrams by stacks of matrices that we refer to as books. $\widetilde{Z}_f(L)$ written in book notation turns out to be exceptionally well-behaved under all these moves. The main idea consists in considering, for each chord on a link, a block matrix, with blocks corresponding to link components. Within each block, indexation of matrix components is done by considering a given link $L$, and sectioning off the whole link into vertical strips. The drawing of lines delimiting strips is performed by considering each local extremum. For a given local extremum, we draw a vertical line going through the local extremum, as well as two neighboring lines, one on the right, and one on the left of the local extremum, so as to ensure that no two local extrema can end up in the same strip. Armed with this formalism, we simply replace tangle chord diagrams by their corresponding books in the expression for $\widetilde{Z}_f(L) \in \overline{\mathcal{A}}(L)$ to yield $\widetilde{Z}_f(L)$ in book notation. The partitioning of links into vertical strips being dependent on the chosen link, for a given $q$-components link $L$ giving rise to $N$ vertical strips, we work with pages that are $qN \times qN$ matrices. Thus for one such link $L$ we have $Z_f(L)$ in the completion of the algebra:
\beq
\oplus_{m \geq 0}\Big( \text{Mat}_{qN \times qN}   \Big)^m := \mathcal{B}[qN]
\eeq
which we denote by $\overline{\mathcal{B}}[qN]$. By increasing the size of matrices, it is clear that we consider a sequence of such completed algebra with embeddings as in:
\beq
\overline{\mathcal{B}}[1] \rightarrow \overline{\mathcal{B}}[2] \rightarrow \cdots \rightarrow \overline{\mathcal{B}}[n] \rightarrow \cdots
\eeq
that we denote by $\overline{\mathcal{B}}$. Such an element as $Z_f(L)$, initially in $\overline{\mathcal{B}}[qN]$, can be mapped to $\overline{\mathcal{B}}[(q+1)(N+4)]$ by adding a trivial circle to it. If we add a second trivial circle to $Z_f(L)$, we have a resulting element in $\overline{\mathcal{B}}[(q+2)(N+8)]$. Proceeding in this fashion we define what we call the thread of $Z_f(L)$, and we represent it by:
\beq
\overline{\mathcal{B}}[qN] \rightarrow \overline{\mathcal{B}}[(q+1)(N+4)] \rightarrow \cdots \rightarrow \overline{\mathcal{B}}[(q+m)(N+4m)] \rightarrow \cdots
\eeq
the map from one subalgebra to the next being given by the simple addition of a trivial circle. Using this concept that we call threading, we can determine some information encoded in $Z_f(L)$ as an element of $\overline{\mathcal{B}}[qN]$ that was otherwise inaccessible. We do not repeat the work of \cite{RG} concerning the behavior of $\check{Z}_f$ and $\widetilde{Z}_f$ under handle slide, and we will just recall the construction of $\widetilde{Z}_f$, at least as much as is required to develop the book formalism.\\

In section 2 we define the framed Kontsevich integral $\widetilde{Z}_f$ and introduce the requisite theory. In section 3 we present the book notation, an alternate representation of tangle chord diagrams of degree $n$ by stacks of $n$ matrices, each matrix giving the position of one chord. We study the behavior of $\widetilde{Z}_f$ in book notation under band sum moves, Reidemeister moves and orientation changes in sections 4, 5 and 6. In section 7, we prove that $Z_f$ in book notation is faithful on links, from which we recover the original expression of $Z_f(L) \in \overline{\mathcal{A}}(\amalg S^1)$. \\

\begin{acknowledgments}
The author would like to thank D.Yetter for fruitful discussions and for spending long hours going over the intricacies of this work, as well as D.Auckly and V.Turchin for very helpful conversations.
\end{acknowledgments}

\section{Preliminary definitions and construction of $\widetilde{Z}_f$}

\subsection{The algebra $\mathcal{A}$ of chord diagrams}
For a singular oriented knot whose only singularities are transversal self-intersections, the preimage of each singular crossing under the embedding map defining the knot yields two distinct points on $S^1$. Each singular point in the image therefore yields a pair of points on $S^1$ that we conventionally connect by a chord for book keeping purposes \cite{K}. A knot with $m$ singular points will yield $m$ distinct chords on $S^1$. We refer to such a circle with $m$ chords on it as a chord diagram of degree $m$, the degree being the number of chords. The support of the graph is an oriented $S^1$, and it is regarded up to orientation preserving diffeomorphisms of the circle. More generally, for a singular oriented link all of whose singularities are double-crossings, the preimage of each singular crossing under the embedding map defining the link yields pairs of distinct points on possibly different circles depending on whether the double crossing was on a same component or between different components of the link. We also connect points making a pair by a chord. A $q$-components link with $m$ singular points will yield $m$ chords on $\coprod^q S^1$. We still call such a graph a chord diagram. The support now is $\coprod^q S^1$ regarded up to orientation preserving diffeomorphism of each $S^1$. We denote by $\mathcal{D}(\amalg^q S^1)$ the $\mathbb{C}$-vector space spanned by chord diagrams with support on $\amalg^q S^1$. We write $\mathcal{D}$ for $\mathcal{D}(S^1)$. There is a grading on $\mathcal{D}(\amalg^q S^1)$ given by the number of chords featured in a diagram. If $\mathcal{D}^{(m)}(\amalg^q S^1)$ denotes the subspace of chord diagrams of degree $m$, then we can write:
\beq
\mathcal{D}(\amalg^q S^1)=\oplus_{m\geq 0} \mathcal{D}^{(m)}(\amalg^q S^1)
\eeq
We demand that chord diagrams with degree greater than 2 satisfy the 4-T relation which locally looks like:\\
\setlength{\unitlength}{1cm}
\begin{picture}(1,3)(-1,-0.5)
\multiput(0,0.75)(0.2,0){5}{\line(1,0){0.1}}
\multiput(1,0.25)(0.2,0){5}{\line(1,0){0.1}}
\put(0,0.9){\vector(0,1){0.2}}
\put(1,0.9){\vector(0,1){0.2}}
\put(2,0.9){\vector(0,1){0.2}}
\linethickness{0.3mm}
\put(0,0){\line(0,1){1}}
\put(1,0){\line(0,1){1}}
\put(2,0){\line(0,1){1}}
\put(2.5,0.5){$+$}
\end{picture}
\setlength{\unitlength}{1cm}
\begin{picture}(1,3)(-3,-0.5)
\multiput(0,0.75)(0.2,0){5}{\line(1,0){0.1}}
\multiput(0,0.25)(0.2,0){10}{\line(1,0){0.1}}
\put(0,0.9){\vector(0,1){0.2}}
\put(1,0.9){\vector(0,1){0.2}}
\put(2,0.9){\vector(0,1){0.2}}
\linethickness{0.3mm}
\put(0,0){\line(0,1){1}}
\put(1,0){\line(0,1){1}}
\put(2,0){\line(0,1){1}}
\put(2.5,0.5){$=$}
\end{picture}
\setlength{\unitlength}{1cm}
\begin{picture}(1,3)(-5,-0.5)
\multiput(0,0.25)(0.2,0){5}{\line(1,0){0.1}}
\multiput(1,0.75)(0.2,0){5}{\line(1,0){0.1}}
\put(0,0.9){\vector(0,1){0.2}}
\put(1,0.9){\vector(0,1){0.2}}
\put(2,0.9){\vector(0,1){0.2}}
\linethickness{0.3mm}
\put(0,0){\line(0,1){1}}
\put(1,0){\line(0,1){1}}
\put(2,0){\line(0,1){1}}
\put(2.5,0.5){$+$}
\end{picture}
\setlength{\unitlength}{1cm}
\begin{picture}(1,3)(-7,-0.5)
\multiput(0,0.25)(0.2,0){5}{\line(1,0){0.1}}
\multiput(0,0.75)(0.2,0){10}{\line(1,0){0.1}}
\put(0,0.9){\vector(0,1){0.2}}
\put(1,0.9){\vector(0,1){0.2}}
\put(2,0.9){\vector(0,1){0.2}}
\linethickness{0.3mm}
\put(0,0){\line(0,1){1}}
\put(1,0){\line(0,1){1}}
\put(2,0){\line(0,1){1}}
\end{picture}\\ \\
where solid lines are intervals on $\amalg^q S^1$ on which a chord foot rests, and arrows indicate the orientation of each strand. We demand that chord diagrams also satisfy the framing independence relation: if a chord diagram has a chord forming an arc on $S^1$ with no other chord ending in between its feet, then the chord diagram is set to zero. The resulting space is the $\mathbb{C}$-vector space generated by chord diagrams mod the 4-T relation and framing independence and is denoted by $\mathcal{A}(\amalg^q S^1)$. We write $\mathcal{A}$ for $\mathcal{A}(S^1)$. The grading of $\mathcal{D}(\amalg^q S^1)$ is preserved while modding out by the 4-T and framing independence relations, inducing a grading on $\mathcal{A}(\amalg^q S^1)$:
\beq
\mathcal{A}(\amalg^q S^1)=\oplus_{m \geq 0}\mathcal{A}^{(m)}(\amalg^q S^1)
\eeq
where $\mathcal{A}^{(m)}(\amalg^q S^1)$ is obtained from $\mathcal{D}^{(m)}(\amalg^q S^1)$ by modding out by the 4-T and the framing independence relations. The connected sum of circles can be extended to chorded circles, thereby defining a product on $\mathcal{A}$ that we denote by $\cdot$, making it into an algebra that is associative and commutative ~\cite{DBN}. More generally $\mathcal{A}(\amalg^q S^1)$ is a module over $\otimes^q \mathcal{A}$. The Kontsevich integral will be valued in the graded completion $\barA(\amalg^q S^1)$ of the algebra $\mathcal{A}(\amalg^q S^1)$.\\

\subsection{Original definition of the Kontsevich integral}
As far as knots are concerned, we will work with Morse knots, and for that purpose we consider the following decomposition of $\mathbb{R}^3$ as the product of the complex plane and the real line: $\mathbb{R}^3=\mathbb{R}^2\times \mathbb{R} \simeq \mathbb{C}\times \mathbb{R}$, with local coordinates $z$ in the complex plane and $t$ on the real line for time. A Morse knot $K$ is such that $t\circ K$ is a Morse function on $S^1$. If we denote by $Z$ the Kontsevich integral functional on knots, if $K$ is a Morse knot, we define ~\cite{K}, ~\cite{DBN}, ~\cite{CL2}:
\beq
Z(K):=\sum_{m\geq 0} \frac{1}{(2 \pi i)^m}\int_{t_{min}< t_1<...<t_m<t_{max}}\sum_{P\; applicable}(-1)^{\varepsilon(P)}D_P\prod_{1\leq i \leq m}\dlog \vartriangle \!\!z(t_i)[P_i] \label{orZK} \nonumber
\eeq
where $t_{min}$ and $t_{max}$ are the min and max values of $t$ on $K$ respectively, $P$ is an $m$-vector each entry of which corresponds to a pair of points on the image of the knot $K$. We write $P=(P_1,...,P_m)$, where the $i$-th entry $P_i$ corresponds to a pair of points on the knot at height $t_i$, and we can denote these two points by $z_i$ and $z'_i$, so that we can write $\varz(t_i)[P_i]:=z_i-z'_i$, and we refer to such $P$'s as pairings. We denote by $K_P$ the knot $K$ with $m$ pairs of points placed on it following the prescription given by $P$, and then connecting points at a same height by a chord. A pairing is said to be applicable if each entry corresponds to a pair of two distinct points on the knot, at the same height ~\cite{DBN}. For a pairing $P=(P_1,\cdots,P_m)$ giving the position of $m$ pairs of points on $K$, we denote by $\varepsilon(P)$ the number of those points ending on portions of $K$ that are locally oriented down. For example if $P=(z(t),z'(t))$ and $K$ is locally oriented down at $z(t)$, then $z(t)$ will contribute 1 to $\varepsilon(P)$. We also define the length of $P=(P_1,\cdots,P_m)$ to be $|P|=m$. If we denote by $\iota_K$ the embedding defining the knot then $D_P$ is defined to be the chord diagram one obtains by taking the inverse image of $K_P$ under $\iota_K$: $D_P =\iota_K ^{-1} K_P$. This generalizes immediately to the case of Morse links, and in this case the geometric coefficient will not be an element of $\barA$ but will be an element of $\barA (\coprod_q S^1)$ if the argument of $Z$ is a $q$-components link.\\

Now if we want to make this integral into a true knot invariant, then we correct it as follows. Consider the embedding in $S^3$ of the trivial knot as:
\beq
U=
\setlength{\unitlength}{0.3cm}
\begin{picture}(5,4)(-1,1)
\thicklines
\put(1,3){\oval(2,2)[t]}
\put(3,3){\oval(2,2)[b]}
\put(5,3){\oval(2,2)[t]}
\put(3,3){\oval(6,5)[b]}
\end{picture}
\eeq
We now make the following correction ~\cite{K}:
\beq
\hat{Z}:=Z(\:\setlength{\unitlength}{0.1cm}
\begin{picture}(5,6)
\thicklines
\put(1,2){\oval(2,2)[t]}
\put(3,2){\oval(2,2)[b]}
\put(5,2){\oval(2,2)[t]}
\put(3,2){\oval(6,5)[b]}
\end{picture}\;)^{-m}.Z \label{modK}
\eeq
where the dot is the product on chord diagrams extended by linearity, and $m$ is a function that captures the number of maximal points of any knot $K$ that is used as an argument of $Z$. Defining $\nu := Z(U)^{-1}$, this reads $\hat{Z}=\nu^m \cdot Z$. Equivalently, we can define $\hat{Z}$ as being $Z$ with the provision that $\nu$ acts on each maximal point of a given knot $K$ in the expression for $Z(K)$. As pointed out in ~\cite{SW1}, there are two possible corrections to the Kontsevich integral: the one we just presented, and the other one obtained by using $1-m$ as an exponent of the Kontsevich integral of the hump instead of just $-m$. In this manner the corrected version is multiplicative under connected sum, while using the above correction it behaves better under cabling operations. We will argue later that a modified version of $\hat{Z}$ using $1-m$ as an exponent of $\nu$ is the right object to consider for handle slide purposes and the construction of topological invariants of 3-manifolds.  In the case of links, there will be one such correction for each component of the link, with a power $m_i$ on the correction term for the $i$-th component, where $m_i$ is the number of maximal points of the $i$-th link component. Equivalently, $\hat{Z}(L)$ is the same as $Z(L)$ save that every $i$-th link component in the expression for $Z(L)$ is multiplied by $\nu^{m_i}$, $1 \leq i \leq q$.\\

\subsection{The Kontsevich Integral of tangles}
We generalize the original definition of the Kontsevich integral of knots to the Kontsevich integral of tangles as discussed in ~\cite{DBN}, ~\cite{LM1}, ~\cite{ChDu}.\\

For this purpose, we will define a slightly more general algebra of chord diagrams ~\cite{LM5}: For $X$ a compact oriented 1-dimensional manifold with labeled components, a chord diagram with support on $X$ is the manifold $X$ together with a collection of chords with feet on $X$. We represent such chord diagrams by drawing the support $X$ as solid lines, the graph consisting of dashed chords. We introduce an equivalence relation on the set of all chord diagrams: two chord diagrams $D$ and $D'$ with support on $X$ are equivalent if there is a homeomorphism $f:D \rightarrow D'$ such that the restriction $f|_X$ of $f$ to $X$ is a homeomorphism of $X$ that preserves components and orientation. We denote by $\mathcal{A}(X)$ the complex vector space spanned by chord diagrams with support on $X$ modulo the 4-T and framing independence relations. $\mathcal{A}(X)$ is still graded by the number of chords as:
\beq
\mathcal{A}(X)=\oplus_{m \geq 0} \mathcal{A}^{(m)}(X)
\eeq
where $\mathcal{A}^{(m)}(X)$ is the complex vector space spanned by chord diagrams of degree $m$. We write $\overline{\mathcal{A}}(X)$ for the graded completion of $\mathcal{A}(X)$.  We define a product on $\mathcal{A}(X)$ case by case. For example, if $X=\amalg^{q >1}S^1$, there is no well defined product defined on $\mathcal{A}(X)$. If $X=I^N$, the concatenation induces a well defined product on $\mathcal{A}(I^N)$. The product of two chord diagrams $D_1$ and $D_2$ in this case is defined by putting $D_1$ on top of $D_2$ and is denoted by $D_1 \times D_2$. More generally, for $D_i \in \mathcal{A}(X_i)$, $i=1,2$, $D_1 \times D_2$ is well-defined if $X_1$ and $X_2$ can be glued strand-wise. One chord diagram of degree 1 we will use repeatedly is the following:
\beq
\Omega_{ij}=
\setlength{\unitlength}{0.5cm}
\begin{picture}(3,3)
\multiput(0.4,0.5)(0.3,0){3}{\circle*{0.15}}
\multiput(1.4,1)(0.3,0){7}{\line(1,0){0.1}}
\multiput(3.9,0.5)(0.3,0){3}{\circle*{0.15}}
\put(0.1,-0.5){$\text{\tiny 1}$}
\put(1.5,-0.5){$\text{\tiny i}$}
\put(3.6,-0.5){$\text{\tiny j}$}
\put(5,-0.5){$\text{\tiny N}$}
\thicklines
\put(0,0){\vector(0,1){1.6}}
\put(1.4,0){\vector(0,1){1.6}}
\put(3.5,0){\vector(0,1){1.6}}
\put(4.9,0){\vector(0,1){1.6}}
\linethickness{0.5mm}
\put(0,0){\line(0,1){2}}
\put(1.4,0){\line(0,1){2}}
\put(3.5,0){\line(0,1){2}}
\put(4.9,0){\line(0,1){2}}
\end{picture}\label{Omij}
\eeq
\\
For $T$ a tangle, we define $Z(T) \in \overline{\mathcal{A}}(T)$ by:
\beq
Z(T):=\sum_{m\geq 0} \frac{1}{(2 \pi i)^m}\int_{t_{min}< t_1<...<t_m<t_{max}}\sum_{P\; applicable}(-1)^{\varepsilon(P)}T_P\prod_{1\leq i \leq m}\dlog \vartriangle \!\!z[P_i] \nonumber
\eeq
exactly as we define $Z(K)$ in ~\eqref{orZK} with the difference that $T_P$ is the tangle $T$ with $m$ chords placed on it following the prescription given by $P$. Following ~\cite{ChDu}, we refer to $T_P$ as a tangle chord diagram. We define the Kontsevich integral of a tangle $T$ to be trivial if $Z(T)=T$. When working with links, we will sometimes omit the orientation on link components for convenience unless it is necessary to specify them.\\

We will sometimes need the map $S$ on chord diagrams \cite{LM5}: suppose $C$ is a component of $X$. If we reverse the orientation of $C$, we get another oriented manifold from $X$ that we will denote by $X'$. This induces a linear map:
\beq
S_{(C)}:\mathcal{A}(X) \rightarrow \mathcal{A}(X') \label{SC}
\eeq
defined by associating to any chord diagram $D$ in $\mathcal{A}(X)$ the element $S_{(C)}(D)$ obtained from $D$ by reversing the orientation of $C$ and multiplying the resulting chord diagram by $(-1)^m$ where $m$ is the number of vertices of $D$ ending on the component $C$. Suppose $Z(T)$ is known for some oriented tangle $T$, and $T'$ is an oriented tangle with the same skeleton as $T$'s, but with possible reversed orientations on some of its components. Then one can find $Z(T')$ by symply applying $S_C(Z(T))$ iteratively as many times as there are components $C$ of $T'$ that have an orientation different from that of $T$.\\

\subsection{Integral of framed oriented links}
In the framed case, we no longer impose the framing independence relation. It follows that when we compute the Kontsevich integral $Z(T)$ of a tangle $T$ with local extrema, we will run into computational problems. For instance, the Kontsevich integral of the following tangle:
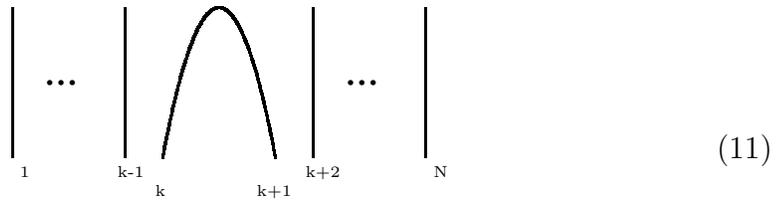
\beq
\setlength{\unitlength}{0.5cm}
\begin{picture}(13,5)
\thicklines
\put(0,0){\line(0,1){4}}
\put(3,0){\line(0,1){4}}
\put(8,0){\line(0,1){4}}
\put(11,0){\line(0,1){4}}
\multiput(1,2)(0.3,0){3}{\circle*{0.15}}
\multiput(9,2)(0.3,0){3}{\circle*{0.15}}
\qbezier(4,0)(5.5,8)(7,0)
\put(0.2,-0.5){\text{\tiny 1}}
\put(2.8,-0.5){\text{\tiny k-1}}
\put(3.8,-1){\text{\tiny k}}
\put(6.5,-1){\text{\tiny k+1}}
\put(7.8,-0.5){\text{\tiny k+2}}
\put(11.2,-0.5){\text{\tiny N}}
\end{picture}
\eeq
\\
\\
has possible integrands $\dlog(z_k-z_{k+1})$ in degree 1 corresponding to the chord diagram:\\
\beq
\setlength{\unitlength}{0.5cm}
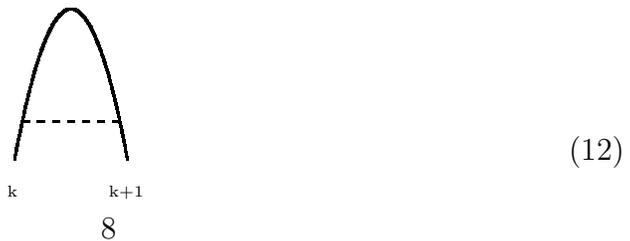
\begin{picture}(13,5)
\thicklines
\qbezier(4,0)(5.5,8)(7,0)
\put(3.8,-1){\text{\tiny k}}
\put(6.5,-1){\text{\tiny k+1}}
\multiput(4.2,1)(0.4,0){7}{\line(1,0){0.2}}
\end{picture}
\eeq \\ \\
where $z_k$ and $z_{k+1}$ are local coordinates on the strands indexed by $k$ and $k+1$ respectively, and such integrands are made to vanish by virtue of the framing independence in the unframed case:
\beq
\setlength{\unitlength}{0.5cm}
\begin{picture}(5,3)(-2,0)
\thicklines
\multiput(0,0)(2,0){2}{\line(0,1){1}}
\put(1,1){\oval(2,2)[t]}
\multiput(0,1)(0.2,0){10}{\line(1,0){0.1}}
\end{picture}=\quad 0
\eeq
However, in the framed setting, we do not impose this relation, and we therefore have to make these integrals convergent. We proceed as follows ~\cite{LM1}: if $\omega$ is the chord diagram defined  by:
\beq
\setlength{\unitlength}{0.5cm}
\begin{picture}(9,5)
\thicklines
\qbezier(0,0)(3,8)(6,0)
\multiput(1.2,2)(0.5,0){8}{\line(1,0){0.2}}
\put(-2,2){$\omega =$}
\end{picture} \nonumber
\eeq
\\
then we can define $\varepsilon^{\pm \omega/(\twopii)}$ for some $\varepsilon \in \mathbb{R}$ as a formal power series expansion of $\exp(\pm\frac{\omega}{\twopii}\log\varepsilon)$ where:
\beq
\setlength{\unitlength}{0.5cm}
\begin{picture}(9,5)
\thicklines
\qbezier(0,0)(3,8)(6,0)
\multiput(1.7,3)(0.5,0){6}{\line(1,0){0.2}}
\multiput(0.7,1)(0.5,0){10}{\line(1,0){0.2}}
\multiput(3,1.5)(0,0.5){3}{\circle*{0.15}}
\put(-2,2){$\omega^n =$}
\put(6,1.8){\Bigg\}}
\put(7,2){$n$}
\end{picture} \nonumber
\eeq
\\
We first define the following tangles and chord diagrams:
\begin{align}
T_a=
\setlength{\unitlength}{0.3cm}
\begin{picture}(5,7)
\thicklines
\put(0,0){\line(0,1){4}}
\put(4,0){\line(0,1){4}}
\put(3.5,2){\vector(-1,0){3}}
\put(0.5,2){\vector(1,0){3}}
\put(1.5,0.5){$a$}
\put(2,4){\oval(4,4)[t]}
\end{picture} &\qquad
\omega=
\setlength{\unitlength}{0.3cm}
\begin{picture}(8,4)
\thicklines
\put(2,0){\oval(4,4)[t]}
\multiput(0.3,1)(0.4,0){9}{\circle*{0.2}}
\put(3.5,0){\vector(-1,0){3}}
\put(0.5,0){\vector(1,0){3}}
\put(1.5,-1){$\mu$}
\end{picture}\\
T_a ^{\mu}&=
\setlength{\unitlength}{0.3cm}
\begin{picture}(8,7)(0,3)
\thicklines
\put(0,0){\line(0,1){3}}
\put(0,0){\vector(0,1){2}}
\put(7,0){\line(0,1){3}}
\put(7,3){\vector(0,-1){1}}
\put(2,3){\oval(4,4)[tl]}
\put(5,3){\oval(4,4)[tr]}
\put(2,6){\oval(2,2)[br]}
\put(5,6){\oval(2,2)[bl]}
\put(2,6.5){\line(1,0){3}}
\put(2,6.5){\vector(1,0){1}}
\put(5,6.5){\vector(-1,0){1}}
\put(3,7){$\mu$}
\put(6.5,-0.5){\vector(-1,0){6}}
\put(0.5,-0.5){\vector(1,0){6}}
\put(3,-2){$a$}
\end{picture}
\end{align}
\\
\\
\\
It is convenient to define a formal tangle chord diagram consisting of a single chord stretching between two strands, and to call such a graph by $\Omega$. This enables one to write $T_a^{\mu}$ above with $m$ chords on it as $T_a^{\mu} \times \Omega^m$ or $\Omega^m \times T_a ^{\mu}$. We would also have $\omega$ above written simply as $T_{\mu} \times \Omega$. If we define $\hat{\mathcal{A}}=\mathcal{D}/\text{\small 4-T}$ using the notation of \cite{K}, then we can define $S: \mathcal{A} \rightarrow \hat{\mathcal{A}}$ to be the standard inclusion algebra map that maps elements of $\mathcal{A}$ to the subspace of $\hat{\mathcal{A}}$ in which all basis chord diagrams with an isolated chord have coefficient zero ~\cite{BNGRT}, ~\cite{ChDu}.
\begin{DefZf}[\cite{LM1}] \label{DefNormZf}
The normalization for $Z_f$ of the above tangle $T_a$ is defined as:
\beq
Z_f(T_a):=\lim_{\mu \rightarrow 0} \mu^{\omega/(2 \pi i)} \times SZ(T_a ^{\mu}) \label{Zfdef}
\eeq
For a local minimum, if we take our tangle $T^a$ to be a single local minimum, then we use the normalization $Z_f(T^a)=\lim_{\mu \rightarrow 0}SZ(T_{\mu}^a)\times \mu^{-\omega/(2 \pi i)}$. In that expression $\omega$ and $T_{\mu}^a$ are our old $\omega$ and $T^{\mu}_a$ respectively, flipped upside down.
\end{DefZf}
This is the starting point for our full development of a theory of the framed Kontsevich integral \cite{RG}. Recall that we showed this definition amounts to writing:
\beq
Z_f(T_a)=
\setlength{\unitlength}{0.2cm}
\begin{picture}(5,3)(-1,0)
\thicklines
\put(2,0){\oval(4,4)[t]}
\put(3.5,0){\vector(-1,0){3}}
\put(0.5,0){\vector(1,0){3}}
\put(1.7,-1.3){$\text{\tiny 1}$}
\end{picture}
 \times SZ(T_a ^1)=:
\setlength{\unitlength}{0.2cm}
\begin{picture}(5,3)(-1,0)
\thicklines
\put(2,0){\oval(4,4)[t]}
\put(3.5,0){\vector(-1,0){3}}
\put(0.5,0){\vector(1,0){3}}
\put(1.7,-1.3){$\text{\tiny 1}$}
\end{picture} \times SZ(T_a ^{\text{\tiny $\;$1-resolved}})
\eeq
where $T_a ^{\text{\tiny 1-resolved}}$ is $T_a$ seen as being analytically probed by isolated chords near its local maximum, leading to a local resolution of that local maximum into a spout of opening width 1. We also have:
\beq
Z_f(T^a)= SZ(T_1 ^a) \times
\setlength{\unitlength}{0.2cm}
\begin{picture}(5,3)(-1,0)
\thicklines
\put(2,2){\oval(4,4)[b]}
\put(3.5,2){\vector(-1,0){3}}
\put(0.5,2){\vector(1,0){3}}
\put(1.7,2.2){$\text{\tiny 1}$}
\end{picture} = SZ( T^{a \; \text{\tiny $\;$1-resolved}}) \times
\setlength{\unitlength}{0.2cm}
\begin{picture}(5,3)(-1,0)
\thicklines
\put(2,2){\oval(4,4)[b]}
\put(3.5,2){\vector(-1,0){3}}
\put(0.5,2){\vector(1,0){3}}
\put(1.7,2.2){$\text{\tiny 1}$}
\end{picture}
\eeq
This led us to generalize the definition of $Z_f$ of local extrema to:
\beq
Z_f[\text{\small M}](T_a)=
\setlength{\unitlength}{0.2cm}
\begin{picture}(5,3)(-1,0)
\thicklines
\put(2,0){\oval(4,4)[t]}
\end{picture}
 \times SZ(
\setlength{\unitlength}{0.3cm}
\begin{picture}(5,4)(-1,0)
\thicklines
\qbezier(0,0)(1,2)(0,3)
\qbezier(3,0)(2,2)(3,3)
\put(2.5,0){\vector(-1,0){2}}
\put(0.5,0){\vector(1,0){2}}
\put(1,-1){$a$}
\put(2.5,3){\vector(-1,0){2}}
\put(0.5,3){\vector(1,0){2}}
\put(1,3.3){$M$}
\end{picture})=:
\setlength{\unitlength}{0.2cm}
\begin{picture}(5,3)(-1,0)
\thicklines
\put(2,0){\oval(4,4)[t]}
\end{picture}
 \times SZ(T_a ^{\text{\tiny $\;M$-resolved}})
\eeq
and:
\beq
Z_f[\text{\small M}](T^a)=SZ(
\setlength{\unitlength}{0.3cm}
\begin{picture}(5,4)(-1,0)
\thicklines
\qbezier(0,0)(1,1)(0,3)
\qbezier(3,0)(2,1)(3,3)
\put(2.5,0){\vector(-1,0){2}}
\put(0.5,0){\vector(1,0){2}}
\put(1,-1.3){$M$}
\put(2.5,3){\vector(-1,0){2}}
\put(0.5,3){\vector(1,0){2}}
\put(1,3.3){$a$}
\end{picture}) \times
\setlength{\unitlength}{0.2cm}
\begin{picture}(5,3)(-1,0)
\thicklines
\put(2,2){\oval(4,4)[b]}
\end{picture}
=:SZ(T^{a \text{\tiny $\;M$-resolved}})
\times
\setlength{\unitlength}{0.2cm}
\begin{picture}(5,3)(-1,0)
\thicklines
\put(2,2){\oval(4,4)[b]}
\end{picture}
\eeq
\\
Without loss of generality, we can focus on a local maximum. An equivalent definition would be:
\begin{align}
Z_f[\text{\small M}](T_a)&=
\setlength{\unitlength}{0.2cm}
\begin{picture}(5,3)(-1,0)
\thicklines
\put(2,0){\oval(4,4)[t]}
\end{picture}
\times \lim_{(\mu/M) \rightarrow 0} \Big( \mu/M \Big)^{\Omega/2 \pi i} \times T^M_{\mu} \times SZ(T_a ^{\mu}) \label{limmuM}\\
&=\setlength{\unitlength}{0.2cm}
\begin{picture}(5,3)(-1,0)
\thicklines
\put(2,0){\oval(4,4)[t]}
\end{picture} \times \lim_{\xi \rightarrow 0} \xi^{\Omega/2 \pi i} \times T_{M\xi}^M \times SZ( T_a ^{M \xi})
\end{align}
Definition \ref{DefNormZf} being a special case thereof for which $M=1$.
\begin{thekill}
In ~\cite{LM1}, Le and Murakami do not generalize $Z_f$ to tangles with more than two strands. They take their definition to hold irrespective of other strands, very much in line with defining $Z(T, \sigma, \tau)$ for a pre-q-tangle $(T, \sigma, \tau)$ as $lim_{\epsilon \rightarrow 0} \epsilon_{\sigma, r}^{-1} Z(T_{\sigma, \tau, \epsilon}) \epsilon_{\tau, s}$ where $\epsilon_{\sigma, r}$ and $\epsilon_{\tau, s}$ are computed by ignoring strands other than those that make an elementary tangle non-trivial. We generalize their definition to the case of tangles with more than one strand with the aim to actually recovering the Kontsevich integral. We computed:
\end{thekill}
\begin{Zfmultitop}(\cite{RG})  \label{Zfmultitop}
For $M>0$,
\beq
Z_f[\text{\small M}](
\setlength{\unitlength}{0.3cm}
\begin{picture}(15,6)(-1,1)
\thicklines
\put(0,0){\line(0,1){3}}
\put(3,0){\line(0,1){3}}
\multiput(1,2)(0.5,0){3}{\circle*{0.15}}
\put(9,0){\line(0,1){3}}
\put(12,0){\line(0,1){3}}
\multiput(10,2)(0.5,0){3}{\circle*{0.15}}
\put(6,0){\oval(4,4)[t]}
\put(7.5,0){\vector(-1,0){3}}
\put(4.5,0){\vector(1,0){3}}
\put(6,-1){$a$}
\end{picture})=\lim_{\epsilon \rightarrow 0}
\setlength{\unitlength}{0.2cm}
\begin{picture}(5,3)(-1,0)
\thicklines
\put(2,0){\oval(4,4)[t]}
\end{picture}
 \times SZ(
\setlength{\unitlength}{0.3cm}
\begin{picture}(15,6)(-1,1)
\thicklines
\put(0,0){\line(0,1){3}}
\put(3,0){\line(0,1){3}}
\multiput(1,2)(0.5,0){3}{\circle*{0.15}}
\put(10,0){\line(0,1){3}}
\put(13,0){\line(0,1){3}}
\multiput(11,2)(0.5,0){3}{\circle*{0.15}}
\put(6,0){\oval(4,6)[tl]}
\put(7,0){\oval(4,6)[tr]}
\put(6,4){\oval(0.6,2)[r]}
\put(7,4){\oval(0.6,2)[l]}
\put(6,7){\oval(2,4)[bl]}
\put(7,7){\oval(2,4)[br]}
\put(8.5,0){\vector(-1,0){4}}
\put(4.5,0){\vector(1,0){4}}
\put(6,-1){$a$}
\put(5.3,4){\vector(1,0){1}}
\put(7.7,4){\vector(-1,0){1}}
\put(8,4){$\epsilon$}
\put(7.8,7){\vector(-1,0){2.6}}
\put(5.2,7){\vector(1,0){2.6}}
\put(6,7.3){$M$}
\end{picture})
\eeq
\end{Zfmultitop}

This led us to using the notation:
\beq
Z_f[\text{\small M}](T)=
\setlength{\unitlength}{0.2cm}
\begin{picture}(5,3)(-1,0)
\thicklines
\put(2,0){\oval(4,4)[t]}
\end{picture}
 \times SZ(T^{\text{\tiny $\;M$-resolved}})
\eeq
with a similar statement for tangles with local minima. $Z_f[\text{\small M}]$ thus defined is multiplicative and invariant under horizontal deformations as defined in\cite{DBN}. For a link $L$ with $2n$ local extrema, we have:
\beq
Z_f[\text{\small M}](L)=
\setlength{\unitlength}{0.2cm}
\begin{picture}(5,3)(-1,0)
\thicklines
\put(2,0){\oval(4,4)[t]}
\end{picture}^n \times SZ(L^{\text{\tiny $\;M$-resolved}}) \times
\setlength{\unitlength}{0.2cm}
\begin{picture}(5,3)(-1,0)
\thicklines
\put(2,2){\oval(4,4)[b]}
\end{picture}^n
\eeq
Further, we have the important reality check:
\begin{pZfeqZ} (\cite{RG})
For $M>0$, $p$ the projection to the algebra of chord diagrams without isolated chords,
$pZ_f[\text{\small M}]=Z$.
\end{pZfeqZ}
\begin{ZfindofM} (\cite{RG})
For a link $L$, $Z_f[\text{\small M}](L)$ and $\hat{Z}_f[\text{\small M}](L)$ are independent of $M>0$ and are denoted $Z_f(L)$ and $\hat{Z}_f(L)$ respectively.
\end{ZfindofM}
Let $L$ be a $q$-components framed oriented link in the blackboard framing represented by a link diagram $\mathcal{D}$. Recall that we consider Morse knots and links; we consider $\mathbb{R}^3$ as $\mathbb{C}\times \mathbb{R}$ and we can arrange that our knots live in $\mathbb{C}\times I$. Let $m_i$ be the number of maximal points of the $i$-th component of $\mathcal{D}$ with respect to $t$. Then we define as in ~\cite{LM5}:
\beq
\hat{Z}_f(L)=Z_f(\mathcal{D})\cdot(\nu^{m_1}\otimes \cdots \otimes \nu^{m_q}) \in \overline{\hatA}(\coprod^q S^1) \label{cumb}
\eeq
where $\nu=Z_f(U)^{-1}$ and:
\beq
U=
\setlength{\unitlength}{0.3cm}
\begin{picture}(5,4)(-1,1)
\thicklines
\put(1,3){\oval(2,2)[t]}
\put(3,3){\oval(2,2)[b]}
\put(5,3){\oval(2,2)[t]}
\put(3,3){\oval(6,5)[b]}
\end{picture}
\eeq
\\
Further, since we regard $\hat{\mathcal{A}}(\amalg^q S^1)$ as an $\otimes^q \hat{\mathcal{A}}$-module, each $\nu^{m_i}$ acts only on the $i$-th component, and it does so by connected sum ~\cite{LM2}. Strictly speaking, we should write:
\beq
\nu^{m_1} \otimes \cdots \otimes \nu^{m_q} \centerdot = (\nu^{m_1} \centerdot)\otimes \cdots \otimes (\nu^{m_q} \centerdot)
\eeq

It is more economical to define $\hat{Z}_f$ as being $Z_f$ with the provision that in the expression for $Z_f(L)$, $\nu$ acts on each local max of each link component. We have that $\hat{Z}_f$ is multiplicative and is an isotopy invariant (\cite{RG}).\\

Though as defined $\hat{Z}_f$ is already an isotopy invariant, we will use a modified version of $\hat{Z}_f$ that is exceptionally well-behaved under handle slide; we will define the modified $\widetilde{Z}_f$ to be $\hat{Z}_f$ with the provision that in the expression for $\hat{Z}_f(L)$, each link component is multiplied by $\nu^{-1}$. It is worth pointing out that our definition for $\hat{Z}_f$ does not yield the same results as Le and Murakami. Indeed, $\hat{Z}_f$ is a normalized version of $Z_f$ which we have defined in such a manner that it enables us to exactly recover the Kontsevich integral. In \cite{LM1} Le and Murakami define $Z_f$ as a framed version of a truncation of the original Kontsevich integral.\\

The doubling map $\Delta$ on strands defined by:
\beq
\Delta:
\setlength{\unitlength}{0.3cm}
\begin{picture}(2,4)(-1,2)
\thicklines
\put(0,0){\vector(0,1){4}}
\put(0,0){\line(0,1){5}}
\end{picture} \mapsto
\setlength{\unitlength}{0.3cm}
\begin{picture}(4,4)(-1,2)
\thicklines
\put(0,0){\vector(0,1){4}}
\put(0,0){\line(0,1){5}}
\put(1,0){\vector(0,1){4}}
\put(1,0){\line(0,1){5}}
\end{picture}
\eeq
\\
induces a map (\cite{LM4}) $\Delta : \mathcal{A}(I) \rightarrow \mathcal{A}(I^2)$ on chord diagrams that is defined as follows on one chord:
\beq
\Delta:
\setlength{\unitlength}{0.3cm}
\begin{picture}(5,4)(-1,2)
\thicklines
\put(0,0){\vector(0,1){4}}
\put(0,0){\line(0,1){5}}
\multiput(0,2)(0.3,0){10}{\line(1,0){0.15}}
\end{picture} \mapsto
\setlength{\unitlength}{0.3cm}
\begin{picture}(5,4)(-1,2)
\thicklines
\put(0,0){\vector(0,1){4}}
\put(0,0){\line(0,1){5}}
\multiput(0,2)(0.3,0){10}{\line(1,0){0.15}}
\put(1,0){\vector(0,1){4}}
\put(1,0){\line(0,1){5}}
\end{picture} +
\setlength{\unitlength}{0.3cm}
\begin{picture}(4,4)(-1,2)
\thicklines
\put(0,0){\vector(0,1){4}}
\put(0,0){\line(0,1){5}}
\put(1,0){\vector(0,1){4}}
\put(1,0){\line(0,1){5}}
\multiput(1,2)(0.3,0){10}{\line(1,0){0.15}}
\end{picture}
\label{dbl}
\eeq
\\
and for chord diagrams of degree greater than 1, we impose that the following square be commutative and use induction:
\beq
\setlength{\unitlength}{0.3cm}
\begin{picture}(21,11)
\thicklines
\put(-4,8){$
\Big(
\setlength{\unitlength}{0.2cm}
\begin{picture}(5,4)(-1,2)
\thicklines
\put(0,0){\vector(0,1){4}}
\put(0,0){\line(0,1){5}}
\multiput(0,2)(0.3,0){10}{\line(1,0){0.15}}
\end{picture},
\setlength{\unitlength}{0.2cm}
\begin{picture}(5,4)(-1,2)
\thicklines
\put(0,0){\vector(0,1){4}}
\put(0,0){\line(0,1){5}}
\multiput(0,2)(0.3,0){10}{\line(1,0){0.15}}
\end{picture}
\Big)
$}
\put(6,8){\vector(1,0){4}}
\put(6,9){$\Delta \times \Delta$}
\put(11,8){$
\Big( \Delta
\setlength{\unitlength}{0.2cm}
\begin{picture}(5,4)(-1,2)
\thicklines
\put(0,0){\vector(0,1){4}}
\put(0,0){\line(0,1){5}}
\multiput(0,2)(0.3,0){10}{\line(1,0){0.15}}
\end{picture}, \Delta
\setlength{\unitlength}{0.2cm}
\begin{picture}(5,4)(-1,2)
\thicklines
\put(0,0){\vector(0,1){4}}
\put(0,0){\line(0,1){5}}
\multiput(0,2)(0.3,0){10}{\line(1,0){0.15}}
\end{picture}
\Big)
$}
\put(17,6){\vector(-1,-1){3}}
\put(16.5,4){$\times$}
\put(1,6){\vector(1,-2){1.6}}
\put(0,4){$\times$}
\put(2,0){$
\setlength{\unitlength}{0.2cm}
\begin{picture}(5,4)(-1,2)
\thicklines
\put(0,0){\vector(0,1){4}}
\put(0,0){\line(0,1){5}}
\multiput(0,1)(0.3,0){10}{\line(1,0){0.15}}
\multiput(0,2)(0.3,0){10}{\line(1,0){0.15}}
\end{picture}
$}
\put(6,0){\vector(1,0){4}}
\put(7,1){$\Delta$}
\put(11,0){$
\Delta
\setlength{\unitlength}{0.2cm}
\begin{picture}(5,4)(-1,2)
\thicklines
\put(0,0){\vector(0,1){4}}
\put(0,0){\line(0,1){5}}
\multiput(0,1)(0.3,0){10}{\line(1,0){0.15}}
\multiput(0,2)(0.3,0){10}{\line(1,0){0.15}}
\end{picture}
$}
\end{picture}
\eeq
\\
Finally, we have the following behavior of $\widetilde{Z}_f$ under band sum moves:
\begin{Goodbehavior}(\cite{RG})  \label{Goodbehavior}
Let $L$ be a framed oriented link. Suppose $K_1$ and $K_2$ are two link components of $L$, and $K_1$ is band summed over $K_2$, which we pictorially represent as:\\
\setlength{\unitlength}{0.5cm}
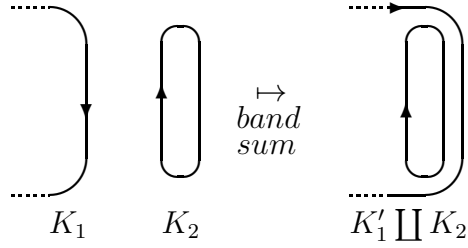
\begin{picture}(14,9)(-8,0)
\thicklines
\multiput(0,1)(0.2,0){5}{\line(1,0){0.08}}
\multiput(0,6)(0.2,0){5}{\line(1,0){0.08}}
\multiput(9,1)(0.2,0){5}{\line(1,0){0.08}}
\multiput(9,6)(0.2,0){5}{\line(1,0){0.08}}
\put(1,3.5){\oval(2,5)[r]}
\put(2,4){\vector(0,-1){1}}
\put(1,0){$K_1$}
\put(4.5,3.5){\oval(1,4)}
\put(4,3){\vector(0,1){1}}
\put(4,0){$K_2$}
\put(6.5,3.5){$\mapsto$}
\put(6,2.8){$\tiny{band}$}
\put(6,2.1){$\tiny{sum}$}
\multiput(10,1)(0,5){2}{\line(1,0){1}}
\put(12,2){\line(0,1){3}}
\put(11,2){\oval(2,2)[br]}
\put(11,5){\oval(2,2)[tr]}
\put(9,0){$K'_1 \coprod K_2$}
\put(10,6){\vector(1,0){0.5}}
\put(11,3.5){\oval(1,4)}
\put(10.5,3){\vector(0,1){0.5}}
\end{picture}\\ \\
where $K'_1$ is the result of doing a band sum move of $K_1$ over $K_2$, and we denote by $L'$ the link obtained from $L$ after such an operation. In the above picture, we have only displayed $K_1$ and $K_2$, and not other components that may be linked to either or both components. If we write:
\beq
\widetilde{Z}_f(L)=\sum_{\substack{chord \\ diagrams \;X}}c_X X
\eeq
for coefficients $c_X$, then we have:
\beq
\widetilde{Z}_f(L')=\sum_{\substack{chord \\ diagrams \;X}}c_X X' \label{hsformula}
\eeq
where $X$ and its corresponding chord diagram $X'$ after the band sum move are given below:\\
\setlength{\unitlength}{0.5cm}
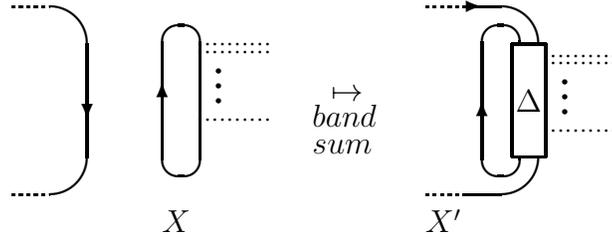
\begin{picture}(14,9)(-6,0)
\thicklines
\multiput(0,1)(0.2,0){5}{\line(1,0){0.08}}
\multiput(0,6)(0.2,0){5}{\line(1,0){0.08}}
\multiput(11,1)(0.2,0){5}{\line(1,0){0.08}}
\multiput(11,6)(0.2,0){5}{\line(1,0){0.08}}
\put(1,3.5){\oval(2,5)[r]}
\put(2,4){\vector(0,-1){1}}
\put(4.5,3.5){\oval(1,4)}
\multiput(5,5)(0.2,0){10}{\circle*{0.1}}
\multiput(5,4.8)(0.2,0){10}{\circle*{0.1}}
\multiput(5,3)(0.2,0){10}{\circle*{0.1}}
\multiput(5.5,3.5)(0,0.4){3}{\circle*{0.15}}
\put(4,3){\vector(0,1){1}}
\put(4,0){$X$}
\put(8.5,3.5){$\mapsto$}
\put(8,2.8){$\tiny{band}$}
\put(8,2.1){$\tiny{sum}$}
\multiput(12,1)(0,5){2}{\line(1,0){1}}
\put(14.2,2){\line(0,1){3}}
\put(13.3,2){\line(0,1){3}}
\put(13,2){\oval(2,2)[br]}
\put(13,5){\oval(2,2)[tr]}
\put(11,0){$X'$}
\put(12,6){\vector(1,0){0.5}}
\put(13,3.5){\oval(1,4)[l]}
\put(13,5){\oval(1,1)[tr]}
\put(13,2){\oval(1,1)[br]}
\multiput(13.3,2)(0,3){2}{\line(1,0){0.9}}
\put(12.5,3){\vector(0,1){0.5}}
\multiput(14.2,4.7)(0.2,0){10}{\circle*{0.1}}
\multiput(14.2,4.5)(0.2,0){10}{\circle*{0.1}}
\multiput(14.2,2.7)(0.2,0){10}{\circle*{0.1}}
\multiput(14.7,3.2)(0,0.4){3}{\circle*{0.15}}
\put(13.4,3.2){$\Delta$}
\end{picture}\\ \\
To be specific, the map $\Delta$ doubles strands, and the chords on those strands as well. Since we operate a band sum move here, the $\Delta$ enclosed in the box means by abuse of notation (\cite{LM2}, \cite{LM5}, \cite{LM6}) that the doubling of strands coming with the band sum move proper has been performed and the only thing left to be done is to double chords accordingly.
\end{Goodbehavior}

\section{The book notation}
\subsection{Linking matrix and degree 1 Kontsevich integral}
Yetter ~\cite{Y2} observed a long time ago that the degree 1 part of the Kontsevich integral behaves like the linking matrix under band sum move. He sees that as a motivation for introducing the book notation that we will cover in the next subsection. For now we show his claim in the simple case of a two components link $L=K_i \cup K_j$ where the two knots $K_i$ and $K_j$ are trivial and unlinked for the simplicity of exposition. We take the following basis for $\mathcal{A}(S^1 \amalg S^1)$:
\begin{align}
&\frac{1}{2}
\setlength{\unitlength}{0.4cm}
\begin{picture}(7,3)(-1,0.8)
\put(1,1){\oval(2,2)}
\put(4,1){\oval(2,2)}
\multiput(0,1)(0.4,0){5}{\line(1,0){0.2}}
\end{picture} \\
&\frac{1}{2}
\setlength{\unitlength}{0.4cm}
\begin{picture}(7,3)(-1,0.8)
\put(1,1){\oval(2,2)}
\put(4,1){\oval(2,2)}
\multiput(3,1)(0.4,0){5}{\line(1,0){0.2}}
\end{picture} \\
&\setlength{\unitlength}{0.4cm}
\begin{picture}(7,3)(-1,0.8)
\put(1,1){\oval(2,2)}
\put(4,1){\oval(2,2)}
\multiput(2,1)(0.3,0){4}{\line(1,0){0.15}}
\end{picture}
\end{align}
We write:
\beq
Z_1(L)=c_i \Bigg(
\frac{1}{2}
\setlength{\unitlength}{0.4cm}
\begin{picture}(7,3)(-1,0.8)
\put(1,1){\oval(2,2)}
\put(4,1){\oval(2,2)}
\multiput(0,1)(0.4,0){5}{\line(1,0){0.2}}
\end{picture}
\Bigg) + c_j \Bigg(
\frac{1}{2}
\setlength{\unitlength}{0.4cm}
\begin{picture}(7,3)(-1,0.8)
\put(1,1){\oval(2,2)}
\put(4,1){\oval(2,2)}
\multiput(3,1)(0.4,0){5}{\line(1,0){0.2}}
\end{picture}
\Bigg) + c_{ij} \Bigg(
\setlength{\unitlength}{0.4cm}
\begin{picture}(7,3)(-1,0.8)
\put(1,1){\oval(2,2)}
\put(4,1){\oval(2,2)}
\multiput(2,1)(0.3,0){4}{\line(1,0){0.15}}
\end{picture}
\Bigg)
\eeq
After the band sum move of $K_i$ over $K_j$, the resulting link is denoted by $L'$ and we have:
\begin{align}
Z_1(L')&=c_i \Bigg(
\frac{1}{2}
\setlength{\unitlength}{0.4cm}
\begin{picture}(7,3)(-1,0.8)
\put(1,1){\oval(2,2)}
\put(4,1){\oval(2,2)}
\multiput(0,1)(0.4,0){5}{\line(1,0){0.2}}
\end{picture}
\Bigg)' + c_j \Bigg(
\frac{1}{2}
\setlength{\unitlength}{0.4cm}
\begin{picture}(7,3)(-1,0.8)
\put(1,1){\oval(2,2)}
\put(4,1){\oval(2,2)}
\multiput(3,1)(0.4,0){5}{\line(1,0){0.2}}
\end{picture}
\Bigg)' + c_{ij} \Bigg(
\setlength{\unitlength}{0.4cm}
\begin{picture}(7,3)(-1,0.8)
\put(1,1){\oval(2,2)}
\put(4,1){\oval(2,2)}
\multiput(2,1)(0.3,0){4}{\line(1,0){0.15}}
\end{picture}
\Bigg)'\\
&=c_i \Bigg(
\frac{1}{2}
\setlength{\unitlength}{0.4cm}
\begin{picture}(7,3)(-1,0.8)
\put(1,1){\oval(2,2)}
\put(4,1){\oval(2,2)}
\multiput(0,1)(0.4,0){5}{\line(1,0){0.2}}
\end{picture}
\Bigg)\nonumber \\
& + c_j \Bigg(
\frac{1}{2}
\setlength{\unitlength}{0.4cm}
\begin{picture}(7,3)(-1,0.8)
\put(1,1){\oval(2,2)}
\put(4,1){\oval(2,2)}
\multiput(0,1)(0.4,0){5}{\line(1,0){0.2}}
\end{picture}+
\frac{1}{2}
\setlength{\unitlength}{0.4cm}
\begin{picture}(7,3)(-1,0.8)
\put(1,1){\oval(2,2)}
\put(4,1){\oval(2,2)}
\multiput(3,1)(0.4,0){5}{\line(1,0){0.2}}
\end{picture} \pm
\setlength{\unitlength}{0.4cm}
\begin{picture}(7,3)(-1,0.8)
\put(1,1){\oval(2,2)}
\put(4,1){\oval(2,2)}
\multiput(2,1)(0.3,0){4}{\line(1,0){0.15}}
\end{picture}
\Bigg)\nonumber\\
& + c_{ij} \Bigg(\pm
\setlength{\unitlength}{0.4cm}
\begin{picture}(7,3)(-1,0.8)
\put(1,1){\oval(2,2)}
\put(4,1){\oval(2,2)}
\multiput(0,1)(0.4,0){5}{\line(1,0){0.2}}
\end{picture}+
\setlength{\unitlength}{0.4cm}
\begin{picture}(7,3)(-1,0.8)
\put(1,1){\oval(2,2)}
\put(4,1){\oval(2,2)}
\multiput(2,1)(0.3,0){4}{\line(1,0){0.15}}
\end{picture}
\Bigg)\\
&=[c_i+c_j \pm 2c_{ij}] \Bigg(
\frac{1}{2}
\setlength{\unitlength}{0.4cm}
\begin{picture}(7,3)(-1,0.8)
\put(1,1){\oval(2,2)}
\put(4,1){\oval(2,2)}
\multiput(0,1)(0.4,0){5}{\line(1,0){0.2}}
\end{picture}
\Bigg) + c_j \Bigg(
\frac{1}{2}
\setlength{\unitlength}{0.4cm}
\begin{picture}(7,3)(-1,0.8)
\put(1,1){\oval(2,2)}
\put(4,1){\oval(2,2)}
\multiput(3,1)(0.4,0){5}{\line(1,0){0.2}}
\end{picture}
\Bigg) \nonumber\\
&+ [\pm c_j +c_{ij}] \Bigg(
\setlength{\unitlength}{0.4cm}
\begin{picture}(7,3)(-1,0.8)
\put(1,1){\oval(2,2)}
\put(4,1){\oval(2,2)}
\multiput(2,1)(0.3,0){4}{\line(1,0){0.15}}
\end{picture}
\Bigg)
\end{align}
Now observe that in the basis for $\mathcal{A}(S^1 \amalg S^1)$ we chose, the coefficients of $Z_1$ transform like the linking matrix entries under band sum move:
\begin{align}
L_i &\mapsto L_i '=L_i + L_j \pm 2 L_{ij} \\
L_j & \mapsto L_j'=L_j \\
L_{ij} & \mapsto L_{ij}'=L_{ij} \pm L_j
\end{align}
The matrix congruence that implements the band sum move on both the coefficients of $Z_1(L)$ and the linking matrix is given by:
\beq
M=\;
\setlength{\unitlength}{0.3cm}
\begin{picture}(8,4)(0,1)
\thicklines
\put(1,2){\oval(1,4)[l]}
\put(5,2){\oval(1,4)[r]}
\put(2,3){\text{\small $1$}}
\put(1,1){\text{\small $\pm 1$}}
\put(4,1){\text{\small $1$}}
\put(4,3){\text{\small $0$}}
\end{picture}
\eeq
We check that if:
\beq
A=\;
\setlength{\unitlength}{0.3cm}
\begin{picture}(8,4)(0,1)
\thicklines
\put(1,2){\oval(1,4)[l]}
\put(6,2){\oval(1,4)[r]}
\put(1,3){\text{\small $a_{ii}$}}
\put(1,1){\text{\small $a_{ji}$}}
\put(4,1){\text{\small $a_{jj}$}}
\put(4,3){\text{\small $a_{ij}$}}
\end{picture}
\eeq
stands for either of
\beq
\setlength{\unitlength}{0.3cm}
\begin{picture}(8,4)(0,1)
\thicklines
\put(1,2){\oval(1,4)[l]}
\put(6,2){\oval(1,4)[r]}
\put(1,3){\text{\small $L_i$}}
\put(1,1){\text{\small $L_{ji}$}}
\put(4,1){\text{\small $L_j$}}
\put(4,3){\text{\small $L_{ij}$}}
\end{picture} \text{or} \;\;\setlength{\unitlength}{0.3cm}
\begin{picture}(8,4)(0,1)
\thicklines
\put(1,2){\oval(1,4)[l]}
\put(6,2){\oval(1,4)[r]}
\put(1,3){\text{\small $c_i$}}
\put(1,1){\text{\small $c_{ji}$}}
\put(4,1){\text{\small $c_j$}}
\put(4,3){\text{\small $c_{ij}$}}
\end{picture}
\eeq
with $L_{ij}=L_{ji}$ and $c_{ij}=c_{ji}$, then:
\begin{align}
M^T AM &=\setlength{\unitlength}{0.3cm}
\begin{picture}(7,4)(0,1)
\thicklines
\put(1,2){\oval(1,4)[l]}
\put(6,2){\oval(1,4)[r]}
\put(2,3){\text{\small $1$}}
\put(2,1){\text{\small $0$}}
\put(4,1){\text{\small $1$}}
\put(4,3){\text{\small $\pm 1$}}
\end{picture}
\setlength{\unitlength}{0.3cm}
\begin{picture}(7,4)(0,1)
\thicklines
\put(1,2){\oval(1,4)[l]}
\put(6,2){\oval(1,4)[r]}
\put(1,3){\text{\small $a_{ii}$}}
\put(1,1){\text{\small $a_{ji}$}}
\put(4,1){\text{\small $a_{jj}$}}
\put(4,3){\text{\small $a_{ij}$}}
\end{picture}
\setlength{\unitlength}{0.3cm}
\begin{picture}(7,4)(0,1)
\thicklines
\put(1,2){\oval(1,4)[l]}
\put(6,2){\oval(1,4)[r]}
\put(1,3){\text{\small $1$}}
\put(1,1){\text{\small $\pm 1$}}
\put(4,1){\text{\small $1$}}
\put(4,3){\text{\small $0$}}
\end{picture} \nonumber \\
&=
\setlength{\unitlength}{0.3cm}
\begin{picture}(19,5)(0,1)
\thicklines
\put(1,2){\oval(1,4)[l]}
\put(16,2){\oval(1,4)[r]}
\put(1,3){\text{\small $a_{ii}+a_{jj} \pm 2a_{ij}$}}
\put(1,1){\text{\small $a_{ji} \pm a_{jj}$}}
\put(11,1){\text{\small $a_{jj}$}}
\put(11,3){\text{\small $a_{ij} \pm a_{jj}$}}
\end{picture}
\end{align}
In the following, we seek to generalize such a transformation.

\subsection{Books of matrices}
Note that in the previous section, as is customary whenever we compute the Kontsevich integral of tangles, the diagrams are tangle chord diagrams and not chord diagrams as elements of $\hatA(\coprod S^1)$ as initially defined by Kontsevich. Indeed, for a link $L$ sliced into $n$ horizontal strips in each of which we have a tangle $T_i$, $1 \leq i \leq n$ such that $L=T_1 \times \cdots \times T_n$, it is generally understood that when we compute $\hat{Z}_f(L)$ by using the multiplicativity property of the Kontsevich integral $\hat{Z}_f$:
\beq
\hat{Z}_f(L)=\hat{Z}_f(T_1) \times \cdots \times \hat{Z}_f(T_n)
\eeq
the resulting object $\hat{Z}_f(L)$ is a sum of tangle chord diagrams with coefficients in front of each diagram being obtained from the Kontsevich integral itself. By definition, $\widetilde{Z}_f$ is also multiplicative and $\widetilde{Z}_f(L)$ is also a sum of tangle chord diagrams with complex coefficients. The Kontsevich integral $\widetilde{Z}_f$ of a link $L$ can be written:
\beq
\widetilde{Z}_f(L)=\sum_{\text{chord diagr. X}} c_X X = \sum_{m \geq 0} \sum_{|X|=m} c_X X
\eeq
where by $|X|=m$ we mean that the tangle chord diagram $X$ has chord degree $m$, and we sum over all such tangle chord diagrams, for all $m \geq 0$. In what follows, we fix $m \geq 1$. We will be working with a $q$-components link $L=\coprod_{1 \leq l \leq q} K_l$.\\

\subsubsection{Vertical slicing of tangles}
Before slicing links, we have to fully determine where local extrema will be located on any given link. The definitions of $\hat{Z}_f$ and $\widetilde{Z}_f$ each introduce factors of $\nu =Z_f(U)^{-1}$ which will yield additional local extrema on the link $L$ upon being multiplied with $Z_f(L)$. Since those factors of $\nu$ are multiplying local maxima, it suffices to consider products of the form:
\beq
\nu \cdot Z_f[\text{\small Q}](
\setlength{\unitlength}{0.3cm}
\begin{picture}(6,3)(-1,0)
\thicklines
\put(2,0){\oval(4,4)[t]}
\put(3.5,0){\vector(-1,0){3}}
\put(0.5,0){\vector(1,0){3}}
\put(2,-1){$a$}
\end{picture})\nonumber
\eeq
for $Q>0$. Observe that if we write:
\beq
Z_f(
\setlength{\unitlength}{0.2cm}
\begin{picture}(11,4)(-1,2)
\thicklines
\put(1.5,3){\oval(1,2)[b]}
\put(4.5,3){\oval(5,2)[t]}
\put(4.5,3){\oval(7,4)[t]}
\put(7.5,3){\oval(1,2)[b]}
\end{picture})=\sum_{|P| \geq 0} c_P \cdot \Bigg(
\setlength{\unitlength}{0.2cm}
\begin{picture}(11,4)(-1,2)
\thicklines
\put(1.5,3){\oval(1,2)[b]}
\put(4.5,3){\oval(5,2)[t]}
\put(4.5,3){\oval(7,4)[t]}
\put(7.5,3){\oval(1,2)[b]}
\end{picture}\Bigg) _P
\eeq
then:
\beq
\nu=Z_f(
\setlength{\unitlength}{0.2cm}
\begin{picture}(11,4)(-1,2)
\thicklines
\put(1.5,3){\oval(1,2)[b]}
\put(4.5,3){\oval(5,2)[t]}
\put(4.5,3){\oval(7,4)[t]}
\put(7.5,3){\oval(1,2)[b]}
\end{picture})^{-1}=
\sum_{|P| \geq 0} d_P \cdot \Bigg(
\setlength{\unitlength}{0.2cm}
\begin{picture}(11,4)(-1,2)
\thicklines
\put(1.5,3){\oval(1,2)[b]}
\put(4.5,3){\oval(5,2)[t]}
\put(4.5,3){\oval(7,4)[t]}
\put(7.5,3){\oval(1,2)[b]}
\end{picture}\Bigg) _P
\eeq
with coefficients $d_P$ such that
\beq
\nu \cdot Z_f(
\setlength{\unitlength}{0.2cm}
\begin{picture}(11,4)(-1,2)
\thicklines
\put(1.5,3){\oval(1,2)[b]}
\put(4.5,3){\oval(5,2)[t]}
\put(4.5,3){\oval(7,4)[t]}
\put(7.5,3){\oval(1,2)[b]}
\end{picture})=\setlength{\unitlength}{0.2cm}
\begin{picture}(11,4)(-1,2)
\thicklines
\put(1.5,3){\oval(1,2)[b]}
\put(4.5,3){\oval(5,2)[t]}
\put(4.5,3){\oval(7,4)[t]}
\put(7.5,3){\oval(1,2)[b]}
\end{picture}
\eeq
It follows that:
\begin{align}
\nu \cdot Z_f[\text{\small Q}](
\setlength{\unitlength}{0.3cm}
\begin{picture}(6,3)(-1,0)
\thicklines
\put(2,0){\oval(4,4)[t]}
\put(3.5,0){\vector(-1,0){3}}
\put(0.5,0){\vector(1,0){3}}
\put(2,-1){$a$}
\end{picture})&= \nu \cdot
\setlength{\unitlength}{0.2cm}
\begin{picture}(5,3)(-1,0)
\thicklines \put(2,0){\oval(4,4)[t]}
\end{picture}
 \times SZ(
\setlength{\unitlength}{0.3cm}
\begin{picture}(5,4)(-1,0)
\thicklines
\qbezier(0,0)(1,2)(0,3)
\qbezier(3,0)(2,2)(3,3)
\put(2.5,0){\vector(-1,0){2}}
\put(0.5,0){\vector(1,0){2}}
\put(1,-1){$a$}
\put(2.5,3){\vector(-1,0){2}}
\put(0.5,3){\vector(1,0){2}}
\put(1,3.5){$Q$}
\end{picture})\\
&=\sum_{|P| \geq 0} d_P \cdot \Bigg(
\setlength{\unitlength}{0.2cm}
\begin{picture}(11,4)(-1,2)
\thicklines
\put(1.5,3){\oval(1,2)[b]}
\put(4.5,3){\oval(5,2)[t]}
\put(4.5,3){\oval(7,4)[t]}
\put(7.5,3){\oval(1,2)[b]}
\end{picture}\Bigg) _P \#
\setlength{\unitlength}{0.2cm}
\begin{picture}(5,3)(-1,0)
\thicklines \put(2,0){\oval(4,4)[t]}
\end{picture}
 \times \sum_{|R| \geq 0}e_R \cdot
\Bigg(
\setlength{\unitlength}{0.3cm}
\begin{picture}(5,4)(-1,1)
\thicklines
\qbezier(0,0)(1,2)(0,3)
\qbezier(3,0)(2,2)(3,3)
\put(2.5,0){\vector(-1,0){2}}
\put(0.5,0){\vector(1,0){2}}
\put(1,-1){$a$}
\put(2.5,3){\vector(-1,0){2}}
\put(0.5,3){\vector(1,0){2}}
\put(1,3.5){$Q$}
\end{picture} \Bigg)_R \nonumber \\
&=\sum_{|P|,|R| \geq 0} d_P e_R \cdot \Bigg(
\setlength{\unitlength}{0.3cm}
\begin{picture}(8,7)(-1,1)
\thicklines
\qbezier(1,0)(4,6)(1,4)
\qbezier(1,4)(0,5)(3,7)
\qbezier(3,7)(5,7)(5,5)
\qbezier(5,5)(5,0)(6,0)
\end{picture} \Bigg)_{(P,R)}
\end{align}
\\
This simple computation shows that the skeleton $L$ will be modified in the expression for $\widetilde{Z}_f(L)$ by the introduction of a hook of the above form at each local max but for one local max on each component by definition of $\widetilde{Z}_f$.
\begin{Ltilde}
Let $\widetilde{L}$ be the link $L$ with each local max being tweaked into a left pointing hook as above, but for one local max on each link component.
\end{Ltilde}
We consider the handle slide of the handle corresponding to the $i$-th component $K_i$ over the handle corresponding to the $j$-th component $K_j$. This is done by doing a band sum move of $K_i$ over $K_j$. The result of such a band sum move is denoted by $L'$. However we work with $\widetilde{L}$ and $\widetilde{L'}$ as the rest of the section will make clear. We slice $\widetilde{L}$ into N vertical strips as follows. For a local maximum on the $s$-th component $K_s$ such as:\\
\setlength{\unitlength}{0.4cm}
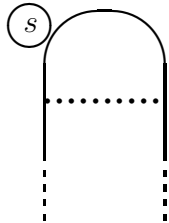
\begin{picture}(7,6)(-13,3)
\thicklines
\put(2,2){\line(0,1){3}}
\put(6,2){\line(0,1){3}}
\multiput(2,0)(0,0.5){4}{\line(0,1){0.2}}
\multiput(6,0)(0,0.5){4}{\line(0,1){0.2}}
\put(4,5){\oval(4,4)[t]}
\multiput(2.1,4)(0.4,0){10}{\circle*{0.2}}
\put(1.5,6.5){\circle{1.5}}
\put(1.3,6.3){$s$}
\end{picture}\\ \\
\\ \\
the slicing is performed as follows:\\
\setlength{\unitlength}{0.4cm}
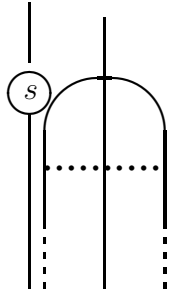
\begin{picture}(7,8)(-13,3)
\put(4,0){\line(0,1){9}}
\put(1.5,0){\line(0,1){5.8}}
\put(1.5,7.5){\line(0,1){2}}
\put(6.5,0){\line(0,1){9}}
\thicklines
\put(2,2){\line(0,1){3}}
\put(6,2){\line(0,1){3}}
\multiput(2,0)(0,0.5){4}{\line(0,1){0.2}}
\multiput(6,0)(0,0.5){4}{\line(0,1){0.2}}
\put(4,5){\oval(4,4)[t]}
\multiput(2.1,4)(0.4,0){10}{\circle*{0.2}}
\put(1.5,6.5){\circle{1.5}}
\put(1.3,6.3){$s$}
\end{picture}\\ \\
\\ \\
We do want each local max to be enclosed within two vertical slices to distinguish neighboring local extrema of a same component. If we call the vertical slices on either side of a local max dividing slices, it follows in practice that consecutive local extrema share a dividing slice is sufficient as we will see later. We do this at each local max of each component $K_s$ of the link $\widetilde{L}$. We slice each local min of each component in like manner, keeping in mind that consecutive local extrema can share dividing slices. We number those vertical strips formed from this slicing procedure starting from the left.\\

In the above situation, we would have the strips with the following labels:\\
\setlength{\unitlength}{0.4cm}
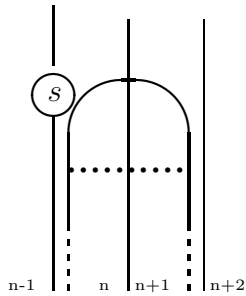
\begin{picture}(7,8)(-13,3)
\put(0,0){$\text{\tiny n-1}$}
\put(3,0){$\text{\tiny n}$}
\put(4.2,0){$\text{\tiny n+1}$}
\put(6.7,0){$\text{\tiny n+2}$}
\put(4,0){\line(0,1){9}}
\put(1.5,0){\line(0,1){5.8}}
\put(1.5,7.5){\line(0,1){2}}
\put(6.5,0){\line(0,1){9}}
\thicklines
\put(2,2){\line(0,1){3}}
\put(6,2){\line(0,1){3}}
\multiput(2,0)(0,0.5){4}{\line(0,1){0.2}}
\multiput(6,0)(0,0.5){4}{\line(0,1){0.2}}
\put(4,5){\oval(4,4)[t]}
\multiput(2.1,4)(0.4,0){10}{\circle*{0.2}}
\put(1.5,6.5){\circle{1.5}}
\put(1.3,6.3){$s$}
\end{picture}\\ \\
\\ \\

We now discuss the labeling of the chords. For each time, we have a chord. Thus it is natural to number the chords from the bottom up. If $t_1 < \cdots t_m$ are the different times corresponding to $m$ different chords, then those corresponding chords will be labeled $1,\cdots,m$ respectively.\\

\subsubsection{Pages as representations of tangle chord diagrams}
Consider the generic situation of one chord stretching between 2 components of $L$ indexed by $s$ and $t$:\\
\setlength{\unitlength}{0.4cm}
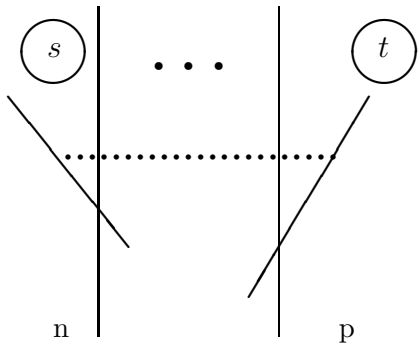
\begin{picture}(15,11)(-10,2)
\put(4,0){\line(0,1){11}}
\put(10,0){\line(0,1){11}}
\thicklines
\put(2.5,0){$\text{\small n}$}
\put(1,8){\line(4,-5){4}}
\put(2.5,9.5){\circle{2}}
\put(2.3,9.3){$s$}
\put(12,0){$\text{\small p}$}
\put(13,8){\line(-3,-5){4}}
\put(13.5,9.5){\circle{2}}
\put(13.3,9.3){$t$}
\multiput(6,9)(1,0){3}{\circle*{0.3}}
\multiput(3,6)(0.4,0){23}{\circle*{0.2}}
\end{picture}\\ \\
\\ \\
where we have displayed only portions of the $s$ and $t$ components on which the $a$-th chord is ending, $1 \leq a \leq m$ and $1 \leq n,p \leq N$ are strip indices.\\

We represent each such chord by a $qN \times qN$ matrix in the basis given by the ordering of the components, and of the strips, as $(1,2,\cdots, N, 1,2, \cdots, N, \cdots, N)$ the first $N$ vectors $1,2,\cdots,N$ corresponding to the first component of $L$, followed by those for the second component, and so on, until the $q$-th component. The $st$ block of that matrix will carry information about chords between the $s$ and $t$ components of the link. In the above situation, the $st$ block has all its entries zero, except the $np$ entry which is one. Note that the matrix is symmetric, so all blocks are empty for that particular chord, except the $ts$-th one, whose $pn$ entry is one as well. We represent such a matrix as follows where without loss of generality we have chosen $s<t$ and $n<p$:
\beq
\setlength{\unitlength}{0.4cm}
\begin{picture}(14,15)
\put(2,10){\line(1,-1){10}}
\put(2,5){\line(1,-1){4}}
\put(7,10){\line(1,-1){4}}
\thicklines
\put(2,5){\oval(2,12)[l]}
\put(13,5){\oval(2,12)[r]}
\put(0,3){$t$}
\put(0,8){$s$}
\put(4,12){$s$}
\put(9,12){$t$}
\put(2,1){\line(0,1){4}}
\put(2,1){\line(1,0){4}}
\put(2,5){\line(1,0){4}}
\put(6,1){\line(0,1){4}}
\put(3,1.8){$1$}
\put(1.5,1.8){$\text{\tiny p}$}
\put(3,5.3){$\text{\tiny n}$}
\put(7,6){\line(1,0){4}}
\put(7,6){\line(0,1){4}}
\put(7,10){\line(1,0){4}}
\put(11,6){\line(0,1){4}}
\put(10,9){$1$}
\put(6.3,9){$\text{\tiny n}$}
\put(10,10.3){$\text{\tiny p}$}
\end{picture}
\eeq
\\
\\
\\
We will refer to such a matrix as a page, and we denote it by $A_{s,t,n,p}$ with obvious notations. We do this for all chords of a given tangle chord diagram $X$ of degree $m$. The information about its chords will therefore be given by $m$ ordered pages from the bottom up, the collection of which will be referred to as a book. We denote a book as follows:
\beq
A_{I,J,U,V}:=\stck_{1 \leq a \leq m}A_{i_a,j_a,u_a,v_a}
\eeq
where $I$, $J$, $U$ and $V$ are multi-indices defined by:
\begin{align}
I&=(i_1,\cdots ,i_m) \\
J&=(j_1,\cdots ,j_m) \\
U&=(u_1,\cdots ,u_m) \\
V&=(v_1,\cdots ,v_m)
\end{align}
with $1 \leq i_l,j_l \leq q$ are component indices, $1 \leq u_l,v_l \leq N$ are strip indices for $1 \leq l \leq m$. We denote the size of such multi-indices by $|I|=|J|=|U|=|V|=m$. In the above example, for the $a$-th chord, we have $i_a=s$, $j_a=t$, $u_a=n$ and $v_a=p$. In the product of pages defining a book, pages are ordered from the bottom up. Now instead of using the notation $X$ for a tangle chord diagram, we use the book notation $A_{I,J,U,V}$ which incorporates the information about the chords on the tangle. We have:
\beq
Z_f(L)=\sum_{\substack{\text{chord}\\ \text{diagr. X}}} a_X X =\sum_{m \geq 0} \sum_{\substack{|I|=|J|=m \\ |U|=|V|=m}} a_{I,J,U,V} A_{I,J,U,V}(L)
\eeq
However $\widetilde{Z}_f(L)$ uses powers of
\beq
\nu=Z_f(U)^{-1}=Z_f(
\setlength{\unitlength}{0.2cm}
\begin{picture}(11,4)(-1,2)
\thicklines
\put(1.5,3){\oval(1,2)[b]}
\put(4.5,3){\oval(5,2)[t]}
\put(4.5,3){\oval(7,4)[t]}
\put(7.5,3){\oval(1,2)[b]}
\end{picture})
\eeq
and with the above slicing performed on $\widetilde{L}$ we are able to use books $A_{IJUV}(\widetilde{L})=A_{IJUV}$. We can therefore write:
\beq
\widetilde{Z}_f(L)=\sum_{\substack{\text{chord}\\ \text{diagr. X}}} c_X X =\sum_{m \geq 0} \sum_{\substack{|I|=|J|=m \\ |U|=|V|=m}} c_{I,J,U,V} A_{I,J,U,V}
\eeq
where we have $c_{I,J,U,V}:=c_{A_{IJUV}}$.\\

Once we have this picture of sum of books describing tangle chord diagrams, we can add another constraint on each book; we insist that whenever we have a given tangle chord diagram, each of its chords is moved up until it reaches a local max, and it is the book of such a tangle chord diagram we exhibit in $\widetilde{Z}_f(L)$.

\section{Behavior of the Kontsevich integral $\widetilde{Z}_f$ under handle slide using the book notation}
What we have in ~\eqref{hsformula} is the following: if $h_{ij}$ denotes the band sum move map on links corresponding to the band sum move of the $i$-th component of $L$ over its $j$-th component, then we can write $L'=h_{ij}(L)$, so that $\widetilde{Z}_f(L')=\widetilde{Z}_f(h_{ij}L)$. What we would like however is to find a map $\mathbf{h}_{ij}$ induced from $h_{ij}$ that acts on the Kontsevich integral $\widetilde{Z}_f(L)$ of links $L$ to yield their corresponding values after handle slide. We write this as $\mathbf{h}_{ij}(\widetilde{Z}_f)$. We claim that there is such a map, and that moreover for any link $L$, we have $\mathbf{h}_{ij}(\widetilde{Z}_f(L))=\widetilde{Z}_f(h_{ij}L)$. In other terms, the following diagram is commutative:\\
\setlength{\unitlength}{0.4cm}
\begin{picture}(14,10)(-8,2)
\thicklines
\put(5,8){\vector(1,0){6}}
\put(3,7){\vector(0,-1){4}}
\put(4.5,1){\vector(1,0){3.5}}
\put(12,7){\vector(0,-1){4}}
\put(3,8){$L$}
\put(0,1){$\widetilde{Z}_f(L)$}
\put(9,1){$\mathbf{h}_{ij}\widetilde{Z}_f(L)=\widetilde{Z}_f(h_{ij}L)$}
\put(12,8){$h_{ij}L$}
\put(7,9){$h_{ij}$}
\put(0,5){$\widetilde{Z}_f$}
\put(6,-0.5){$\mathbf{h_{ij}}$}
\put(13,5){$\widetilde{Z}_f$}
\end{picture}\\ \\
\\ \\
It is important to remember that using the isotopy invariance of $\widetilde{Z}_f$, we can arrange that the band sum does not introduce new local extrema, and is far away from the rest of the link that we may be able to use the Long Chords Lemma (\cite{ChDu}).\\

We have the following fact about books; since no two chords can be positioned at the same height $t$ on a tangle, pages, which represent chords, can allow for the possibility to hold many other chords by virtue of the non-simultaneity of chords. In this manner there is no ambiguity as to what chord in a page is represented by which entry. Further if a page holds information about more than one chord, we can split the matrix in as many matrices as there are chords represented in the original page. We illustrate this situation presently:\\
\setlength{\unitlength}{0.4cm}
\begin{picture}(15,16)(-10,2)
\put(2,11){\line(1,-1){10}}
\thicklines
\put(2,6){\oval(2,12)[l]}
\put(12,6){\oval(2,12)[r]}
\put(4,5){$1$}
\put(6,3){$1$}
\put(8,9){$1$}
\put(10,7){$1$}
\put(-1,3){$l$}
\put(0,3){$\text{\tiny v}$}
\put(-1,5){$k$}
\put(0,5){$\text{\tiny u}$}
\put(-1,7){$h$}
\put(0,7){$\text{\tiny p}$}
\put(-1,9){$g$}
\put(0,9){$\text{\tiny n}$}
\put(4,14){$g$}
\put(4,13){$\text{\tiny n}$}
\put(6,14){$h$}
\put(6,13){$\text{\tiny p}$}
\put(8,14){$k$}
\put(8,13){$\text{\tiny u}$}
\put(10,14){$l$}
\put(10,13){$\text{\tiny v}$}
\end{picture}\\ \\
\\ \\
This represents a page carrying information about two chords, one between the $g$-th and $k$-th components, the other between the $h$-th and $l$-th components. For the first chord, the foot on $K_g$ is in the $n$-th strip, and the foot on $K_k$ is in the $u$-th strip. For the second chord, the foot on $K_h$ is in the $p$-th strip, and the foot on $K_l$ is in the $v$-th strip. We have indicated to the left and above the matrix the block indices $g$, $h$, $k$ and $l$, and in small letters the strips within the blocks where a foot is ending, and those are $n$, $p$, $u$ and $v$. Such a matrix splits as follows:
\beq
\setlength{\unitlength}{0.2cm}
\begin{picture}(15,12)(0,5)
\put(2,11){\line(1,-1){10}}
\thicklines
\put(2,6){\oval(2,12)[l]}
\put(12,6){\oval(2,12)[r]}
\put(4,5){$1$}
\put(8,9){$1$}
\put(-2,3){$l$}
\put(0,3){$\text{\tiny v}$}
\put(-2,5){$k$}
\put(0,5){$\text{\tiny u}$}
\put(-2,7){$h$}
\put(0,7){$\text{\tiny p}$}
\put(-2,9){$g$}
\put(0,9){$\text{\tiny n}$}
\put(4,14){$g$}
\put(4,13){$\text{\tiny n}$}
\put(6,14){$h$}
\put(6,13){$\text{\tiny p}$}
\put(8,14){$k$}
\put(8,13){$\text{\tiny u}$}
\put(10,14){$l$}
\put(10,13){$\text{\tiny v}$}
\end{picture}\\ \\ + \setlength{\unitlength}{0.2cm}
\begin{picture}(15,12)(-4,5)
\put(2,11){\line(1,-1){10}}
\thicklines
\put(2,6){\oval(2,12)[l]}
\put(12,6){\oval(2,12)[r]}
\put(6,3){$1$}
\put(10,7){$1$}
\put(-2,3){$l$}
\put(0,3){$\text{\tiny v}$}
\put(-2,5){$k$}
\put(0,5){$\text{\tiny u}$}
\put(-2,7){$h$}
\put(0,7){$\text{\tiny p}$}
\put(-2,9){$g$}
\put(0,9){$\text{\tiny n}$}
\put(4,14){$g$}
\put(4,13){$\text{\tiny n}$}
\put(6,14){$h$}
\put(6,13){$\text{\tiny p}$}
\put(8,14){$k$}
\put(8,13){$\text{\tiny u}$}
\put(10,14){$l$}
\put(10,13){$\text{\tiny v}$}
\end{picture}\\ \\
\eeq
\\
\\
Since a page holds some information about one chord only by non-simultaneity of chords, if a page displays the information about more than one chord, we can isolate the information about each chord as a direct sum of pages each of which carries information about a unique chord. If we call the original matrix $A$ and the two spin-off matrices $B$ and $C$, then inserting $A$ in a book of $m$ pages as follows:
\beq
A^{-} \times A \times A^{+} \label{3A}
\eeq
where $A^{-}$ are the first $m^{-}$ pages of the book, and $A^{+}$ are the last $m^{+}$ pages, with $m^{-}+1+m^+=m$, then we can write:
\beq
A^{-} \times A \times A^{+}=A^{-} \times B \times A^{+}+A^{-} \times C \times A^{+} \label{ABC}
\eeq
representing two chord diagrams, one of which has its $(m^- +1)$-st chord represented by the page $B$, the other has its $(m^- +1)$-st chord represented by $C$. We have done this for a page $A$ containing the information about two chords. We generalize ~\eqref{3A} by iterating this process for pages that contain information about more than 2 chords and generalize ~\eqref{ABC} by iteration for books with two or more pages written as a sum of matrices.\\

We now discuss the band sum move proper. By virtue of the fact that we have:
\beq
\Delta
\setlength{\unitlength}{0.3cm}
\begin{picture}(5,4)(-0.5,2)
\thicklines
\put(1,0){\vector(0,1){5}}
\multiput(1,1)(0.4,0){5}{\circle*{0.2}}
\multiput(1,3)(0.4,0){5}{\circle*{0.2}}
\end{picture}=
\Delta
\setlength{\unitlength}{0.3cm}
\begin{picture}(4.5,4)(-0.5,2)
\thicklines
\put(1,0){\vector(0,1){5}}
\multiput(1,3)(0.4,0){5}{\circle*{0.2}}
\end{picture} \times \;
\Delta
\setlength{\unitlength}{0.3cm}
\begin{picture}(5,4)(-0.5,2)
\thicklines
\put(1,0){\vector(0,1){5}}
\multiput(1,1)(0.4,0){5}{\circle*{0.2}}
\end{picture}
\eeq
\\
or by abuse of notation:
\beq
\setlength{\unitlength}{0.4cm}
\begin{picture}(7,5)(0,3)
\thicklines
\put(0.5,0){\line(0,1){2}}
\put(1.5,0){\line(0,1){2}}
\put(0.5,5){\line(0,1){2}}
\put(1.5,5){\line(0,1){2}}
\put(0,2){\line(0,1){3}}
\put(0,2){\line(1,0){2}}
\put(0,5){\line(1,0){2}}
\put(2,2){\line(0,1){3}}
\put(0.5,3){$\Delta$}
\multiput(2.2,4)(0.4,0){5}{\circle*{0.2}}
\multiput(2.2,3)(0.4,0){5}{\circle*{0.2}}
\end{picture}\\ \\=\setlength{\unitlength}{0.4cm}
\begin{picture}(7,5)(-4,3)
\thicklines
\put(0.5,0){\line(0,1){2}}
\put(1.5,0){\line(0,1){2}}
\put(0.5,5){\line(0,1){2}}
\put(1.5,5){\line(0,1){2}}
\put(0.5,3){\line(0,1){1}}
\put(1.5,3){\line(0,1){1}}
\put(0,2){\line(0,1){1}}
\put(0,2){\line(1,0){2}}
\put(0,5){\line(1,0){2}}
\put(2,2){\line(0,1){1}}
\put(0,4){\line(0,1){1}}
\put(2,4){\line(0,1){1}}
\put(0,3){\line(1,0){2}}
\put(0,4){\line(1,0){2}}
\put(0.5,2.2){$\Delta$}
\put(0.5,4.2){$\Delta$}
\multiput(2.2,4.5)(0.4,0){5}{\circle*{0.2}}
\multiput(2.2,2.5)(0.4,0){5}{\circle*{0.2}}
\end{picture}\\ \\
\eeq
\\ \\
then studying the doubling of chords during the band sum move can be done one chord after another. It suffices to work with one page at a time. During the handle slide of the handle corresponding to the $i$-th component $K_i$ over the $j$-th component $K_j$ of the link $L$, we encounter two different situations:
\beq
\setlength{\unitlength}{0.3cm}
\begin{picture}(18,14)
\put(6,0){\line(0,1){13}}
\put(12,0){\line(0,1){13}}
\linethickness{0.5mm}
\put(3,2){\line(0,1){7}}
\put(14,2){\line(0,1){7}}
\multiput(3,5)(0.5,0){22}{\line(1,0){0.2}}
\multiput(8,8)(1,0){3}{\circle*{0.2}}
\put(5,1){$\text{\small n}$}
\put(13,1){$\text{\small p}$}
\put(1,9){\circle{1.5}}
\put(0.7,8.7){$j$}
\put(16,9){\circle{1.5}}
\put(15.7,8.7){$j$}
\end{picture}
\eeq
\\
\\
where $n<p$. In this first case, a given chord starts and ends on the $j$-th component, with matrix representation given by:
\beq
\setlength{\unitlength}{0.4cm}
\begin{picture}(14,14)
\put(2,10){\line(1,-1){5}}
\put(7,5){\line(1,-1){4}}
\put(11,1){\line(1,-1){1}}
\thicklines
\put(2,5){\oval(2,12)[l]}
\put(13,5){\oval(2,12)[r]}
\put(0,3){$j$}
\put(0,8){$i$}
\put(4,12){$i$}
\put(9,12){$j$}
\put(7,1){\line(0,1){4}}
\put(7,1){\line(1,0){4}}
\put(7,5){\line(1,0){4}}
\put(11,1){\line(0,1){4}}
\put(8,1.8){$1$}
\put(6.5,1.8){$\text{\tiny p}$}
\put(10,4){$1$}
\put(6.5,4){$\text{\tiny n}$}
\put(10,5.3){$\text{\tiny p}$}
\put(8,5.3){$\text{\tiny n}$}
\end{picture}
\eeq
\\
\\
\\
where without loss of generality we have chosen $i<j$ and the case $i>j$ is dealt with by a simple change of basis. Under a band sum move we obtain:
\beq
\setlength{\unitlength}{0.3cm}
\begin{picture}(18,14)
\put(6,0){\line(0,1){13}}
\put(12,0){\line(0,1){13}}
\linethickness{0.5mm}
\put(3,2){\line(0,1){2}}
\put(4,2){\line(0,1){2}}
\put(3,6){\line(0,1){3}}
\put(4,6){\line(0,1){3}}
\put(14,2){\line(0,1){2}}
\put(15,2){\line(0,1){2}}
\put(14,6){\line(0,1){3}}
\put(15,6){\line(0,1){3}}
\multiput(5,5)(0.5,0){16}{\line(1,0){0.2}}
\multiput(8,8)(1,0){3}{\circle*{0.2}}
\put(2,4){\line(1,0){3}}
\put(2,4){\line(0,1){2}}
\put(2,6){\line(1,0){3}}
\put(5,4){\line(0,1){2}}
\put(3,4.5){$\Delta$}
\put(13,4){\line(1,0){3}}
\put(13,4){\line(0,1){2}}
\put(13,6){\line(1,0){3}}
\put(16,4){\line(0,1){2}}
\put(14,4.5){$\Delta$}
\put(5,1){$\text{\small n}$}
\put(13,1){$\text{\small p}$}
\put(5,9){\circle{1.5}}
\put(4.7,8.7){$i$}
\put(13,9){\circle{1.5}}
\put(12.7,8.7){$i$}
\put(1,9){\circle{1.5}}
\put(0.7,8.7){$j$}
\put(16,9){\circle{1.5}}
\put(15.7,8.7){$j$}
\end{picture}
\eeq
\\
which equals:
\begin{align}
\setlength{\unitlength}{0.3cm}
\begin{picture}(18,11)(0,4)
\put(6,0){\line(0,1){13}}
\put(12,0){\line(0,1){13}}
\linethickness{0.5mm}
\put(3,2){\line(0,1){7}}
\put(4,2){\line(0,1){7}}
\put(14,2){\line(0,1){7}}
\put(15,2){\line(0,1){7}}
\multiput(4,5)(0.5,0){22}{\line(1,0){0.2}}
\multiput(8,8)(1,0){3}{\circle*{0.2}}
\put(5,1){$\text{\small n}$}
\put(13,1){$\text{\small p}$}
\put(5,9){\circle{1.5}}
\put(4.7,8.7){$i$}
\put(13,9){\circle{1.5}}
\put(12.7,8.7){$i$}
\put(1,9){\circle{1.5}}
\put(0.7,8.7){$j$}
\put(16,9){\circle{1.5}}
\put(15.7,8.7){$j$}
\end{picture} &+ \setlength{\unitlength}{0.3cm}
\begin{picture}(18,11)(0,4)
\put(6,0){\line(0,1){13}}
\put(12,0){\line(0,1){13}}
\linethickness{0.5mm}
\put(3,2){\line(0,1){7}}
\put(4,2){\line(0,1){7}}
\put(14,2){\line(0,1){7}}
\put(15,2){\line(0,1){7}}
\multiput(4,5)(0.5,0){20}{\line(1,0){0.2}}
\multiput(8,8)(1,0){3}{\circle*{0.2}}
\put(5,1){$\text{\small n}$}
\put(13,1){$\text{\small p}$}
\put(5,9){\circle{1.5}}
\put(4.7,8.7){$i$}
\put(13,9){\circle{1.5}}
\put(12.7,8.7){$i$}
\put(1,9){\circle{1.5}}
\put(0.7,8.7){$j$}
\put(16,9){\circle{1.5}}
\put(15.7,8.7){$j$}
\end{picture} \nonumber \\
+ \setlength{\unitlength}{0.3cm}
\begin{picture}(18,14)(0,4)
\put(6,0){\line(0,1){13}}
\put(12,0){\line(0,1){13}}
\linethickness{0.5mm}
\put(3,2){\line(0,1){7}}
\put(4,2){\line(0,1){7}}
\put(14,2){\line(0,1){7}}
\put(15,2){\line(0,1){7}}
\multiput(3,5)(0.5,0){22}{\line(1,0){0.2}}
\multiput(8,8)(1,0){3}{\circle*{0.2}}
\put(5,1){$\text{\small n}$}
\put(13,1){$\text{\small p}$}
\put(5,9){\circle{1.5}}
\put(4.7,8.7){$i$}
\put(13,9){\circle{1.5}}
\put(12.7,8.7){$i$}
\put(1,9){\circle{1.5}}
\put(0.7,8.7){$j$}
\put(16,9){\circle{1.5}}
\put(15.7,8.7){$j$}
\end{picture} &+ \setlength{\unitlength}{0.3cm}
\begin{picture}(18,14)(0,4)
\put(6,0){\line(0,1){13}}
\put(12,0){\line(0,1){13}}
\linethickness{0.5mm}
\put(3,2){\line(0,1){7}}
\put(4,2){\line(0,1){7}}
\put(14,2){\line(0,1){7}}
\put(15,2){\line(0,1){7}}
\multiput(3,5)(0.5,0){24}{\line(1,0){0.2}}
\multiput(8,8)(1,0){3}{\circle*{0.2}}
\put(5,1){$\text{\small n}$}
\put(13,1){$\text{\small p}$}
\put(5,9){\circle{1.5}}
\put(4.7,8.7){$i$}
\put(13,9){\circle{1.5}}
\put(12.7,8.7){$i$}
\put(1,9){\circle{1.5}}
\put(0.7,8.7){$j$}
\put(16,9){\circle{1.5}}
\put(15.7,8.7){$j$}
\end{picture}\\ \nonumber
\end{align}
\\ \\
\\
a sum of chord diagrams that correspond, in this order, to the following sum of matrices:\\
\begin{align}
\setlength{\unitlength}{0.4cm}
\begin{picture}(14,14)(0,-1)
\put(2,10){\line(1,-1){4}}
\put(7,5){\line(1,-1){4}}
\put(7,10){\line(1,-1){4}}
\put(2,5){\line(1,-1){4}}
\thicklines
\put(2,5){\oval(2,12)[l]}
\put(12,5){\oval(2,12)[r]}
\put(0,3){$j$}
\put(0,8){$i$}
\put(4,12){$i$}
\put(9,12){$j$}
\put(7,1){\line(0,1){4}}
\put(7,1){\line(1,0){4}}
\put(7,5){\line(1,0){4}}
\put(11,1){\line(0,1){4}}
\put(6.5,1.8){$\text{\tiny p}$}
\put(6.5,4){$\text{\tiny n}$}
\put(10,5.3){$\text{\tiny p}$}
\put(8,5.3){$\text{\tiny n}$}
\put(2,1){\line(0,1){4}}
\put(2,1){\line(1,0){4}}
\put(2,5){\line(1,0){4}}
\put(6,1){\line(0,1){4}}
\put(3,1.8){$1$}
\put(1.5,1.8){$\text{\tiny p}$}
\put(1.5,4){$\text{\tiny n}$}
\put(5,5.3){$\text{\tiny p}$}
\put(3,5.3){$\text{\tiny n}$}
\put(2,6){\line(0,1){4}}
\put(2,6){\line(1,0){4}}
\put(2,10){\line(1,0){4}}
\put(6,6){\line(0,1){4}}
\put(1.5,6.8){$\text{\tiny p}$}
\put(1.5,9){$\text{\tiny n}$}
\put(5,10.3){$\text{\tiny p}$}
\put(3,10.3){$\text{\tiny n}$}
\put(7,6){\line(0,1){4}}
\put(7,6){\line(1,0){4}}
\put(7,10){\line(1,0){4}}
\put(11,6){\line(0,1){4}}
\put(6.5,6.8){$\text{\tiny p}$}
\put(10,9){$1$}
\put(6.5,9){$\text{\tiny n}$}
\put(10,10.3){$\text{\tiny p}$}
\put(8,10.3){$\text{\tiny n}$}
\end{picture}
 &+ \setlength{\unitlength}{0.4cm}
\begin{picture}(14,14)(-1,-1)
\put(2,10){\line(1,-1){4}}
\put(7,5){\line(1,-1){4}}
\put(7,10){\line(1,-1){4}}
\put(2,5){\line(1,-1){4}}
\thicklines
\put(2,5){\oval(2,12)[l]}
\put(12,5){\oval(2,12)[r]}
\put(0,3){$j$}
\put(0,8){$i$}
\put(4,12){$i$}
\put(9,12){$j$}
\put(7,1){\line(0,1){4}}
\put(7,1){\line(1,0){4}}
\put(7,5){\line(1,0){4}}
\put(11,1){\line(0,1){4}}
\put(6.5,1.8){$\text{\tiny p}$}
\put(6.5,4){$\text{\tiny n}$}
\put(10,5.3){$\text{\tiny p}$}
\put(8,5.3){$\text{\tiny n}$}
\put(2,1){\line(0,1){4}}
\put(2,1){\line(1,0){4}}
\put(2,5){\line(1,0){4}}
\put(6,1){\line(0,1){4}}
\put(1.5,1.8){$\text{\tiny p}$}
\put(1.5,4){$\text{\tiny n}$}
\put(5,5.3){$\text{\tiny p}$}
\put(3,5.3){$\text{\tiny n}$}
\put(2,6){\line(0,1){4}}
\put(2,6){\line(1,0){4}}
\put(2,10){\line(1,0){4}}
\put(6,6){\line(0,1){4}}
\put(3,6.8){$1$}
\put(1.5,6.8){$\text{\tiny p}$}
\put(5,9){$1$}
\put(1.5,9){$\text{\tiny n}$}
\put(5,10.3){$\text{\tiny p}$}
\put(3,10.3){$\text{\tiny n}$}
\put(7,6){\line(0,1){4}}
\put(7,6){\line(1,0){4}}
\put(7,10){\line(1,0){4}}
\put(11,6){\line(0,1){4}}
\put(6.5,6.8){$\text{\tiny p}$}
\put(6.5,9){$\text{\tiny n}$}
\put(10,10.3){$\text{\tiny p}$}
\put(8,10.3){$\text{\tiny n}$}
\end{picture}
 \nonumber \\
+ \setlength{\unitlength}{0.4cm}
\begin{picture}(14,14)(-1,0)
\put(2,10){\line(1,-1){4}}
\put(7,5){\line(1,-1){4}}
\put(7,10){\line(1,-1){4}}
\put(2,5){\line(1,-1){4}}
\thicklines
\put(2,5){\oval(2,12)[l]}
\put(12,5){\oval(2,12)[r]}
\put(0,3){$j$}
\put(0,8){$i$}
\put(4,12){$i$}
\put(9,12){$j$}
\put(7,1){\line(0,1){4}}
\put(7,1){\line(1,0){4}}
\put(7,5){\line(1,0){4}}
\put(11,1){\line(0,1){4}}
\put(6.5,1.8){$\text{\tiny p}$}
\put(6.5,4){$\text{\tiny n}$}
\put(10,5.3){$\text{\tiny p}$}
\put(8,5.3){$\text{\tiny n}$}
\put(2,1){\line(0,1){4}}
\put(2,1){\line(1,0){4}}
\put(2,5){\line(1,0){4}}
\put(6,1){\line(0,1){4}}
\put(1.5,1.8){$\text{\tiny p}$}
\put(5,4){$1$}
\put(1.5,4){$\text{\tiny n}$}
\put(5,5.3){$\text{\tiny p}$}
\put(3,5.3){$\text{\tiny n}$}
\put(2,6){\line(0,1){4}}
\put(2,6){\line(1,0){4}}
\put(2,10){\line(1,0){4}}
\put(6,6){\line(0,1){4}}
\put(1.5,6.8){$\text{\tiny p}$}
\put(1.5,9){$\text{\tiny n}$}
\put(5,10.3){$\text{\tiny p}$}
\put(3,10.3){$\text{\tiny n}$}
\put(7,6){\line(0,1){4}}
\put(7,6){\line(1,0){4}}
\put(7,10){\line(1,0){4}}
\put(11,6){\line(0,1){4}}
\put(8,6.8){$1$}
\put(6.5,6.8){$\text{\tiny p}$}
\put(6.5,9){$\text{\tiny n}$}
\put(10,10.3){$\text{\tiny p}$}
\put(8,10.3){$\text{\tiny n}$}
\end{picture}
 &\quad + \setlength{\unitlength}{0.4cm}
\begin{picture}(14,14)(-1,0)
\put(2,10){\line(1,-1){4}}
\put(7,5){\line(1,-1){4}}
\put(7,10){\line(1,-1){4}}
\put(2,5){\line(1,-1){4}}
\thicklines
\put(2,5){\oval(2,12)[l]}
\put(12,5){\oval(2,12)[r]}
\put(0,3){$j$}
\put(0,8){$i$}
\put(4,12){$i$}
\put(9,12){$j$}
\put(7,1){\line(0,1){4}}
\put(7,1){\line(1,0){4}}
\put(7,5){\line(1,0){4}}
\put(11,1){\line(0,1){4}}
\put(8,1.8){$1$}
\put(6.5,1.8){$\text{\tiny p}$}
\put(10,4){$1$}
\put(6.5,4){$\text{\tiny n}$}
\put(10,5.3){$\text{\tiny p}$}
\put(8,5.3){$\text{\tiny n}$}
\put(2,1){\line(0,1){4}}
\put(2,1){\line(1,0){4}}
\put(2,5){\line(1,0){4}}
\put(6,1){\line(0,1){4}}
\put(1.5,1.8){$\text{\tiny p}$}
\put(1.5,4){$\text{\tiny n}$}
\put(5,5.3){$\text{\tiny p}$}
\put(3,5.3){$\text{\tiny n}$}
\put(2,6){\line(0,1){4}}
\put(2,6){\line(1,0){4}}
\put(2,10){\line(1,0){4}}
\put(6,6){\line(0,1){4}}
\put(1.5,6.8){$\text{\tiny p}$}
\put(1.5,9){$\text{\tiny n}$}
\put(5,10.3){$\text{\tiny p}$}
\put(3,10.3){$\text{\tiny n}$}
\put(7,6){\line(0,1){4}}
\put(7,6){\line(1,0){4}}
\put(7,10){\line(1,0){4}}
\put(11,6){\line(0,1){4}}
\put(6.5,6.8){$\text{\tiny p}$}
\put(6.5,9){$\text{\tiny n}$}
\put(10,10.3){$\text{\tiny p}$}
\put(8,10.3){$\text{\tiny n}$}
\end{picture}
 \nonumber
\end{align}
\\ \\
\\ \\
\newpage
which combines into:
\beq
\setlength{\unitlength}{0.4cm}
\begin{picture}(14,13)
\put(2,10){\line(1,-1){4}}
\put(7,5){\line(1,-1){4}}
\put(7,10){\line(1,-1){4}}
\put(2,5){\line(1,-1){4}}
\thicklines
\put(2,5){\oval(2,12)[l]}
\put(12,5){\oval(2,12)[r]}
\put(0,3){$j$}
\put(0,8){$i$}
\put(4,12){$i$}
\put(9,12){$j$}
\put(7,1){\line(0,1){4}}
\put(7,1){\line(1,0){4}}
\put(7,5){\line(1,0){4}}
\put(11,1){\line(0,1){4}}
\put(8,1.8){$1$}
\put(6.5,1.8){$\text{\tiny p}$}
\put(10,4){$1$}
\put(6.5,4){$\text{\tiny n}$}
\put(10,5.3){$\text{\tiny p}$}
\put(8,5.3){$\text{\tiny n}$}
\put(2,1){\line(0,1){4}}
\put(2,1){\line(1,0){4}}
\put(2,5){\line(1,0){4}}
\put(6,1){\line(0,1){4}}
\put(3,1.8){$1$}
\put(1.5,1.8){$\text{\tiny p}$}
\put(5,4){$1$}
\put(1.5,4){$\text{\tiny n}$}
\put(5,5.3){$\text{\tiny p}$}
\put(3,5.3){$\text{\tiny n}$}
\put(2,6){\line(0,1){4}}
\put(2,6){\line(1,0){4}}
\put(2,10){\line(1,0){4}}
\put(6,6){\line(0,1){4}}
\put(3,6.8){$1$}
\put(1.5,6.8){$\text{\tiny p}$}
\put(5,9){$1$}
\put(1.5,9){$\text{\tiny n}$}
\put(5,10.3){$\text{\tiny p}$}
\put(3,10.3){$\text{\tiny n}$}
\put(7,6){\line(0,1){4}}
\put(7,6){\line(1,0){4}}
\put(7,10){\line(1,0){4}}
\put(11,6){\line(0,1){4}}
\put(8,6.8){$1$}
\put(6.5,6.8){$\text{\tiny p}$}
\put(10,9){$1$}
\put(6.5,9){$\text{\tiny n}$}
\put(10,10.3){$\text{\tiny p}$}
\put(8,10.3){$\text{\tiny n}$}
\end{picture}
\eeq
\\
\\
The second situation we can have is the case where the chord starts on the $j$-th component but ends on some other $l$-th component. Without loss of generality we can pick $i<j<l$. For other arrangements of these indices we modify the basis for our matrices accordingly. We discuss the case $l=i$ right after since it doesn't consist in a basis change. Pictorially we have:
\beq
\setlength{\unitlength}{0.3cm}
\begin{picture}(18,14)
\put(6,0){\line(0,1){13}}
\put(12,0){\line(0,1){13}}
\linethickness{0.5mm}
\put(3,2){\line(0,1){7}}
\put(14,2){\line(0,1){7}}
\multiput(3,5)(0.5,0){22}{\line(1,0){0.2}}
\multiput(8,8)(1,0){3}{\circle*{0.2}}
\put(5,1){$\text{\small n}$}
\put(13,1){$\text{\small p}$}
\put(1,9){\circle{1.5}}
\put(0.7,8.7){$j$}
\put(16,9){\circle{1.5}}
\put(15.7,8.7){$l$}
\end{picture}
\eeq
\\
\\
with the following matrix representation:
\beq
\setlength{\unitlength}{0.3cm}
\begin{picture}(19,21)
\put(8,6){\line(1,-1){4}}
\put(13,11){\line(1,-1){4}}
\thicklines
\put(3,9){\oval(2,18)[l]}
\put(17,9){\oval(2,18)[r]}
\put(8,2){\line(1,0){4}}
\put(8,2){\line(0,1){4}}
\put(8,6){\line(1,0){4}}
\put(12,2){\line(0,1){4}}
\put(9,2.8){$\text{\small 1}$}
\put(13,7){\line(1,0){4}}
\put(13,7){\line(0,1){4}}
\put(13,11){\line(1,0){4}}
\put(17,7){\line(0,1){4}}
\put(16,10){$\text{\small 1}$}
\put(9,6.3){$\text{\tiny n}$}
\put(7.3,3){$\text{\tiny p}$}
\put(12.3,10){$\text{\tiny n}$}
\put(16,11.3){$\text{\tiny p}$}
\put(1,14){$i$}
\put(1,9){$j$}
\put(1,4){$l$}
\put(5,19){$i$}
\put(10,19){$j$}
\put(15,19){$l$}
\end{picture}
\eeq
\\
After a band sum move we get the following chord diagram:
\beq
\setlength{\unitlength}{0.3cm}
\begin{picture}(18,14)
\put(6,0){\line(0,1){13}}
\put(12,0){\line(0,1){13}}
\linethickness{0.5mm}
\put(3,2){\line(0,1){2}}
\put(4,2){\line(0,1){2}}
\put(3,6){\line(0,1){3}}
\put(4,6){\line(0,1){3}}
\put(14,2){\line(0,1){7}}
\multiput(5,5)(0.5,0){18}{\line(1,0){0.2}}
\multiput(8,8)(1,0){3}{\circle*{0.2}}
\put(2,4){\line(1,0){3}}
\put(2,4){\line(0,1){2}}
\put(2,6){\line(1,0){3}}
\put(5,4){\line(0,1){2}}
\put(3,4.5){$\Delta$}
\put(5,1){$\text{\small n}$}
\put(13,1){$\text{\small p}$}
\put(5,9){\circle{1.5}}
\put(4.7,8.7){$i$}
\put(1,9){\circle{1.5}}
\put(0.7,8.7){$j$}
\put(13,9){\circle{1.5}}
\put(12.7,8.7){$l$}
\end{picture}
\eeq
which equals:
\begin{align}
\setlength{\unitlength}{0.3cm}
\begin{picture}(18,11)(0,3)
\put(6,0){\line(0,1){13}}
\put(12,0){\line(0,1){13}}
\linethickness{0.5mm}
\put(3,2){\line(0,1){7}}
\put(4,2){\line(0,1){7}}
\put(14,2){\line(0,1){7}}
\multiput(4,5)(0.5,0){20}{\line(1,0){0.2}}
\multiput(8,8)(1,0){3}{\circle*{0.2}}
\put(5,1){$\text{\small n}$}
\put(13,1){$\text{\small p}$}
\put(5,9){\circle{1.5}}
\put(4.7,8.7){$i$}
\put(1,9){\circle{1.5}}
\put(0.7,8.7){$j$}
\put(16,9){\circle{1.5}}
\put(15.7,8.7){$l$}
\end{picture}+\setlength{\unitlength}{0.3cm}
\begin{picture}(18,11)(0,3)
\put(6,0){\line(0,1){13}}
\put(12,0){\line(0,1){13}}
\linethickness{0.5mm}
\put(3,2){\line(0,1){7}}
\put(4,2){\line(0,1){7}}
\put(14,2){\line(0,1){7}}
\multiput(3,5)(0.5,0){22}{\line(1,0){0.2}}
\multiput(8,8)(1,0){3}{\circle*{0.2}}
\put(5,1){$\text{\small n}$}
\put(13,1){$\text{\small p}$}
\put(5,9){\circle{1.5}}
\put(4.7,8.7){$i$}
\put(1,9){\circle{1.5}}
\put(0.7,8.7){$j$}
\put(16,9){\circle{1.5}}
\put(15.7,8.7){$l$}
\end{picture} \nonumber
\end{align}
\\ \\
\\
represented, in this order, by the sum of matrices:
\begin{align}
\setlength{\unitlength}{0.3cm}
\begin{picture}(19,13)(0,8)
\put(3,6){\line(1,-1){4}}
\put(13,16){\line(1,-1){4}}
\thicklines
\put(3,9){\oval(2,18)[l]}
\put(17,9){\oval(2,18)[r]}
\put(3,2){\line(1,0){4}}
\put(3,2){\line(0,1){4}}
\put(3,6){\line(1,0){4}}
\put(7,2){\line(0,1){4}}
\put(3.5,2.5){$\text{\small 1}$}
\put(13,12){\line(1,0){4}}
\put(13,12){\line(0,1){4}}
\put(13,16){\line(1,0){4}}
\put(17,12){\line(0,1){4}}
\put(16,15){$1$}
\put(1,14){$i$}
\put(1,9){$j$}
\put(1,4){$l$}
\put(5,19){$i$}
\put(10,19){$j$}
\put(15,19){$l$}
\end{picture} + \setlength{\unitlength}{0.3cm}
\begin{picture}(19,13)(0,8)
\put(8,6){\line(1,-1){4}}
\put(13,11){\line(1,-1){4}}
\thicklines
\put(3,9){\oval(2,18)[l]}
\put(17,9){\oval(2,18)[r]}
\put(8,2){\line(1,0){4}}
\put(8,2){\line(0,1){4}}
\put(8,6){\line(1,0){4}}
\put(12,2){\line(0,1){4}}
\put(8.5,2.5){$\text{\small 1}$}
\put(13,7){\line(1,0){4}}
\put(13,7){\line(0,1){4}}
\put(13,11){\line(1,0){4}}
\put(17,7){\line(0,1){4}}
\put(16,10){$1$}
\put(1,14){$i$}
\put(1,9){$j$}
\put(1,4){$l$}
\put(5,19){$i$}
\put(10,19){$j$}
\put(15,19){$l$}
\end{picture} \nonumber
\end{align}
\\ \\
\\
\\
\\
\\
combining into:
\beq
\setlength{\unitlength}{0.3cm}
\begin{picture}(19,21)
\put(8,6){\line(1,-1){4}}
\put(13,11){\line(1,-1){4}}
\put(3,6){\line(1,-1){4}}
\put(13,16){\line(1,-1){4}}
\thicklines
\put(3,9){\oval(2,18)[l]}
\put(17,9){\oval(2,18)[r]}
\put(8,2){\line(1,0){4}}
\put(8,2){\line(0,1){4}}
\put(8,6){\line(1,0){4}}
\put(12,2){\line(0,1){4}}
\put(8.5,2.5){$\text{\small 1}$}
\put(13,7){\line(1,0){4}}
\put(13,7){\line(0,1){4}}
\put(13,11){\line(1,0){4}}
\put(17,7){\line(0,1){4}}
\put(16,10){$\text{\small 1}$}
\put(3,2){\line(1,0){4}}
\put(3,2){\line(0,1){4}}
\put(3,6){\line(1,0){4}}
\put(7,2){\line(0,1){4}}
\put(3.5,2.5){$\text{\small 1}$}
\put(13,12){\line(1,0){4}}
\put(13,12){\line(0,1){4}}
\put(13,16){\line(1,0){4}}
\put(17,12){\line(0,1){4}}
\put(16,15){$\text{\small 1}$}
\put(9,6.3){$\text{\tiny n}$}
\put(4,6.3){$\text{\tiny n}$}
\put(12.3,15){$\text{\tiny n}$}
\put(2.3,3){$\text{\tiny p}$}
\put(12.3,10){$\text{\tiny n}$}
\put(16,16.3){$\text{\tiny p}$}
\put(1,14){$i$}
\put(1,9){$j$}
\put(1,4){$l$}
\put(5,19){$i$}
\put(10,19){$j$}
\put(15,19){$l$}
\end{picture}
\eeq
\\
\\
For the case $l=i$, we pictorially have:
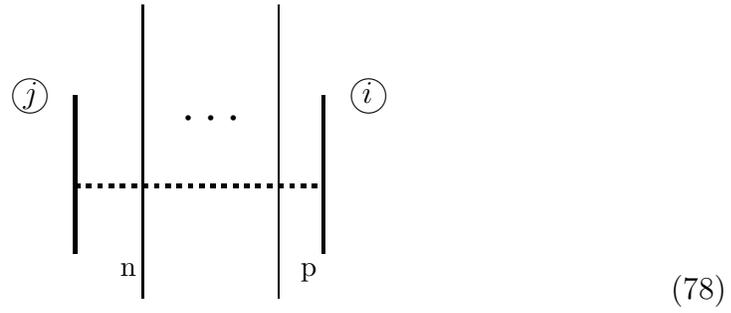
\beq
\setlength{\unitlength}{0.3cm}
\begin{picture}(18,15)
\put(6,0){\line(0,1){13}}
\put(12,0){\line(0,1){13}}
\linethickness{0.5mm}
\put(3,2){\line(0,1){7}}
\put(14,2){\line(0,1){7}}
\multiput(3,5)(0.5,0){22}{\line(1,0){0.2}}
\multiput(8,8)(1,0){3}{\circle*{0.2}}
\put(5,1){$\text{\small n}$}
\put(13,1){$\text{\small p}$}
\put(1,9){\circle{1.5}}
\put(0.7,8.7){$j$}
\put(16,9){\circle{1.5}}
\put(15.7,8.7){$i$}
\end{picture}
\eeq
\\
with the following matrix representation:
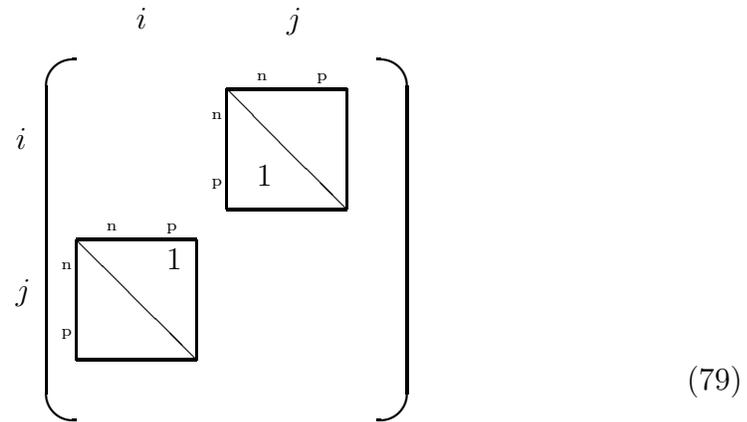
\beq
\setlength{\unitlength}{0.4cm}
\begin{picture}(14,14)
\put(7,10){\line(1,-1){4}}
\put(2,5){\line(1,-1){4}}
\thicklines
\put(2,5){\oval(2,12)[l]}
\put(12,5){\oval(2,12)[r]}
\put(0,3){$j$}
\put(0,8){$i$}
\put(4,12){$i$}
\put(9,12){$j$}
\put(2,1){\line(0,1){4}}
\put(2,1){\line(1,0){4}}
\put(2,5){\line(1,0){4}}
\put(6,1){\line(0,1){4}}
\put(1.5,1.8){$\text{\tiny p}$}
\put(5,4){$1$}
\put(1.5,4){$\text{\tiny n}$}
\put(5,5.3){$\text{\tiny p}$}
\put(3,5.3){$\text{\tiny n}$}
\put(7,6){\line(0,1){4}}
\put(7,6){\line(1,0){4}}
\put(7,10){\line(1,0){4}}
\put(11,6){\line(0,1){4}}
\put(8,6.8){$1$}
\put(6.5,6.8){$\text{\tiny p}$}
\put(6.5,9){$\text{\tiny n}$}
\put(10,10.3){$\text{\tiny p}$}
\put(8,10.3){$\text{\tiny n}$}
\end{picture}
\eeq
\\
\\
After a band sum move, we get the following diagram:
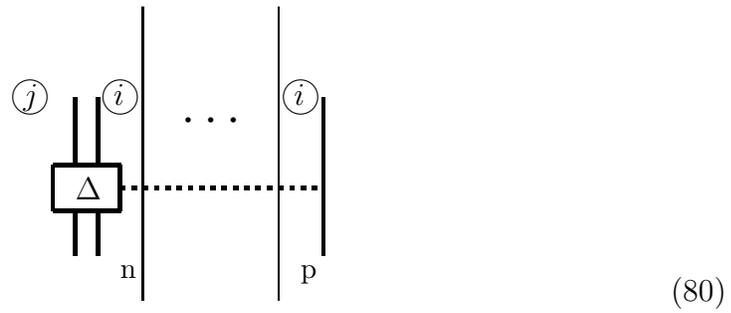
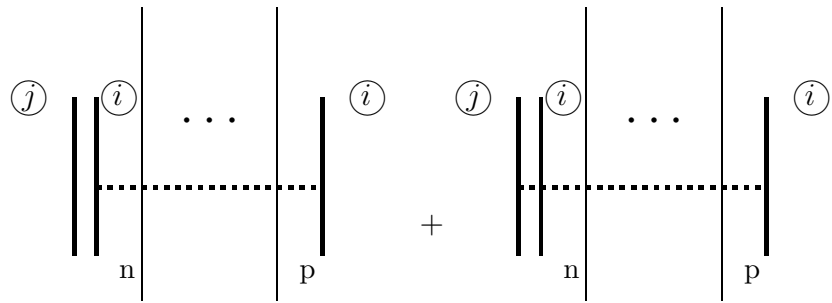
\beq
\setlength{\unitlength}{0.3cm}
\begin{picture}(18,13)
\put(6,0){\line(0,1){13}}
\put(12,0){\line(0,1){13}}
\linethickness{0.5mm}
\put(3,2){\line(0,1){2}}
\put(4,2){\line(0,1){2}}
\put(3,6){\line(0,1){3}}
\put(4,6){\line(0,1){3}}
\put(14,2){\line(0,1){7}}
\multiput(5,5)(0.5,0){18}{\line(1,0){0.2}}
\multiput(8,8)(1,0){3}{\circle*{0.2}}
\put(2,4){\line(1,0){3}}
\put(2,4){\line(0,1){2}}
\put(2,6){\line(1,0){3}}
\put(5,4){\line(0,1){2}}
\put(3,4.5){$\Delta$}
\put(5,1){$\text{\small n}$}
\put(13,1){$\text{\small p}$}
\put(5,9){\circle{1.5}}
\put(4.7,8.7){$i$}
\put(1,9){\circle{1.5}}
\put(0.7,8.7){$j$}
\put(13,9){\circle{1.5}}
\put(12.7,8.7){$i$}
\end{picture}
\eeq
\\
which equals:
\begin{align}
\setlength{\unitlength}{0.3cm}
\begin{picture}(18,11)(0,3)
\put(6,0){\line(0,1){13}}
\put(12,0){\line(0,1){13}}
\linethickness{0.5mm}
\put(3,2){\line(0,1){7}}
\put(4,2){\line(0,1){7}}
\put(14,2){\line(0,1){7}}
\multiput(4,5)(0.5,0){20}{\line(1,0){0.2}}
\multiput(8,8)(1,0){3}{\circle*{0.2}}
\put(5,1){$\text{\small n}$}
\put(13,1){$\text{\small p}$}
\put(5,9){\circle{1.5}}
\put(4.7,8.7){$i$}
\put(1,9){\circle{1.5}}
\put(0.7,8.7){$j$}
\put(16,9){\circle{1.5}}
\put(15.7,8.7){$i$}
\end{picture}+\setlength{\unitlength}{0.3cm}
\begin{picture}(18,11)(0,3)
\put(6,0){\line(0,1){13}}
\put(12,0){\line(0,1){13}}
\linethickness{0.5mm}
\put(3,2){\line(0,1){7}}
\put(4,2){\line(0,1){7}}
\put(14,2){\line(0,1){7}}
\multiput(3,5)(0.5,0){22}{\line(1,0){0.2}}
\multiput(8,8)(1,0){3}{\circle*{0.2}}
\put(5,1){$\text{\small n}$}
\put(13,1){$\text{\small p}$}
\put(5,9){\circle{1.5}}
\put(4.7,8.7){$i$}
\put(1,9){\circle{1.5}}
\put(0.7,8.7){$j$}
\put(16,9){\circle{1.5}}
\put(15.7,8.7){$i$}
\end{picture} \nonumber
\end{align}
\\ \\
\\
represented in this order by the sum of matrices:
\begin{align}
\setlength{\unitlength}{0.4cm}
\begin{picture}(14,9)(0,5)
\put(2,10){\line(1,-1){4}}
\thicklines
\put(2,5){\oval(2,12)[l]}
\put(12,5){\oval(2,12)[r]}
\put(0,3){$j$}
\put(0,8){$i$}
\put(4,12){$i$}
\put(9,12){$j$}
\put(2,6){\line(0,1){4}}
\put(2,6){\line(1,0){4}}
\put(2,10){\line(1,0){4}}
\put(6,6){\line(0,1){4}}
\put(3,6.8){$1$}
\put(1.5,6.8){$\text{\tiny p}$}
\put(5,9){$1$}
\put(1.5,9){$\text{\tiny n}$}
\put(5,10.3){$\text{\tiny p}$}
\put(3,10.3){$\text{\tiny n}$}
\end{picture}+
\setlength{\unitlength}{0.4cm}
\begin{picture}(14,9)(0,5)
\put(7,10){\line(1,-1){4}}
\put(2,5){\line(1,-1){4}}
\thicklines
\put(2,5){\oval(2,12)[l]}
\put(12,5){\oval(2,12)[r]}
\put(0,3){$j$}
\put(0,8){$i$}
\put(4,12){$i$}
\put(9,12){$j$}
\put(2,1){\line(0,1){4}}
\put(2,1){\line(1,0){4}}
\put(2,5){\line(1,0){4}}
\put(6,1){\line(0,1){4}}
\put(1.5,1.8){$\text{\tiny p}$}
\put(5,4){$1$}
\put(1.5,4){$\text{\tiny n}$}
\put(5,5.3){$\text{\tiny p}$}
\put(3,5.3){$\text{\tiny n}$}
\put(7,6){\line(0,1){4}}
\put(7,6){\line(1,0){4}}
\put(7,10){\line(1,0){4}}
\put(11,6){\line(0,1){4}}
\put(8,6.8){$1$}
\put(6.5,6.8){$\text{\tiny p}$}
\put(6.5,9){$\text{\tiny n}$}
\put(10,10.3){$\text{\tiny p}$}
\put(8,10.3){$\text{\tiny n}$}
\end{picture} \nonumber
\end{align}\\ \\
\\ \\
\\ \\
\\
combining into:
\beq
\setlength{\unitlength}{0.4cm}
\begin{picture}(14,14)
\put(2,10){\line(1,-1){4}}
\put(7,10){\line(1,-1){4}}
\put(2,5){\line(1,-1){4}}
\thicklines
\put(2,5){\oval(2,12)[l]}
\put(12,5){\oval(2,12)[r]}
\put(0,3){$j$}
\put(0,8){$i$}
\put(4,12){$i$}
\put(9,12){$j$}
\put(2,1){\line(0,1){4}}
\put(2,1){\line(1,0){4}}
\put(2,5){\line(1,0){4}}
\put(6,1){\line(0,1){4}}
\put(1.5,1.8){$\text{\tiny p}$}
\put(5,4){$1$}
\put(1.5,4){$\text{\tiny n}$}
\put(5,5.3){$\text{\tiny p}$}
\put(3,5.3){$\text{\tiny n}$}
\put(2,6){\line(0,1){4}}
\put(2,6){\line(1,0){4}}
\put(2,10){\line(1,0){4}}
\put(6,6){\line(0,1){4}}
\put(3,6.8){$1$}
\put(1.5,6.8){$\text{\tiny p}$}
\put(5,9){$1$}
\put(1.5,9){$\text{\tiny n}$}
\put(5,10.3){$\text{\tiny p}$}
\put(3,10.3){$\text{\tiny n}$}
\put(7,6){\line(0,1){4}}
\put(7,6){\line(1,0){4}}
\put(7,10){\line(1,0){4}}
\put(11,6){\line(0,1){4}}
\put(8,6.8){$1$}
\put(6.5,6.8){$\text{\tiny p}$}
\put(6.5,9){$\text{\tiny n}$}
\put(10,10.3){$\text{\tiny p}$}
\put(8,10.3){$\text{\tiny n}$}
\end{picture}
\eeq
\\

We present now our main result:
\begin{Thm}
For a $q$-components link $L$ on which the handle corresponding to the $i$-th component $K_i$ is sliding over the handle corresponding to the $j$-th component $K_j$, the induced map on $\widetilde{Z}_f$ is denoted by $\mathbf{h}_{ij}$ and is defined by:
\beq
\mathbf{h}_{ij}\widetilde{Z}_f(L)=\sum_{m \geq 0} \sum_{\substack{|I|=|J|=m \\ |U|=|V|=m}} c_{IJUV} \stck_{1 \leq a \leq m}M_{ij} ^T A_{i_a,j_a,u_a,v_a}M_{ij}
\eeq
where $M_{ij}$ is a $qN\times qN$ matrix with ones on its diagonal and the $ji$ block is the $N\times N$ identity matrix $I_N$. For $i<j$, we write such a matrix as:
\beq
\setlength{\unitlength}{0.3cm}
\begin{picture}(14,14)
\put(2,10){$1$}
\put(3,9){$1$}
\put(6,6){\circle*{0.2}}
\put(7,5){\circle*{0.2}}
\put(8,4){\circle*{0.2}}
\put(11,1){$1$}
\thicklines
\put(2,5){\oval(2,12)[l]}
\put(12,5){\oval(2,12)[r]}
\put(0,3){$j$}
\put(0,8){$i$}
\put(4,12){$i$}
\put(9,12){$j$}
\put(2,1){\line(0,1){4}}
\put(2,1){\line(1,0){4}}
\put(2,5){\line(1,0){4}}
\put(6,1){\line(0,1){4}}
\put(3.5,2.5){$\text{\Large I}_N$}
\end{picture}
\eeq
\\
\\
If we define:
\beq
M_{ij}^m:=\stck_{1 \leq a \leq m}M_{ij}
\eeq
acting on books of $m$ pages chord-wise, then we can rewrite the above formula in compact form as:
\beq
\mathbf{h}_{ij}\widetilde{Z}_f(L)=\sum_{m \geq 0} \sum_{\substack{|I|=|J|=m \\ |U|=|V|=m}} c_{IJUV} M_{ij}^{m\;T} A_{IJUV}M_{ij}^m
\eeq
We may even generalize this further by defining $M_{ij}^{\centerdot}$ to the product of as many matrices $M_{ij}$ as the book they operate on have pages, which leads to having the even simpler formula:
\beq
\mathbf{h}_{ij}\widetilde{Z}_f(L)=M^{\centerdot \;T} \left(\widetilde{Z}_f(L) \right) M^{\centerdot}
\eeq
\end{Thm}
\begin{proof}
The only cases that need to be studied are those where a chord has a foot on the $j$-th component, since all other tangle chord diagrams are left invariant by $\mathbf{h}_{ij}$, and of those there are only two kinds that we presented just before the statement of the Theorem. By elementary matrix multiplication, we have in the first case, for one page:\\
\beq
\setlength{\unitlength}{0.3cm}
\begin{picture}(21,14)(-2,0)
\put(2,10){\line(1,-1){5}}
\put(7,5){\line(1,-1){4}}
\put(11,1){\line(1,-1){1}}
\put(-4,0){$\text{\huge M}_{\text{\Large ij}}^{\text{\large T}}$}
\put(14.5,0){$\text{\huge M}_{\text{\Large ij}}$}
\thicklines
\put(2,5){\oval(2,12)[l]}
\put(13,5){\oval(2,12)[r]}
\put(0,3){$j$}
\put(0,8){$i$}
\put(4,12){$i$}
\put(9,12){$j$}
\put(7,1){\line(0,1){4}}
\put(7,1){\line(1,0){4}}
\put(7,5){\line(1,0){4}}
\put(11,1){\line(0,1){4}}
\put(8,1.8){$\text{\small 1}$}
\put(6.2,1.8){$\text{\tiny p}$}
\put(10,4){$\text{\small 1}$}
\put(6.2,4){$\text{\tiny n}$}
\put(10,5.3){$\text{\tiny p}$}
\put(8,5.3){$\text{\tiny n}$}
\end{picture} =
\setlength{\unitlength}{0.4cm}
\begin{picture}(14,13)
\put(2,10){\line(1,-1){4}}
\put(7,5){\line(1,-1){4}}
\put(7,10){\line(1,-1){4}}
\put(2,5){\line(1,-1){4}}
\thicklines
\put(2,5){\oval(2,12)[l]}
\put(12,5){\oval(2,12)[r]}
\put(0,3){$j$}
\put(0,8){$i$}
\put(4,12){$i$}
\put(9,12){$j$}
\put(7,1){\line(0,1){4}}
\put(7,1){\line(1,0){4}}
\put(7,5){\line(1,0){4}}
\put(11,1){\line(0,1){4}}
\put(8,1.8){$1$}
\put(6.5,1.8){$\text{\tiny p}$}
\put(10,4){$1$}
\put(6.5,4){$\text{\tiny n}$}
\put(10,5.3){$\text{\tiny p}$}
\put(8,5.3){$\text{\tiny n}$}
\put(2,1){\line(0,1){4}}
\put(2,1){\line(1,0){4}}
\put(2,5){\line(1,0){4}}
\put(6,1){\line(0,1){4}}
\put(3,1.8){$1$}
\put(1.5,1.8){$\text{\tiny p}$}
\put(5,4){$1$}
\put(1.5,4){$\text{\tiny n}$}
\put(5,5.3){$\text{\tiny p}$}
\put(3,5.3){$\text{\tiny n}$}
\put(2,6){\line(0,1){4}}
\put(2,6){\line(1,0){4}}
\put(2,10){\line(1,0){4}}
\put(6,6){\line(0,1){4}}
\put(3,6.8){$1$}
\put(1.5,6.8){$\text{\tiny p}$}
\put(5,9){$1$}
\put(1.5,9){$\text{\tiny n}$}
\put(5,10.3){$\text{\tiny p}$}
\put(3,10.3){$\text{\tiny n}$}
\put(7,6){\line(0,1){4}}
\put(7,6){\line(1,0){4}}
\put(7,10){\line(1,0){4}}
\put(11,6){\line(0,1){4}}
\put(8,6.8){$1$}
\put(6.5,6.8){$\text{\tiny p}$}
\put(10,9){$1$}
\put(6.5,9){$\text{\tiny n}$}
\put(10,10.3){$\text{\tiny p}$}
\put(8,10.3){$\text{\tiny n}$}
\end{picture} \nonumber
\eeq
\\
\\
which is what we expected from the above considerations. In the second case of interest, if $l \neq i$:
\begin{align}
\setlength{\unitlength}{0.3cm}
\begin{picture}(23,19)(0,2)
\put(8,6){\line(1,-1){4}}
\put(13,11){\line(1,-1){4}}
\put(-3,2){$\text{\huge M}_{\text{\Large ij}}^{\text{\large T}}$}
\put(18.5,2){$\text{\huge M}_{\text{\Large ij}}$}
\thicklines
\put(3,9){\oval(2,18)[l]}
\put(17,9){\oval(2,18)[r]}
\put(8,2){\line(1,0){4}}
\put(8,2){\line(0,1){4}}
\put(8,6){\line(1,0){4}}
\put(12,2){\line(0,1){4}}
\put(8.5,2.5){$1$}
\put(13,7){\line(1,0){4}}
\put(13,7){\line(0,1){4}}
\put(13,11){\line(1,0){4}}
\put(17,7){\line(0,1){4}}
\put(16,10){$1$}
\put(9,6.3){$\text{\tiny n}$}
\put(7.3,3){$\text{\tiny p}$}
\put(12.3,10){$\text{\tiny n}$}
\put(16,11.3){$\text{\tiny p}$}
\put(1,14){$i$}
\put(1,9){$j$}
\put(1,4){$l$}
\put(5,19){$i$}
\put(10,19){$j$}
\put(15,19){$l$}
\end{picture} \nonumber \\
\setlength{\unitlength}{0.3cm}
\begin{picture}(19,21)(0,2)
\put(8,6){\line(1,-1){4}}
\put(13,11){\line(1,-1){4}}
\put(3,6){\line(1,-1){4}}
\put(13,16){\line(1,-1){4}}
\put(-2,9){$\text{\Huge =}$}
\thicklines
\put(3,9){\oval(2,18)[l]}
\put(17,9){\oval(2,18)[r]}
\put(8,2){\line(1,0){4}}
\put(8,2){\line(0,1){4}}
\put(8,6){\line(1,0){4}}
\put(12,2){\line(0,1){4}}
\put(8.5,2.5){$1$}
\put(13,7){\line(1,0){4}}
\put(13,7){\line(0,1){4}}
\put(13,11){\line(1,0){4}}
\put(17,7){\line(0,1){4}}
\put(16,10){$1$}
\put(3,2){\line(1,0){4}}
\put(3,2){\line(0,1){4}}
\put(3,6){\line(1,0){4}}
\put(7,2){\line(0,1){4}}
\put(3.5,2.5){$1$}
\put(13,12){\line(1,0){4}}
\put(13,12){\line(0,1){4}}
\put(13,16){\line(1,0){4}}
\put(17,12){\line(0,1){4}}
\put(16,15){$1$}
\put(9,6.3){$\text{\tiny n}$}
\put(4,6.3){$\text{\tiny n}$}
\put(12.3,15){$\text{\tiny n}$}
\put(2.3,3){$\text{\tiny p}$}
\put(12.3,10){$\text{\tiny n}$}
\put(16,16.3){$\text{\tiny p}$}
\put(1,14){$i$}
\put(1,9){$j$}
\put(1,4){$l$}
\put(5,19){$i$}
\put(10,19){$j$}
\put(15,19){$l$}
\end{picture}
\end{align}
\\ \\
\\
\\
\\
which is what was expected. In case $l=i$:
\begin{align}
\setlength{\unitlength}{0.4cm}
\begin{picture}(16,14)
\put(2,5){\line(1,-1){4}}
\put(7,10){\line(1,-1){4}}
\put(-3,0){$\text{\huge M}_{\text{\Large ij}}^{\text{\large T}}$}
\put(14.5,0){$\text{\huge M}_{\text{\Large ij}}$}
\thicklines
\put(2,5){\oval(2,12)[l]}
\put(13,5){\oval(2,12)[r]}
\put(0,3){$j$}
\put(0,8){$i$}
\put(4,12){$i$}
\put(9,12){$j$}
\put(0,3){$j$}
\put(0,8){$i$}
\put(4,12){$i$}
\put(9,12){$j$}
\put(2,1){\line(0,1){4}}
\put(2,1){\line(1,0){4}}
\put(2,5){\line(1,0){4}}
\put(6,1){\line(0,1){4}}
\put(1.5,1.8){$\text{\tiny p}$}
\put(5,4){$1$}
\put(1.5,4){$\text{\tiny n}$}
\put(5,5.3){$\text{\tiny p}$}
\put(3,5.3){$\text{\tiny n}$}
\put(7,6){\line(0,1){4}}
\put(7,6){\line(1,0){4}}
\put(7,10){\line(1,0){4}}
\put(11,6){\line(0,1){4}}
\put(8,6.8){$1$}
\put(6.5,6.8){$\text{\tiny p}$}
\put(6.5,9){$\text{\tiny n}$}
\put(10,10.3){$\text{\tiny p}$}
\put(8,10.3){$\text{\tiny n}$}
\end{picture}  \nonumber \\
\setlength{\unitlength}{0.4cm}
\begin{picture}(14,7)(0,7)
\put(2,10){\line(1,-1){4}}
\put(7,10){\line(1,-1){4}}
\put(2,5){\line(1,-1){4}}
\put(-2,5){$\text{\Huge =}$}
\thicklines
\put(2,5){\oval(2,12)[l]}
\put(12,5){\oval(2,12)[r]}
\put(0,3){$j$}
\put(0,8){$i$}
\put(4,12){$i$}
\put(9,12){$j$}
\put(2,1){\line(0,1){4}}
\put(2,1){\line(1,0){4}}
\put(2,5){\line(1,0){4}}
\put(6,1){\line(0,1){4}}
\put(1.5,1.8){$\text{\tiny p}$}
\put(5,4){$1$}
\put(1.5,4){$\text{\tiny n}$}
\put(5,5.3){$\text{\tiny p}$}
\put(3,5.3){$\text{\tiny n}$}
\put(2,6){\line(0,1){4}}
\put(2,6){\line(1,0){4}}
\put(2,10){\line(1,0){4}}
\put(6,6){\line(0,1){4}}
\put(3,6.8){$1$}
\put(1.5,6.8){$\text{\tiny p}$}
\put(5,9){$1$}
\put(1.5,9){$\text{\tiny n}$}
\put(5,10.3){$\text{\tiny p}$}
\put(3,10.3){$\text{\tiny n}$}
\put(7,6){\line(0,1){4}}
\put(7,6){\line(1,0){4}}
\put(7,10){\line(1,0){4}}
\put(11,6){\line(0,1){4}}
\put(8,6.8){$1$}
\put(6.5,6.8){$\text{\tiny p}$}
\put(6.5,9){$\text{\tiny n}$}
\put(10,10.3){$\text{\tiny p}$}
\put(8,10.3){$\text{\tiny n}$}
\end{picture} \nonumber
\end{align}
\\
\\
\\
\\
\\
\\
\\
This completes the proof.
\end{proof}

\section{Isotopy invariance of $\widetilde{Z}_f$ in book notation}
In \cite{RG} we proved that $\hat{Z}_f$ is an isotopy invariant. By definition, so is $\widetilde{Z}_f$. If we consider $\widetilde{Z}_f(L)$ as just a formal sum of books with complex coefficients however, we have a discretization of tangle chord diagrams by books that is no longer isotopy invariant. We use the fact that ambient isotopic links are in the same class if and only if they are related by the Reidemeister moves ~\cite{Rei} ~\cite{Y}. Instead of studying the behavior of $\widetilde{Z}_f(L)$ under any arbitrary isotopy, we study Reidemeister moves of tangles. They are pictured as follows:\\
\setlength{\unitlength}{0.3cm}
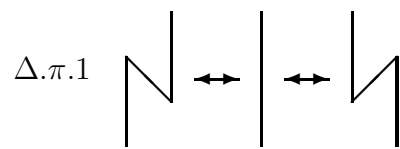
\begin{picture}(19,7)(-12,2)
\put(1,3){$\Delta.\pi.1$}
\thicklines
\put(6,0){\line(0,1){4}}
\put(6,4){\line(1,-1){2}}
\put(8,2){\line(0,1){4}}
\put(11,3){\vector(-1,0){2}}
\put(9,3){\vector(1,0){2}}
\put(12,0){\line(0,1){6}}
\put(15,3){\vector(-1,0){2}}
\put(13,3){\vector(1,0){2}}
\put(16,2){\line(0,1){4}}
\put(16,2){\line(1,1){2}}
\put(18,0){\line(0,1){4}}
\end{picture}\\ \\
\\
\\
\setlength{\unitlength}{0.3cm}
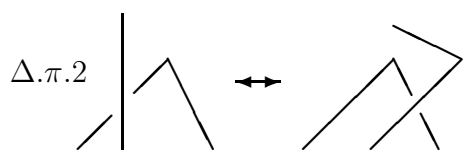
\begin{picture}(21,7)(-12,2)
\put(0,3){$\Delta.\pi.2$}
\thicklines
\put(3,0){\line(1,1){1.5}}
\put(5.5,2.5){\line(1,1){1.5}}
\put(5,0){\line(0,1){6}}
\put(7,4){\line(1,-2){2}}
\put(12,3){\vector(-1,0){2}}
\put(10,3){\vector(1,0){2}}
\put(13,0){\line(1,1){4}}
\put(17,4){\line(1,-2){0.8}}
\put(16,0){\line(1,1){4}}
\put(20,4){\line(-2,1){3}}
\put(19,0){\line(-1,2){0.8}}
\end{picture}\\ \\
\\
\setlength{\unitlength}{0.3cm}
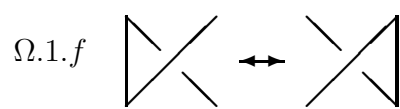
\begin{picture}(19,7)(-12,2)
\put(1,4){$\Omega.1.f$}
\thicklines
\put(6,2){\line(0,1){4}}
\put(6,6){\line(1,-1){1.5}}
\put(8.5,3.5){\line(1,-1){1.5}}
\put(6,2){\line(1,1){4}}
\put(13,4){\vector(-1,0){2}}
\put(11,4){\vector(1,0){2}}
\put(14,2){\line(1,1){4}}
\put(14,6){\line(1,-1){1.5}}
\put(16.5,3.5){\line(1,-1){1.5}}
\put(18,2){\line(0,1){4}}
\end{picture}\\ \\
\\
which is the framed version of the original $\Omega.1$ Reidemeister move.\\
\setlength{\unitlength}{0.3cm}
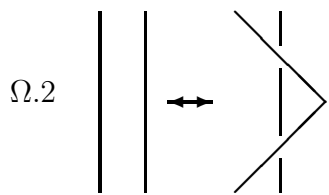
\begin{picture}(15,9)(-15,2)
\put(0,4){$\Omega.2$}
\thicklines
\put(4,0){\line(0,1){8}}
\put(6,0){\line(0,1){8}}
\put(9,4){\vector(-1,0){2}}
\put(7,4){\vector(1,0){2}}
\put(10,0){\line(1,1){4}}
\put(10,8){\line(1,-1){4}}
\put(12,0){\line(0,1){1.5}}
\put(12,2.5){\line(0,1){3}}
\put(12,6.5){\line(0,1){1.5}}
\end{picture}\\ \\
\\
\setlength{\unitlength}{0.3cm}
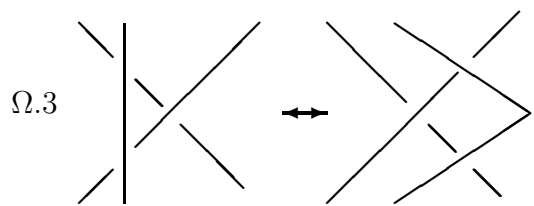
\begin{picture}(24,9)(-11,2)
\put(0,4){$\Omega.3$}
\thicklines
\put(3,0){\line(1,1){1.5}}
\put(5.5,2.5){\line(1,1){5.5}}
\put(3,8){\line(1,-1){1.5}}
\put(5,0){\line(0,1){8}}
\put(5.5,5.5){\line(1,-1){1.2}}
\put(7.5,3.5){\line(1,-1){2.8}}
\put(14,4){\vector(-1,0){2}}
\put(12,4){\vector(1,0){2}}
\put(14,0){\line(1,1){5.8}}
\put(20.5,6.5){\line(1,1){2}}
\put(17,0){\line(3,2){6}}
\put(17,8){\line(3,-2){6}}
\put(14,8){\line(1,-1){3.5}}
\put(18.5,3.5){\line(1,-1){1.2}}
\put(20.5,1.5){\line(1,-1){1.2}}
\end{picture}\\ \\
\\
\begin{Iso1}
The book representation of $\widetilde{Z}_f(L)$ of any link $L$ is invariant under the Reidemeister moves $\Omega.2$ and $\Omega.3$.
\end{Iso1}
\begin{proof}
In both cases we have an equivalence of two tangle diagrams which yield the same page for any chord ending on them after possibly moving the chord up the tangle.
\end{proof}
Those are the only moves that leave $\widetilde{Z}_f(L)$ invariant. For the other Reidemeister moves, there are a few changes to be performed on each page.

\subsection{Behavior of $\widetilde{Z}_f$ under $\Delta.\pi.1$}
In a slice presentation of knots, and using the isotopy invariance of $\widetilde{Z}_f$, it suffices to consider strands other than those involved in the move $\Delta.\pi.1$ to be vertical, and the straightened out strand involved in $\Delta.\pi.1$ to be straight as well. For instance, one of the moves involved in $\Delta.\pi.1$ would look like:
\beq
\setlength{\unitlength}{0.3cm}
\begin{picture}(14,5)(0,3)
\thicklines
\put(0,0){\line(0,1){6}}
\put(3,0){\line(0,1){6}}
\put(9,0){\line(0,1){6}}
\put(12,0){\line(0,1){6}}
\multiput(1,3)(0.5,0){3}{\circle*{0.2}}
\multiput(10,3)(0.5,0){3}{\circle*{0.2}}
\put(6,0){\line(0,1){2}}
\put(5.5,2){\oval(1,2)[t]}
\put(4.5,2){\oval(1,2)[b]}
\put(5,2){\oval(2,4)[tl]}
\put(5,6){\oval(2,4)[br]}
\end{picture}
\longrightarrow
\setlength{\unitlength}{0.3cm}
\begin{picture}(14,5)(-2,3)
\thicklines
\put(0,0){\line(0,1){6}}
\put(3,0){\line(0,1){6}}
\put(9,0){\line(0,1){6}}
\put(12,0){\line(0,1){6}}
\multiput(1,3)(0.5,0){3}{\circle*{0.2}}
\multiput(10,3)(0.5,0){3}{\circle*{0.2}}
\put(6,0){\line(0,1){6}}
\end{picture}
\label{htos}
\eeq
\\
\\
In the expression for $\hat{Z}_f$, powers of $\nu$ cancel self-chords from the Kontsevich integral of the hump. Chords with one foot on the hump however are not canceled, whereas in the expression for $\hat{Z}_f$ of the right hand side of ~\eqref{htos}, all of whose strands are vertical, there are no such chords. This is in particular true of $\widetilde{Z}_f$. Moreover mixed chords between self chords on the hump and long chords with one foot on the hump are not canceled by $\nu$. They are canceled by squeezing the hump into a window of vanishing width. For each long chord with a foot on an ascending (resp. descending) part of the hump, there is a long chord with a foot on a descending (resp. ascending) part of the hump, each tangle chord diagram with log differentials that differ in sign and therefore cancel each other off. This is how we have all long chords ending on the hump go.
Thus we seek a transformation on the expression for $\widetilde{Z}_f$ of the left hand side of ~\eqref{htos} that effectively gets rid of chords with a foot on the hump.
\begin{Zhumpvert}
The $qN \times q(N-8)$ matrix $\text{diag}(M_{\text{hump}\mapsto |})$ with:
\beq
\text{\Large M}_{\text{hump} \mapsto |}=
\setlength{\unitlength}{0.4cm}
\begin{picture}(10,9)(-1,3)
\put(1,7){\line(1,0){3}}
\put(4,7){\line(0,1){3}}
\put(4,0){\line(0,1){4}}
\put(4,4){\line(1,0){4}}
\put(1,6.8){\line(1,0){7}}
\put(1,4.2){\line(1,0){7}}
\thicklines
\put(1,5){\oval(2,10)[l]}
\put(8,5){\oval(2,10)[r]}
\put(2,8){$I_n$}
\put(5,2){$I_{N-n-8}$}
\put(3,5){$O_{8 \times (N-8)}$}
\end{picture} \label{starr1}
\eeq
\\ \\
implements the move ~\eqref{htos}. In case the hump is on the first strand from the left:
\beq
\text{\Large M}_{\text{hump} \mapsto |}=
\setlength{\unitlength}{0.4cm}
\begin{picture}(10,9)(-1,3)
\put(1,6.2){\line(1,0){7}}
\thicklines
\put(1,5){\oval(2,10)[l]}
\put(8,5){\oval(2,10)[r]}
\put(3,2){$I_{N-8}$}
\put(3,7){$O_{8 \times (N-8)}$}
\end{picture}
\eeq
\\
\\
If we let $I_0$ be the empty matrix then \eqref{starr1} implements the move \eqref{htos} for all $n \geq 0$.
\end{Zhumpvert}
\begin{proof}
First observe that since we have a local max on the hump, the skeleton for both $\hat{Z}_f$ and $\widetilde{Z}_f$ is as follows, with added vertical slices for numbering purposes:
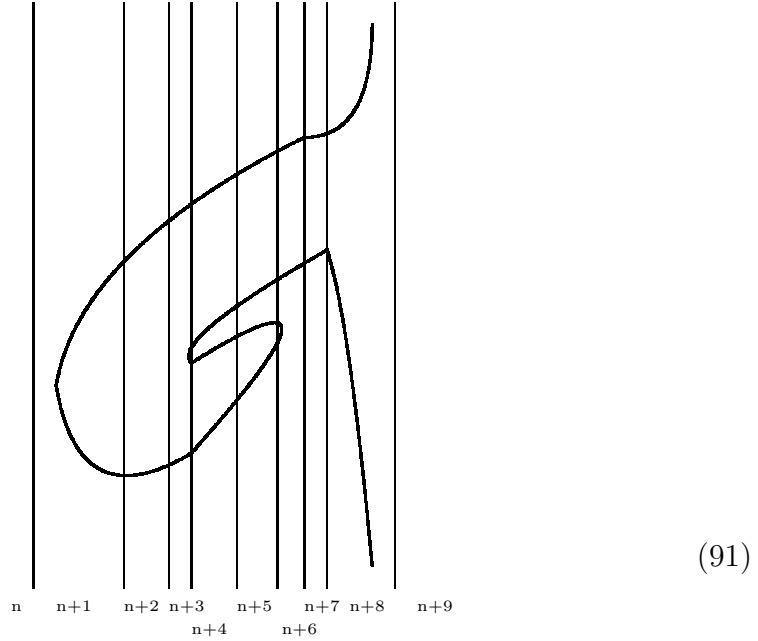
\beq
\setlength{\unitlength}{0.3cm}
\begin{picture}(16,27)
\put(-1,-1){\line(0,1){26}}
\put(3,-1){\line(0,1){26}}
\put(5,-1){\line(0,1){26}}
\put(6,-1){\line(0,1){26}}
\put(8,-1){\line(0,1){26}}
\put(9.8,-1){\line(0,1){26}}
\put(11,-1){\line(0,1){26}}
\put(12,-1){\line(0,1){26}}
\put(15,-1){\line(0,1){26}}
\put(-2,-2){\text{\tiny n}}
\put(0,-2){\text{\tiny n+1}}
\put(3,-2){\text{\tiny n+2}}
\put(5,-2){\text{\tiny n+3}}
\put(6,-3){\text{\tiny n+4}}
\put(8,-2){\text{\tiny n+5}}
\put(10,-3){\text{\tiny n+6}}
\put(11,-2){\text{\tiny n+7}}
\put(13,-2){\text{\tiny n+8}}
\put(16,-2){\text{\tiny n+9}}
\thicklines
\qbezier(6,9)(14,14)(6,5)
\qbezier(6,5)(1,2)(0,8)
\qbezier(0,8)(1,14)(11,19)
\qbezier(11,19)(14,19)(14,24)
\qbezier(6,9)(5,10)(12,14)
\qbezier(12,14)(13,11)(14,0)
\end{picture}
\label{slicehtos}
\eeq
\\
\\
\\
\\
The straightened out strand on the right hand side of ~\eqref{htos} is in the $(n+1)$-st strip. The matrix that takes care of relabeling all the strips under the map depicted in ~\eqref{htos} and kills all the chords on the hump is, for the block corresponding to the component on which the hump is located:
\beq
\text{\Large M}_{\text{hump} \mapsto |}=
\setlength{\unitlength}{0.4cm}
\begin{picture}(10,9)(-1,3)
\put(1,7){\line(1,0){3}}
\put(4,7){\line(0,1){3}}
\put(4,0){\line(0,1){4}}
\put(4,4){\line(1,0){4}}
\put(1,6.8){\line(1,0){7}}
\put(1,4.2){\line(1,0){7}}
\thicklines
\put(1,5){\oval(2,10)[l]}
\put(8,5){\oval(2,10)[r]}
\put(2,8){$I_n$}
\put(5,2){$I_{N-n-8}$}
\put(3,5){$O_{8 \times (N-8)}$}
\end{picture} \label{star1}
\eeq
\\ \\
It is a $N \times (N-8)$ matrix. In other horizontal slices, chords ending on strands other than the one with the hump are moved up to be either in strips indexed $1$ through $n$, or $n+9$ through $N$. Therefore the above matrix also takes care of the transformation of those chords under the move ~\eqref{htos}. We conclude that the $qN \times q(N-8)$ matrix $\text{diag}(M_{\text{\tiny hump} \mapsto |})$ with $M_{\text{\tiny hump} \mapsto |}$ given above implements the move ~\eqref{htos}.
\end{proof}

One would similarly show that the move:
\beq
\setlength{\unitlength}{0.3cm}
\begin{picture}(14,6)(0,3)
\thicklines
\put(0,0){\line(0,1){6}}
\put(3,0){\line(0,1){6}}
\put(9,0){\line(0,1){6}}
\put(12,0){\line(0,1){6}}
\multiput(1,3)(0.5,0){3}{\circle*{0.2}}
\multiput(10,3)(0.5,0){3}{\circle*{0.2}}
\put(6,0){\line(0,1){2}}
\put(6.5,2){\oval(1,2)[t]}
\put(7.5,2){\oval(1,2)[b]}
\put(7,2){\oval(2,4)[tr]}
\put(7,6){\oval(2,4)[bl]}
\end{picture}
\longrightarrow
\setlength{\unitlength}{0.3cm}
\begin{picture}(14,6)(-2,3)
\thicklines
\put(0,0){\line(0,1){6}}
\put(3,0){\line(0,1){6}}
\put(9,0){\line(0,1){6}}
\put(12,0){\line(0,1){6}}
\multiput(1,3)(0.5,0){3}{\circle*{0.2}}
\multiput(10,3)(0.5,0){3}{\circle*{0.2}}
\put(6,0){\line(0,1){6}}
\end{picture}
\label{otherhtos}
\eeq
\\
\\
\\
with a local slicing of the hump given by:
\beq
\setlength{\unitlength}{0.3cm}
\begin{picture}(14,26)
\put(0,-1){\line(0,1){26}}
\put(1,-1){\line(0,1){26}}
\put(2,-1){\line(0,1){26}}
\put(3.5,-1){\line(0,1){26}}
\put(6,-1){\line(0,1){26}}
\put(7,-1){\line(0,1){26}}
\put(8,-1){\line(0,1){26}}
\put(11,-1){\line(0,1){26}}
\put(14,-1){\line(0,1){26}}
\put(-1,-2){\text{\tiny n}}
\put(0,-3){\text{\tiny n+1}}
\put(1,-2){\text{\tiny n+2}}
\put(2,-3){\text{\tiny n+3}}
\put(4,-2){\text{\tiny n+4}}
\put(6,-2){\text{\tiny n+5}}
\put(7,-3){\text{\tiny n+6}}
\put(9,-2){\text{\tiny n+7}}
\put(12,-2){\text{\tiny n+8}}
\put(15,-2){\text{\tiny n+9}}
\thicklines
\qbezier(4,0)(7,10)(1,7)
\qbezier(1,7)(0,8)(7,11)
\qbezier(7,11)(10,0)(13,6)
\qbezier(13,6)(16,18)(6,17)
\qbezier(6,17)(5,18)(4,24)
\end{picture}
\label{sliceotherhtos}
\eeq
\\
\\
\\
is taken care of by the same $qN \times q(N-8)$ matrix $\text{diag}(M_{\text{\tiny hump} \mapsto |})$ defined above.
\begin{rmkslice}
We could equally have taken the following slicing and skeleton:
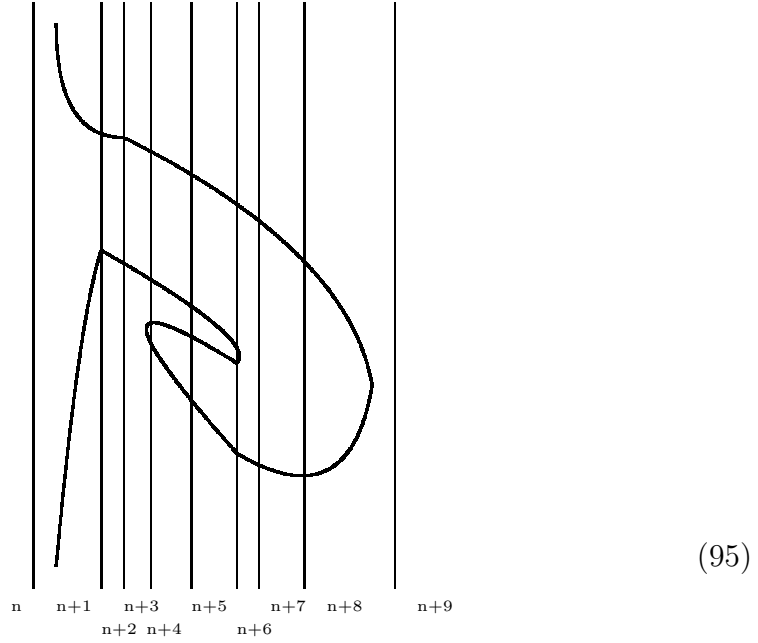
\beq
\setlength{\unitlength}{0.3cm}
\begin{picture}(16,27)
\put(-1,-1){\line(0,1){26}}
\put(2,-1){\line(0,1){26}}
\put(3,-1){\line(0,1){26}}
\put(4.2,-1){\line(0,1){26}}
\put(6,-1){\line(0,1){26}}
\put(8,-1){\line(0,1){26}}
\put(9,-1){\line(0,1){26}}
\put(11,-1){\line(0,1){26}}
\put(15,-1){\line(0,1){26}}
\put(-2,-2){\text{\tiny n}}
\put(0,-2){\text{\tiny n+1}}
\put(2,-3){\text{\tiny n+2}}
\put(3,-2){\text{\tiny n+3}}
\put(4,-3){\text{\tiny n+4}}
\put(6,-2){\text{\tiny n+5}}
\put(8,-3){\text{\tiny n+6}}
\put(9.5,-2){\text{\tiny n+7}}
\put(12,-2){\text{\tiny n+8}}
\put(16,-2){\text{\tiny n+9}}
\thicklines
\qbezier(8,9)(0,14)(8,5)
\qbezier(8,5)(13,2)(14,8)
\qbezier(14,8)(13,14)(3,19)
\qbezier(3,19)(0,19)(0,24)
\qbezier(8,9)(9,10)(2,14)
\qbezier(2,14)(1,11)(0,0)
\end{picture}
\eeq
\\
\\
\\
\\
and the matrix would have been the same.
\end{rmkslice}

We have shown:
\begin{humptos}
The map $\mathbf{h}_{\Delta.\pi.1, \;\text{hump}\mapsto |}$ that acts on $\widetilde{Z}_f(L)$ in book notation to effect two of the $\Delta.\pi.1$ Reidemeister moves as depicted in ~\eqref{htos} and ~\eqref{otherhtos} is given on a page by:
\beq
\mathbf{h}_{\substack{\Delta.\pi.1 \\ \text{\tiny hump} \mapsto |}}A:=M_{\substack{\Delta.\pi.1 \\ \text{\tiny hump} \mapsto |}}^T A M_{\substack{\Delta.\pi.1 \\ \text{\tiny hump} \mapsto |}}
\eeq
where $M_{\Delta.\pi.1,\;\text{\tiny hump} \mapsto |}$ is the $qN \times q(N-8)$ matrix $\text{diag}(M_{\text{\tiny hump} \mapsto |})$ with $M_{\text{\tiny hump} \mapsto |}$ given by:
\beq
\text{\Large M}_{\text{hump} \mapsto |}=
\setlength{\unitlength}{0.4cm}
\begin{picture}(10,7)(-1,3)
\put(1,7){\line(1,0){3}}
\put(4,7){\line(0,1){3}}
\put(4,0){\line(0,1){4}}
\put(4,4){\line(1,0){4}}
\put(1,6.8){\line(1,0){7}}
\put(1,4.2){\line(1,0){7}}
\thicklines
\put(1,5){\oval(2,10)[l]}
\put(8,5){\oval(2,10)[r]}
\put(2,8){$I_n$}
\put(5,2){$I_{N-n-8}$}
\put(3,5){$O_{8 \times (N-8)}$}
\end{picture} \label{star1}
\eeq
\\ \\
for $n \geq 0$.
\end{humptos}
\begin{Zhumpget}
The matrix that implements the move:
\beq
\setlength{\unitlength}{0.3cm}
\begin{picture}(14,5)(-2,3)
\thicklines
\put(0,0){\line(0,1){6}}
\put(3,0){\line(0,1){6}}
\put(9,0){\line(0,1){6}}
\put(12,0){\line(0,1){6}}
\multiput(1,3)(0.5,0){3}{\circle*{0.2}}
\multiput(10,3)(0.5,0){3}{\circle*{0.2}}
\put(6,0){\line(0,1){6}}
\end{picture}
\longrightarrow
\setlength{\unitlength}{0.3cm}
\begin{picture}(14,5)(0,3)
\thicklines
\put(0,0){\line(0,1){6}}
\put(3,0){\line(0,1){6}}
\put(9,0){\line(0,1){6}}
\put(12,0){\line(0,1){6}}
\multiput(1,3)(0.5,0){3}{\circle*{0.2}}
\multiput(10,3)(0.5,0){3}{\circle*{0.2}}
\put(6,0){\line(0,1){2}}
\put(5.5,2){\oval(1,2)[t]}
\put(4.5,2){\oval(1,2)[b]}
\put(5,2){\oval(2,4)[tl]}
\put(5,6){\oval(2,4)[br]}
\end{picture}
\label{stoh}
\eeq
\\
\\
\\
with the same slicing convention as in ~\eqref{slicehtos} is given by the $N \times (N+8)$ matrix:
\beq
\setlength{\unitlength}{0.4cm}
\begin{picture}(15,11)(-1,0)
\put(0,5){\line(1,0){3}}
\put(3,5){\line(0,1){3}}
\put(3.2,0){\line(0,1){8}}
\put(5.8,0){\line(0,1){8}}
\put(6,0){\line(0,1){5}}
\put(6,5){\line(1,0){5}}
\thicklines
\put(0,4){\oval(2,8)[l]}
\put(11,4){\oval(2,8)[r]}
\put(0.5,6){$I_n$}
\put(3.2,4){$O_{N \times 8}$}
\put(7,3){$I_{N-n}$}
\put(-7,4){$\text{\Large M}_{| \mapsto \text{hump}}=$}
\end{picture}
\eeq
for $n \geq 0$.
\end{Zhumpget}
\begin{proof}
In other horizontal slices, chords ending on strands other than the one on which the hump is located are moved up to end in strips indexed $1$ through $n$ or $n+1$ through $N$. Therefore the above matrix also takes care of the transformation of those chords under the move ~\eqref{stoh}. We conclude that the $qN \times q(N+8)$ matrix $\text{diag}(M_{| \mapsto \text{\tiny hump}})$ with $M_{| \mapsto \text{\tiny hump}}$ given above implements the move ~\eqref{stoh}.
\end{proof}

Likewise, one would show that the same matrix implements the move:
\beq
\setlength{\unitlength}{0.3cm}
\begin{picture}(14,5)(-2,3)
\thicklines
\put(0,0){\line(0,1){6}}
\put(3,0){\line(0,1){6}}
\put(9,0){\line(0,1){6}}
\put(12,0){\line(0,1){6}}
\multiput(1,3)(0.5,0){3}{\circle*{0.2}}
\multiput(10,3)(0.5,0){3}{\circle*{0.2}}
\put(6,0){\line(0,1){6}}
\end{picture}
\longrightarrow
\setlength{\unitlength}{0.3cm}
\begin{picture}(14,5)(0,3)
\thicklines
\put(0,0){\line(0,1){6}}
\put(3,0){\line(0,1){6}}
\put(9,0){\line(0,1){6}}
\put(12,0){\line(0,1){6}}
\multiput(1,3)(0.5,0){3}{\circle*{0.2}}
\multiput(10,3)(0.5,0){3}{\circle*{0.2}}
\put(6,0){\line(0,1){2}}
\put(6.5,2){\oval(1,2)[t]}
\put(7.5,2){\oval(1,2)[b]}
\put(7,2){\oval(2,4)[tr]}
\put(7,6){\oval(2,4)[bl]}
\end{picture}
\label{otherstoh}
\eeq
\\
\\
\\
with the slicing on the hump given as in ~\eqref{sliceotherhtos}. We have shown:
\begin{stohump}
The map $\mathbf{h}_{\Delta.\pi.1, \;|\mapsto \text{\tiny hump}}$ that acts on $\widetilde{Z}_f(L)$ in book notation to effect two of the $\Delta.\pi.1$ Reidemeister moves as depicted in ~\eqref{stoh} and ~\eqref{otherstoh} is given on a page by:
\beq
\mathbf{h}_{\substack{\Delta.\pi.1 \\ | \mapsto \text{\tiny hump} }}A:=M_{\substack{\Delta.\pi.1 \\ | \mapsto \text{\tiny hump}}}^T A M_{\substack{\Delta.\pi.1 \\ | \mapsto \text{\tiny hump}}}
\eeq
where $M_{\Delta.\pi.1,\;| \mapsto \text{\tiny hump}}$ is the $qN \times q(N+8)$ $\text{diag}(M_{| \mapsto \text{\tiny hump}})$ matrix with $M_{| \mapsto \text{\tiny hump}}$ given by:
\beq
\setlength{\unitlength}{0.4cm}
\begin{picture}(15,11)(-1,0)
\put(0,5){\line(1,0){3}}
\put(3,5){\line(0,1){3}}
\put(3.2,0){\line(0,1){8}}
\put(5.8,0){\line(0,1){8}}
\put(6,0){\line(0,1){5}}
\put(6,5){\line(1,0){5}}
\thicklines
\put(0,4){\oval(2,8)[l]}
\put(11,4){\oval(2,8)[r]}
\put(0.5,6){$I_n$}
\put(3.2,4){$O_{N \times 8}$}
\put(7,3){$I_{N-n}$}
\put(-7,4){$\text{\Large M}_{| \mapsto \text{hump}}=$}
\end{picture}
\eeq
\\
\end{stohump}
\newpage
Finally we consider the map that implements the following move on strands:
\beq
\setlength{\unitlength}{0.3cm}
\begin{picture}(14,5)(0,3)
\thicklines
\put(0,0){\line(0,1){6}}
\put(3,0){\line(0,1){6}}
\put(9,0){\line(0,1){6}}
\put(12,0){\line(0,1){6}}
\multiput(1,3)(0.5,0){3}{\circle*{0.2}}
\multiput(10,3)(0.5,0){3}{\circle*{0.2}}
\put(6,0){\line(0,1){2}}
\put(5.5,2){\oval(1,2)[t]}
\put(4.5,2){\oval(1,2)[b]}
\put(5,2){\oval(2,4)[tl]}
\put(5,6){\oval(2,4)[br]}
\end{picture}
\longrightarrow
\setlength{\unitlength}{0.3cm}
\begin{picture}(14,5)(-1,3)
\thicklines
\put(0,0){\line(0,1){6}}
\put(3,0){\line(0,1){6}}
\put(9,0){\line(0,1){6}}
\put(12,0){\line(0,1){6}}
\multiput(1,3)(0.5,0){3}{\circle*{0.2}}
\multiput(10,3)(0.5,0){3}{\circle*{0.2}}
\put(6,0){\line(0,1){2}}
\put(6.5,2){\oval(1,2)[t]}
\put(7.5,2){\oval(1,2)[b]}
\put(7,2){\oval(2,4)[tr]}
\put(7,6){\oval(2,4)[bl]}
\end{picture}
\label{htoh}
\eeq
\\
\\
With slicing conventions as in ~\eqref{slicehtos} and ~\eqref{sliceotherhtos}, the corresponding transformation matrix is given by the $N \times N$ matrix:
\beq
\text{\Large M}_{\text{hump} \mapsto \text{hump}}=
\setlength{\unitlength}{0.4cm}
\begin{picture}(12,8)(-1,4)
\put(0,7){\line(1,0){3}}
\put(3,7){\line(0,1){3}}
\put(6,0){\line(0,1){4}}
\put(6,4){\line(1,0){4}}
\thicklines
\put(0,5){\oval(2,10)[l]}
\put(10,5){\oval(2,10)[r]}
\put(1,8){$I_n$}
\put(7,2){$I_{N-n-8}$}
\put(3.5,5){$O_{8 \times 8}$}
\end{picture}=\text{\Large M}_{\text{hump} \mapsto |} \cdot \text{\Large M}_{| \mapsto \text{hump}}
\eeq
\\ \\ \\
In other horizontal slices, chords ending on strands other than the one on which the hump is located are moved up to end in strips indexed $1$ through $n$ or $n+9$ through $N$. Therefore the above matrix also takes care of the transformation of those chords under the move ~\eqref{htoh}. We conclude that the $qN \times qN$ matrix $\text{diag}(M_{\text{\tiny hump} \mapsto \text{\tiny hump}})$ with $M_{\text{\tiny hump} \mapsto \text{\tiny hump}}$ given above implements the move ~\eqref{htoh}. We have shown:
\begin{humptohump}
The map $\mathbf{h}_{\Delta.\pi.1, \;\text{\tiny hump} \mapsto \text{\tiny hump}}$ that acts on $\widetilde{Z}_f(L)$ in book notation to effect the $\Delta.\pi.1$ Reidemeister move as depicted in ~\eqref{htoh} is given on a page by:
\beq
\mathbf{h}_{\substack{\Delta.\pi.1 \\  \text{\tiny hump} \mapsto \text{\tiny hump} }}A:=M_{\substack{\Delta.\pi.1 \\ \text{\tiny hump} \mapsto \text{\tiny hump}}}^T A M_{\substack{\Delta.\pi.1 \\ \text{\tiny hump} \mapsto \text{\tiny hump}}}
\eeq
where $M_{\Delta.\pi.1,\;\text{\tiny hump} \mapsto \text{\tiny hump}}$ is the $qN \times qN$ matrix $\text{diag}(M_{\text{\tiny hump} \mapsto \text{\tiny hump}})$ with $M_{\text{\tiny hump} \mapsto \text{\tiny hump}}$ given by:
\beq
\text{\Large M}_{\text{hump} \mapsto \text{hump}}=
\setlength{\unitlength}{0.4cm}
\begin{picture}(12,8)(-1,4)
\put(0,7){\line(1,0){3}}
\put(3,7){\line(0,1){3}}
\put(6,0){\line(0,1){4}}
\put(6,4){\line(1,0){4}}
\thicklines
\put(0,5){\oval(2,10)[l]}
\put(10,5){\oval(2,10)[r]}
\put(1,8){$I_n$}
\put(7,2){$I_{N-n-8}$}
\put(3.5,5){$O_{8 \times 8}$}
\end{picture}
\eeq
\\
\\
\end{humptohump}

\subsection{Behavior of $\widetilde{Z}_f$ under $\Delta.\pi.2$}
The two tangle diagrams involved in the statement for the $\Delta.\pi.2$ Reidemeister move have the following presentation with vertical strips for book purposes:
\beq
\setlength{\unitlength}{0.3cm}
\begin{picture}(16,9)
\put(4,-2){\line(0,1){10}}
\put(7,-2){\line(0,1){10}}
\put(10,-2){\line(0,1){10}}
\put(2,-2){$\text{\tiny n}$}
\put(5,-2){$\text{\tiny n+1}$}
\put(8,-2){$\text{\tiny n+2}$}
\put(11,-2){$\text{\tiny n+3}$}
\thicklines
\put(0,0){\line(1,1){2.2}}
\put(2.7,2.7){\line(1,1){2.3}}
\put(7,5){\oval(4,2)[t]}
\put(9,5){\line(1,-1){5}}
\put(2.5,0){\line(0,1){9}}
\put(1,8){\circle{2}}
\put(0.8,7.8){$t$}
\put(12.8,3.8){$s$}
\put(13,4){\circle{2}}
\end{picture}
\eeq
\\
and:
\beq
\setlength{\unitlength}{0.3cm}
\begin{picture}(16,16)
\put(4,-2){\line(0,1){18}}
\put(7,-2){\line(0,1){18}}
\put(10,-2){\line(0,1){18}}
\put(1.5,-2){$\text{\tiny n}$}
\put(5,-2){$\text{\tiny n+1}$}
\put(8,-2){$\text{\tiny n+2}$}
\put(11,-2){$\text{\tiny n+3}$}
\thicklines
\put(0,2){\line(1,1){5}}
\put(7,7){\oval(4,2)[t]}
\put(9,7){\line(1,-1){2}}
\put(11.5,4.5){\line(1,-1){2}}
\put(3,0){\line(0,1){2}}
\put(3,2){\line(3,1){9}}
\put(12,5){\line(0,1){4}}
\put(3,12){\line(3,-1){9}}
\put(3,12){\line(0,1){2}}
\put(1,13){\circle{2}}
\put(0.8,12.8){$t$}
\put(1,6){\circle{2}}
\put(0.8,5.8){$s$}
\end{picture}
\eeq
\\
\\
\begin{Iso3}
With tangle diagrams as above, the map $\mathbf{h}_{\Delta.\pi.2}$ that acts on $\widetilde{Z}_f(L)$ in book notation to effect the Reidemeister move $\Delta.\pi.2$, is given on a page by:
\beq
\mathbf{h}_{\Delta.\pi.2}A:=M_{\Delta.\pi.2}^T A M_{\Delta.\pi.2}
\eeq
where $M_{\Delta.\pi.2}$ is the $qN \times qN$ identity matrix save for the $tt$ block which is of the form:
\beq
\setlength{\unitlength}{0.4cm}
\begin{picture}(14,16)
\put(0,9){\line(1,0){3}}
\put(3,9){\line(0,1){3}}
\put(7,0){\line(0,1){5}}
\put(7,5){\line(1,0){5}}
\put(3.2,8.2){$0$}
\put(4.2,8.2){$0$}
\put(5.2,8.2){$0$}
\put(6.2,8.2){$1$}
\put(3.2,7.2){$0$}
\put(4.2,7.2){$0$}
\put(5.2,7.2){$0$}
\put(6.2,7.2){$0$}
\put(3.2,6.2){$0$}
\put(4.2,6.2){$0$}
\put(5.2,6.2){$0$}
\put(6.2,6.2){$0$}
\put(3.2,5.2){$1$}
\put(4.2,5.2){$0$}
\put(5.2,5.2){$0$}
\put(6.2,5.2){$0$}
\put(3,13){$\text{\tiny n}$}
\put(4,14){$\text{\tiny n+1}$}
\put(5,13){$\text{\tiny n+2}$}
\put(6,14){$\text{\tiny n+3}$}
\put(-2.5,8.2){$\text{\tiny n}$}
\put(-2.5,7.2){$\text{\tiny n+1}$}
\put(-2.5,6.2){$\text{\tiny n+2}$}
\put(-2.5,5.2){$\text{\tiny n+3}$}
\thicklines
\put(0,6){\oval(2,12)[l]}
\put(12,6){\oval(2,12)[r]}
\put(0.5,10){$I_{n-1}$}
\put(8,3){$I_{N-n-3}$}
\end{picture}
\eeq
\\
\end{Iso3}
\begin{proof}
Recall that chords on a given tangle are moved until they reach a local max. Keeping this in mind, under the move $\Delta.\pi.2$, only the $t$-th component moves, and thus only chords with a foot on this component need be considered. Suppose a chord is stretching between that component and the $s$-th component. Moving the chord up the $s$-th component so that it's localized near its local max, the foot of that chord on the $s$-th component is either in the $n+1$-st or $n+2$-nd strip, and its other foot on the $t$-th component is then in the $n$-th or $n+3$-rd strip depending on which tangle we are looking at. Under $\Delta.\pi.2$, chords with one foot in the $n$-th strip (resp. the $n+3$-rd strip) are moved to end up in the $n+3$-rd strip (resp. the $n$-th strip). Suppose now a chord is stretching between the $t$-th component and another $l$-th component, $l \neq s$. For a needle shaped local max on the $s$-th component, the corresponding vertical strips indexed by $n+1$ and $n+2$ are sufficient narrow that chords from some $l$-th component, $l\neq s$ end on the $t$-th component in some strip other than the $n+1$-st or $n+2$-nd strips, that is in the $n$-th or $n+3$-rd strips, which are interchanged under $\Delta.\pi.2$. We conclude that $\Delta.\pi.2$ is implemented by a $qN \times qN$ matrix with a $tt$ block of the form above which effectively switches the $n$-th and $n+3$-rd strips for $n \geq 1$, taking $I_0$ to be the empty matrix in the event that $n=1$ indexes the first vertical strip.
\end{proof}

\subsection{Behavior of $\widetilde{Z}_f$ under $\Omega.1.f$}
The two equivalent tangle diagrams under this move are represented as follows along with the strips for book representation purposes:
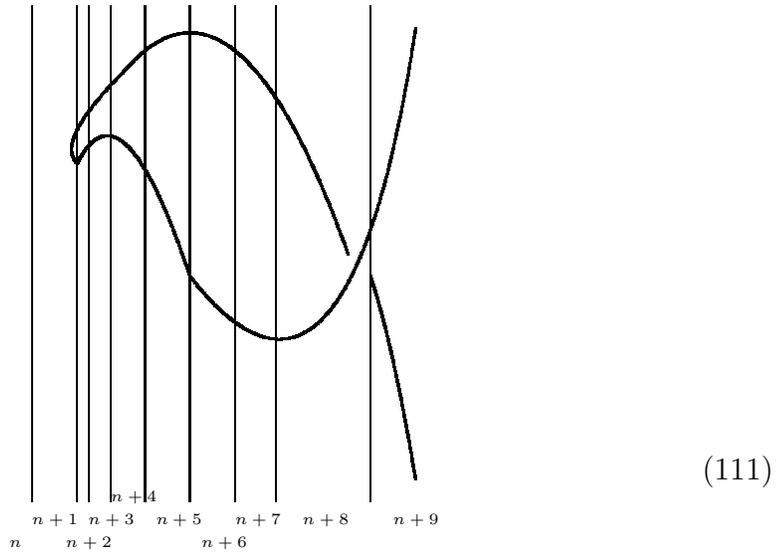
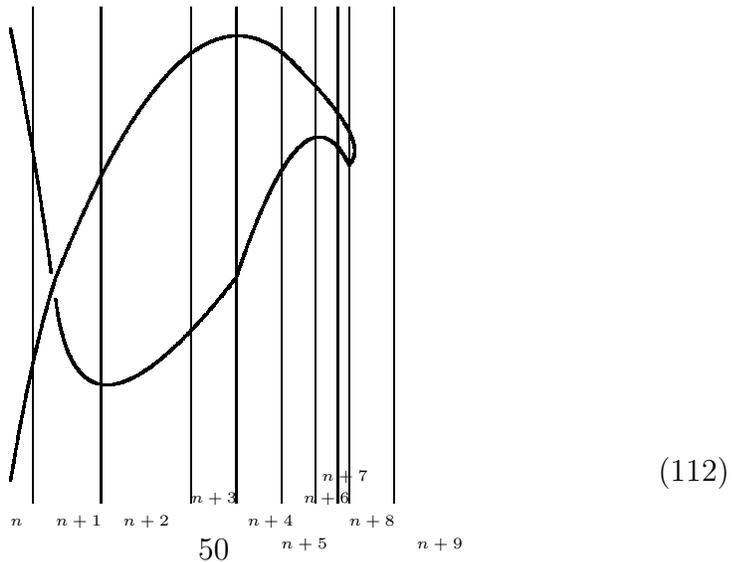
\beq
\setlength{\unitlength}{0.3cm}
\begin{picture}(18,22)
\put(-1,-1){\line(0,1){22}}
\put(1,-1){\line(0,1){22}}
\put(1.5,-1){\line(0,1){22}}
\put(2.5,-1){\line(0,1){22}}
\put(4,-1){\line(0,1){22}}
\put(6,-1){\line(0,1){22}}
\put(8,-1){\line(0,1){22}}
\put(9.8,-1){\line(0,1){22}}
\put(14,-1){\line(0,1){22}}
\put(-2,-3){\text{\tiny $n$}}
\put(-1,-2){\text{\tiny $n+1$}}
\put(0.5,-3){\text{\tiny $n+2$}}
\put(1.5,-2){\text{\tiny $n+3$}}
\put(2.5,-1){\text{\tiny $n+4$}}
\put(4.5,-2){\text{\tiny $n+5$}}
\put(6.5,-3){\text{\tiny $n+6$}}
\put(8,-2){\text{\tiny $n+7$}}
\put(11,-2){\text{\tiny $n+8$}}
\put(15,-2){\text{\tiny $n+9$}}
\thicklines
\qbezier(16,0)(15,6)(14,9)
\qbezier(13,10)(8,24)(3,18)
\qbezier(3,18)(0,15)(1,14)
\qbezier(1,14)(3,18)(6,9)
\qbezier(6,9)(13,0)(16,20)
\end{picture} \label{lbird}
\eeq
\\
\\
and:
\beq
\setlength{\unitlength}{0.3cm}
\begin{picture}(18,22)
\put(17,-1){\line(0,1){22}}
\put(15,-1){\line(0,1){22}}
\put(14.5,-1){\line(0,1){22}}
\put(13.5,-1){\line(0,1){22}}
\put(12,-1){\line(0,1){22}}
\put(10,-1){\line(0,1){22}}
\put(8,-1){\line(0,1){22}}
\put(4,-1){\line(0,1){22}}
\put(1,-1){\line(0,1){22}}
\put(18,-3){\text{\tiny $n+9$}}
\put(15,-2){\text{\tiny $n+8$}}
\put(13.8,0){\text{\tiny $n+7$}}
\put(13,-1){\text{\tiny $n+6$}}
\put(12,-3){\text{\tiny $n+5$}}
\put(10.5,-2){\text{\tiny $n+4$}}
\put(8,-1){\text{\tiny $n+3$}}
\put(5,-2){\text{\tiny $n+2$}}
\put(2,-2){\text{\tiny $n+1$}}
\put(0,-2){\text{\tiny $n$}}
\thicklines
\qbezier(0,0)(1,6)(2,9)
\qbezier(2,9)(8,24)(13,18)
\qbezier(13,18)(16,15)(15,14)
\qbezier(15,14)(13,18)(10,9)
\qbezier(10,9)(3,0)(2,8)
\qbezier(1.8,9.2)(1,15)(0,20)
\end{picture} \label{rbird}
\eeq
\\
\\
\begin{Iso4}
With tangle diagrams as above, the map $\mathbf{h}_{\Omega.1.f}$ that acts on $\widetilde{Z}_f(L)$ in book notation to effect the Reidemeister move $\Omega.1.f$ on the $l$-th component of $L$, is given on a page by:
\beq
\mathbf{h}_{\Omega.1.f}A:=M_{\Omega.1.f}^T A M_{\Omega.1.f}
\eeq
where $M_{\Omega.1.f}$ is the $qN \times qN$ identity matrix save for the $ll$ block which is of the form:
\beq
\setlength{\unitlength}{0.4cm}
\begin{picture}(14,16)
\put(0,9){\line(1,0){3}}
\put(3,9){\line(0,1){3}}
\put(7,0){\line(0,1){5}}
\put(7,5){\line(1,0){5}}
\put(-2,8.2){\text{\tiny n}}
\put(-2,7.2){$\cdot$}
\put(-2,6.7){$\cdot$}
\put(-2,6.2){$\cdot$}
\put(-3,5.2){\text{\tiny n+9}}
\put(3.2,12.5){\text{\tiny n}}
\put(4.2,12.5){$\cdot$}
\put(4.7,12.5){$\cdot$}
\put(5.2,12.5){$\cdot$}
\put(6.2,12.5){\text{\tiny n+9}}
\put(5.2,7.2){$\cdot$}
\put(4.2,6.2){$\cdot$}
\put(4.7,6.7){$\cdot$}
\put(3.2,5.2){$1$}
\put(6.2,8.2){$1$}
\thicklines
\put(0,6){\oval(2,12)[l]}
\put(12,6){\oval(2,12)[r]}
\put(1,10){$I_{n-1}$}
\put(8,3){$I_{N-n-11}$}
\end{picture}
\eeq
\\
\end{Iso4}
\begin{proof}
With the vertical slicing as in \eqref{lbird} and \eqref{rbird}, it is immediate that the above matrix implements the move $\Omega.1.f$ on the $l$-th component. Chords with feet on other strands are in strips indexed $1$ through $n$ or $n+3$ through $N$ and stay put under $\Omega.1.f$, so for those we just use an identity matrix. We conclude that the above matrix implements the move $\Omega.1.f$.
\end{proof}

\subsection{Adding and deleting strips}
If $M_n^+$ denotes the matrix that inserts a strip between the $n$-th and $n+1$-st strip, then the transformation on pages that effects such a change is the $qN \times q(N+1)$ matrix $\text{diag}(M_n^+)$ with:
\beq
\setlength{\unitlength}{0.3cm}
\begin{picture}(11,10)
\put(0,5){\line(1,0){3}}
\put(3,5){\line(0,1){3}}
\put(4,5){\line(1,0){5}}
\put(4,0){\line(0,1){5}}
\put(3.2,9){$\text{\tiny n+1}$}
\put(-2.8,4.2){$\text{\tiny n+1}$}
\put(3.2,7){$\text{\small 0}$}
\put(3.2,5.2){$\cdot$}
\put(3.2,4.2){$\cdot$}
\put(3.2,3.2){$\cdot$}
\put(3.2,0){$\text{\small 0}$}
\put(-8,4){$\text{\Large M}_n^+ =$}
\thicklines
\put(0,4){\oval(2,8)[l]}
\put(9,4){\oval(2,8)[r]}
\put(1,6){$I_n$}
\put(5,2){$I_{N-n}$}
\end{picture}
\eeq
\\
and if $M_n^-$ denotes the matrix that deletes the $n$-th strip, then the transformation on pages that effects such a change is the $q(N+1) \times qN$ matrix $\text{diag}(M_n^-)$ with:
\beq
\setlength{\unitlength}{0.3cm}
\begin{picture}(11,11)
\put(0,5){\line(1,0){3}}
\put(3,5){\line(0,1){3}}
\put(2,4.9){\line(1,0){5}}
\put(2,0){\line(0,1){4.9}}
\put(2,9){$\text{\tiny n-1}$}
\put(-7,4){$\text{\Large M}_n^- =$}
\thicklines
\put(0,4){\oval(2,8)[l]}
\put(9,4){\oval(2,8)[r]}
\put(0,6){$I_{n-1}$}
\put(3,2){$I_{N-n+2}$}
\end{picture}
\eeq
\\

\section{Behavior of $\widetilde{Z}_f(L)$ in book notation under a change of orientation}
\subsection{Behavior of $\widetilde{Z}_f$ under orientation change}
Recall that the links we deal with are oriented. Under a change of orientation on a component of a link $L$, any chord diagram with a foot on that component has its coefficient being multiplied by $-1$. The map on chord diagrams in $\hatA(\coprod S^1)$ that effects this change is denoted by $S_r$ for the $r$-th component of the link $L$ whose orientation is being changed, $1 \leq r \leq q$ ~\cite{LM3}. We generalize this to tangle chord diagrams: $S_r$ is the map $\hat{A}(\coprod_{1 \leq l \leq q}K_l) \rightarrow \hat{A}(\coprod_{1 \leq l\neq r \leq q}K_l + S_r K_r)$ induced from a change of orientation $S_rK_r$ on $K_r$. The book notation should reflect this change. Pictorially, the following local tangle chord diagram:
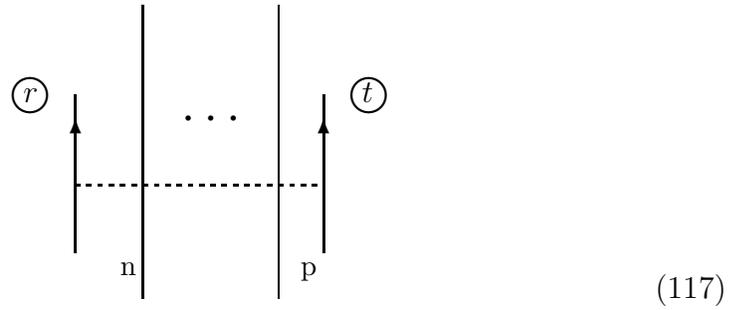
\beq
\setlength{\unitlength}{0.3cm}
\begin{picture}(18,13)
\put(6,0){\line(0,1){13}}
\put(12,0){\line(0,1){13}}
\thicklines
\put(3,2){\line(0,1){7}}
\put(3,7){\vector(0,1){1}}
\put(14,2){\line(0,1){7}}
\put(14,7){\vector(0,1){1}}
\multiput(3,5)(0.5,0){22}{\line(1,0){0.2}}
\multiput(8,8)(1,0){3}{\circle*{0.2}}
\put(5,1){$\text{\small n}$}
\put(13,1){$\text{\small p}$}
\put(1,9){\circle{1.5}}
\put(0.7,8.7){$r$}
\put(16,9){\circle{1.5}}
\put(15.7,8.7){$t$}
\end{picture}
\eeq
\\
is represented in book notation by:
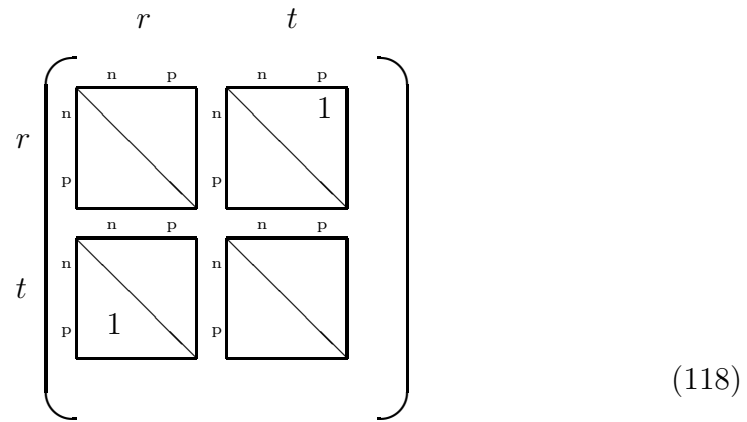
\beq
\setlength{\unitlength}{0.4cm}
\begin{picture}(14,13)
\put(2,10){\line(1,-1){4}}
\put(7,5){\line(1,-1){4}}
\put(7,10){\line(1,-1){4}}
\put(2,5){\line(1,-1){4}}
\thicklines
\put(2,5){\oval(2,12)[l]}
\put(12,5){\oval(2,12)[r]}
\put(0,3){$t$}
\put(0,8){$r$}
\put(4,12){$r$}
\put(9,12){$t$}
\put(7,1){\line(0,1){4}}
\put(7,1){\line(1,0){4}}
\put(7,5){\line(1,0){4}}
\put(11,1){\line(0,1){4}}
\put(6.5,1.8){$\text{\tiny p}$}
\put(6.5,4){$\text{\tiny n}$}
\put(10,5.3){$\text{\tiny p}$}
\put(8,5.3){$\text{\tiny n}$}
\put(2,1){\line(0,1){4}}
\put(2,1){\line(1,0){4}}
\put(2,5){\line(1,0){4}}
\put(6,1){\line(0,1){4}}
\put(3,1.8){$1$}
\put(1.5,1.8){$\text{\tiny p}$}
\put(1.5,4){$\text{\tiny n}$}
\put(5,5.3){$\text{\tiny p}$}
\put(3,5.3){$\text{\tiny n}$}
\put(2,6){\line(0,1){4}}
\put(2,6){\line(1,0){4}}
\put(2,10){\line(1,0){4}}
\put(6,6){\line(0,1){4}}
\put(1.5,6.8){$\text{\tiny p}$}
\put(1.5,9){$\text{\tiny n}$}
\put(5,10.3){$\text{\tiny p}$}
\put(3,10.3){$\text{\tiny n}$}
\put(7,6){\line(0,1){4}}
\put(7,6){\line(1,0){4}}
\put(7,10){\line(1,0){4}}
\put(11,6){\line(0,1){4}}
\put(6.5,6.8){$\text{\tiny p}$}
\put(10,9){$1$}
\put(6.5,9){$\text{\tiny n}$}
\put(10,10.3){$\text{\tiny p}$}
\put(8,10.3){$\text{\tiny n}$}
\end{picture}
\eeq
\\
where without loss of generality we have chosen $r<t$.
The same tangle chord diagram with the reverse orientation on the $r$-th component:
\beq
\setlength{\unitlength}{0.3cm}
\begin{picture}(18,13)
\put(6,0){\line(0,1){13}}
\put(12,0){\line(0,1){13}}
\thicklines
\put(3,2){\line(0,1){7}}
\put(3,7){\vector(0,-1){1}}
\put(14,2){\line(0,1){7}}
\put(14,7){\vector(0,1){1}}
\multiput(3,5)(0.5,0){22}{\line(1,0){0.2}}
\multiput(8,8)(1,0){3}{\circle*{0.2}}
\put(5,1){$\text{\small n}$}
\put(13,1){$\text{\small p}$}
\put(1,9){\circle{1.5}}
\put(0.7,8.7){$r$}
\put(16,9){\circle{1.5}}
\put(15.7,8.7){$t$}
\end{picture}
\eeq
\\
has book representation:
\beq
\setlength{\unitlength}{0.4cm}
\begin{picture}(14,13)
\put(2,10){\line(1,-1){4}}
\put(7,5){\line(1,-1){4}}
\put(7,10){\line(1,-1){4}}
\put(2,5){\line(1,-1){4}}
\thicklines
\put(2,5){\oval(2,12)[l]}
\put(12,5){\oval(2,12)[r]}
\put(0,3){$t$}
\put(0,8){$r$}
\put(4,12){$r$}
\put(9,12){$t$}
\put(7,1){\line(0,1){4}}
\put(7,1){\line(1,0){4}}
\put(7,5){\line(1,0){4}}
\put(11,1){\line(0,1){4}}
\put(6.5,1.8){$\text{\tiny p}$}
\put(6.5,4){$\text{\tiny n}$}
\put(10,5.3){$\text{\tiny p}$}
\put(8,5.3){$\text{\tiny n}$}
\put(2,1){\line(0,1){4}}
\put(2,1){\line(1,0){4}}
\put(2,5){\line(1,0){4}}
\put(6,1){\line(0,1){4}}
\put(2.5,1.8){$-1$}
\put(1.5,1.8){$\text{\tiny p}$}
\put(1.5,4){$\text{\tiny n}$}
\put(5,5.3){$\text{\tiny p}$}
\put(3,5.3){$\text{\tiny n}$}
\put(2,6){\line(0,1){4}}
\put(2,6){\line(1,0){4}}
\put(2,10){\line(1,0){4}}
\put(6,6){\line(0,1){4}}
\put(1.5,6.8){$\text{\tiny p}$}
\put(1.5,9){$\text{\tiny n}$}
\put(5,10.3){$\text{\tiny p}$}
\put(3,10.3){$\text{\tiny n}$}
\put(7,6){\line(0,1){4}}
\put(7,6){\line(1,0){4}}
\put(7,10){\line(1,0){4}}
\put(11,6){\line(0,1){4}}
\put(6.5,6.8){$\text{\tiny p}$}
\put(9.5,9){$-1$}
\put(6.5,9){$\text{\tiny n}$}
\put(10,10.3){$\text{\tiny p}$}
\put(8,10.3){$\text{\tiny n}$}
\end{picture}
\eeq
\\
\begin{Or}
With tangle diagrams as above, the map $\mathbf{h}_{S,r}$ that acts on $\widetilde{Z}_f(L)$ in book notation to effect the orientation change $S_r$ on the $r$-th component of $L$, is given on a page by:
\beq
\mathbf{h}_{S,r}A:=M_{S,r}^T A M_{S,r}
\eeq
where $M_{S,r}$ is the $qN \times qN$ identity matrix save for the $rr$ block which is the negative of the $N \times N$ identity matrix, $-I_N$:
\beq
\setlength{\unitlength}{0.4cm}
\begin{picture}(14,14)
\put(7,5){$1$}
\put(8,4){\circle*{0.2}}
\put(9,3){\circle*{0.2}}
\put(10,2){\circle*{0.2}}
\put(11,1){$1$}
\thicklines
\put(2,5){\oval(2,12)[l]}
\put(12,5){\oval(2,12)[r]}
\put(0,3){$t$}
\put(0,8){$r$}
\put(4,12){$r$}
\put(9,12){$t$}
\put(2,6){\line(0,1){4}}
\put(2,6){\line(1,0){4}}
\put(2,10){\line(1,0){4}}
\put(6,6){\line(0,1){4}}
\put(3,8){$-\text{\Large I}_N$}
\end{picture}
\eeq
\\
There are as many modified blocks such as the $rr$-th above as there are components whose orientation is changed. For a chord with feet on the $r$-th and $t$-th components both of whose orientations are reversed, the matrix $M_{S,r,t}$ with two blocks $rr$ and $tt$ as $-I_N$ will effect the orientation change while leaving the coefficient in front of corresponding chord diagrams unchanged since $(-1)\times (-1)=1$.
\end{Or}
\begin{proof}
Matrix multiplication.
\end{proof}

\subsection{Behavior of $\widetilde{Z}_f(L)$ as $K_i$ is being subtracted from $K_j$}
We now consider the effect of having a subtraction of $K_i$ from $K_j$ as a result of operating a band sum move of $K_i$ over $K_j$. As defined in ~\cite{RKI}, this corresponds to having a band such that upon doing the band sum move the components $K_i$ and $K_j$ end up having opposite orientations. If we locally represent the pieces of those two components on which the foot of a given chord rests as in:
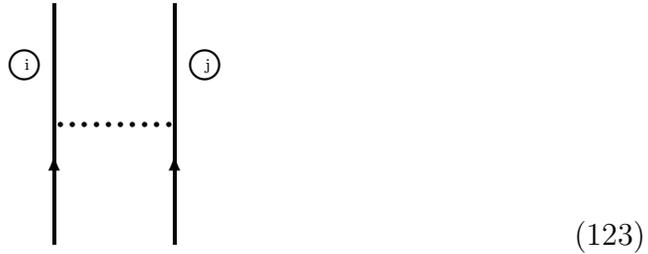
\beq
\setlength{\unitlength}{0.4cm}
\begin{picture}(5,10)
\thicklines
\put(0,0){\line(0,1){8}}
\put(4,0){\line(0,1){8}}
\multiput(0.2,4)(0.4,0){10}{\circle*{0.2}}
\put(0,2){\vector(0,1){1}}
\put(4,2){\vector(0,1){1}}
\put(-1,6){\circle{1}}
\put(-1,5.8){$\text{\tiny i}$}
\put(5,6){\circle{1}}
\put(5,5.8){$\text{\tiny j}$}
\end{picture}
\eeq
\\
then it behaves as follows under a band sum move of $K_i$ over $K_j$:
\beq
\setlength{\unitlength}{0.4cm}
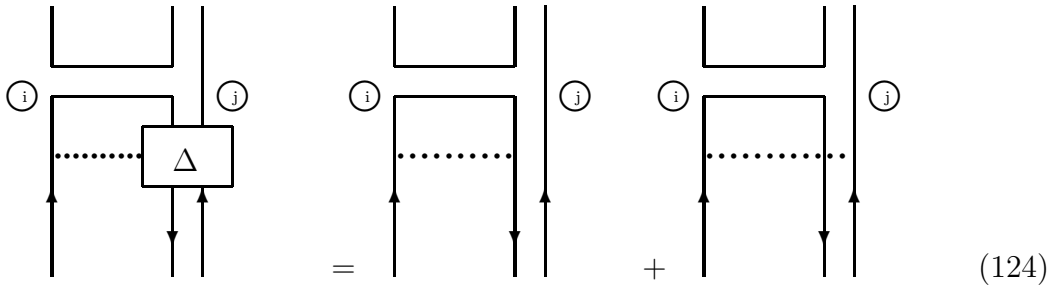
\begin{picture}(9,10)
\thicklines
\put(0,0){\line(0,1){6}}
\put(0,7){\line(0,1){2}}
\put(0,6){\line(1,0){4}}
\put(0,7){\line(1,0){4}}
\put(4,0){\line(0,1){3}}
\put(4,5){\line(0,1){1}}
\put(4,7){\line(0,1){2}}
\put(5,0){\line(0,1){3}}
\put(5,5){\line(0,1){4}}
\multiput(0.2,4)(0.3,0){10}{\circle*{0.2}}
\put(3,3){\line(1,0){3}}
\put(3,3){\line(0,1){2}}
\put(6,3){\line(0,1){2}}
\put(3,5){\line(1,0){3}}
\put(4,3.5){$\Delta$}
\put(0,2){\vector(0,1){1}}
\put(4,2){\vector(0,-1){1}}
\put(5,2){\vector(0,1){1}}
\put(-1,6){\circle{1}}
\put(-1,5.8){$\text{\tiny i}$}
\put(6,6){\circle{1}}
\put(6,5.8){$\text{\tiny j}$}
\end{picture}=
\setlength{\unitlength}{0.4cm}
\begin{picture}(9,10)(-1,0)
\thicklines
\put(0,0){\line(0,1){6}}
\put(0,7){\line(0,1){2}}
\put(0,6){\line(1,0){4}}
\put(0,7){\line(1,0){4}}
\put(4,0){\line(0,1){6}}
\put(4,7){\line(0,1){2}}
\put(5,0){\line(0,1){9}}
\put(5,5){\line(0,1){4}}
\multiput(0.2,4)(0.4,0){10}{\circle*{0.2}}
\put(0,2){\vector(0,1){1}}
\put(4,2){\vector(0,-1){1}}
\put(5,2){\vector(0,1){1}}
\put(-1,6){\circle{1}}
\put(-1,5.8){$\text{\tiny i}$}
\put(6,6){\circle{1}}
\put(6,5.8){$\text{\tiny j}$}
\end{picture}+
\setlength{\unitlength}{0.4cm}
\begin{picture}(9,10)(-1,0)
\thicklines
\put(0,0){\line(0,1){6}}
\put(0,7){\line(0,1){2}}
\put(0,6){\line(1,0){4}}
\put(0,7){\line(1,0){4}}
\put(4,0){\line(0,1){6}}
\put(4,5){\line(0,1){1}}
\put(4,7){\line(0,1){2}}
\put(5,0){\line(0,1){9}}
\put(5,5){\line(0,1){4}}
\multiput(0.2,4)(0.4,0){12}{\circle*{0.2}}
\put(0,2){\vector(0,1){1}}
\put(4,2){\vector(0,-1){1}}
\put(5,2){\vector(0,1){1}}
\put(-1,6){\circle{1}}
\put(-1,5.8){$\text{\tiny i}$}
\put(6,6){\circle{1}}
\put(6,5.8){$\text{\tiny j}$}
\end{picture}
\eeq
\\
If the latter tangle chord diagram has a coefficient of $c_X$ in the expression for $\widetilde{Z}_f$, the former has a coefficient $-c_X$ however due to the orientation change resulting in the Kontsevich integral picking up an overall minus sign. It follows that the statement of Theorem 1.5 is no longer true in this case. We can nevertheless remedy this as follows. Using the map $S_{(C)}$ introduced in ~\eqref{SC}, we denote by $S_{K_{i/j}}$ the map that switches the orientation on $K_i$ (or $K_j$) and we write $L''=S_{K_{i/j}}L$. Theorem 4 of ~\cite{LM4} states that in this case we have:
\beq
\hat{Z}_f(L'')=S_{K_{i/j}}\hat{Z}_f(L)
\eeq
However as pointed out in \cite{RG} and \cite{RG2}, their invariant $\hat{Z}_f$ is not the framed Kontsevich integral. Nevertheless, using the construction of the framed Kontsevich integral based on $Z_f$ defined in \cite{LM1}, we have:
\begin{ZfSLeqSZfL}
$Z_f(S_{K_i}L)=S_{K_i}Z_f(L)$ for $1 \leq i \leq q$.
\end{ZfSLeqSZfL}
\begin{proof}
This follows from the definition of $Z_f$ and the fact that $Z(S_{K_i}L)=S_{K_i}Z(L)$.
\end{proof}
\begin{ZfhSLeqSZfhL}
$\hat{Z}_f(L'')=S_{K_i}\hat{Z}_f(L)$
\end{ZfhSLeqSZfhL}
\begin{proof}
It suffices to write:
\begin{align}
\hat{Z}_f(S_{K_i}L)&=\nu^{m_1} \otimes \cdots \otimes \nu ^{m_i} \otimes \cdots \otimes \nu ^{m_q} \cdot Z_f(S_{K_i}L) \\
&=\nu^{m_1} \otimes \cdots \otimes S_{K_i}S_{K_i}\nu ^{m_i} \otimes \cdots \otimes \nu ^{m_q} \cdot S_{K_i}Z_f(L) \\
&=S_{K_i} \Big( \nu^{m_1} \otimes \cdots \otimes (S_{K_i}\nu) ^{m_i} \otimes \cdots \otimes \nu ^{m_q} \cdot Z_f(L)\Big) \\
&=S_{K_i} \hat{Z}_f(L)
\end{align}
\end{proof}
\begin{ZftSLeqSZftL} \label{Zft}
$\widetilde{Z}_f(L'')=S_{K_i}\widetilde{Z}_f(L)$
\end{ZftSLeqSZftL}
\begin{proof}
Same as for the previous Theorem.
\end{proof}
\newpage
Then we consider the following diagram:\\
\setlength{\unitlength}{0.5cm}
\begin{picture}(15,5)(-6,3)
\put(0,5){\line(0,-1){4}}
\put(0,2){\vector(0,-1){1}}
\put(4,0){\line(1,0){4}}
\put(7,0){\vector(1,0){1}}
\put(15,1){\line(0,1){4}}
\put(15,4){\vector(0,1){1}}
\multiput(2,6)(0.4,0){15}{\line(1,0){0.15}}
\put(7.8,6){\vector(1,0){0.15}}
\put(-6,0){$\widetilde{Z}_f(S_{K_{i/j}}L)=\sum c_X \cdot X $}
\put(-2,6){$\widetilde{Z}_f(L)$}
\put(9,0){$\sum c_X \cdot X'=\widetilde{Z}_f([S_{K_{i/j}}L]')$}
\put(9,6){$\widetilde{Z}_f(L')=\widetilde{Z}_f(S_{K_{i/j}}[S_{K_{i/j}}L]')$}
\put(-2,3){$S_{K_{i/j}}$}
\put(15.5,3){$S_{K_{i/j}}$}
\put(4,-1){$\text{\tiny band sum move}$}
\end{picture}\\
\\
\\
\\
\\
\\
First applying the map $S_{K_{i/j}}$ to $\widetilde{Z}_f(L)$ we obtain $\widetilde{Z}_f(S_{K_{i/j}}L)$ by Theorem \ref{Zft}. We write this quantity as $\sum c_X \cdot X$. We can apply Proposition 1.5 for $S_{K_{i/j}}L$ since then the band sum move results in $K_i$ being added to $K_j$. We obtain $\sum c_X \cdot X'$ under the band sum move. This equals $\widetilde{Z}_f([S_{K_{i/j}}L]')$. By further reversing the orientation of $K_i$ (or $K_j$), we get, using Theorem \ref{Zft} again, $\widetilde{Z}_f(S_{K_{i/j}}[S_{K_{i/j}}L]')$. Since $S_{K_{i/j}}[S_{K_{i/j}}L]'=L'$, this last quantity is $\widetilde{Z}_f(L')$. We have that the band sum move in the event of a subtraction is given by closing the above diagram, and this corresponds to the composition of $S_{K_{i/j}}$, a band sum, and $S_{K_{i/j}}$.\\

\section{Recovering $Z_f(L) \in \overline{\mathcal{A}}(\amalg S^1)$ from $Z_f(L)$ in book notation}

We prove that once we have the expression $Z_f(L)$ written in book notation, then we can recover $Z_f(L)$ as an element of $\overline{\mathcal{A}}(\amalg S^1)$. While doing so, we prove that $Z_f$ in book notation is a faithful map on framed links, that is given an expression $Z_f(L)$ for some link $L$ yet to be determined, we can unambiguously determine what the link $L$ is.\\

A first important result that will have far reaching consequences follows from a stronger result than the following Proposition:
\begin{Zfta} (\cite{RG})
If $K_i$ and $K_j$ are unlinked:
\beq
\hat{Z}_f(K_i \amalg K_j)=\hat{Z}_f(K_i) + \hat{Z}_f(K_j)
\eeq
\end{Zfta}
This statement was proved in \cite{RG} by using the isotopy invariance of $\hat{Z}_f$ to move components around. A stronger result can be proved: the result holds for the bare framed Kontsevich integral as well.
\begin{bareZfta} \label{bareZfta}
If $K_0$ is a knot unlinked from a link $L$, then:
\beq
Z_f(K_0 \amalg L)=Z_f(K_0)+Z_f(L)
\eeq
\end{bareZfta}
\begin{proof}
From the previous Proposition, we have:
\beq
\hat{Z}_f(K_0 \amalg L)=\hat{Z}_f(K_0) + \hat{Z}_f(L)
\eeq
or equivalently, writing $L=\amalg_{1 \leq i \leq q}K_i$, and $m_j$ denoting the number of local maxima of a component $K_j$, $0 \leq j \leq q$:
\begin{align}
\nu^{m_0} \otimes \nu^{m_1}\otimes & \cdots \otimes \nu^{m_q} \cdot Z_f(K_0 \amalg L) \nonumber \\
&=\nu^{m_0}Z_f(K_0) + \nu^{m_1}\otimes \cdots \otimes \nu^{m_q}Z_f(L)
\end{align}
Focusing on the first term, we can write:
\begin{align}
Z_f(K_0)&=\sum_{\substack{m \geq 0 \\ |P|=m}}c_P (K_0)_P \\
&=\sum_{\substack{m \geq 0 \\ |P|=m}}c_P (K_0)_P + \sum_{\substack{m \geq 0 \\ |P*|=m}}0 \cdot L_{P*}\\
&=\sum_{\substack{m \geq 0 \\ |P|=m}}d_P (K_0 \amalg L)_P
\end{align}
where $P^*$ are pairings that represent a tangle chord diagram with at least one chord with at least one foot on $L-K_0$, and $d_P=c_P$ if $P$ is a pairing all of whose chords lie on $K_0$, zero otherwise. This enables us to write $Z_f(K_0)$ as an element of $\overline{\mathcal{A}}(K_0 \amalg L)$ on which $\otimes_{0 \leq i \leq q}\nu^{m_i}$ can be applied:
\beq
\nu^{m_0}Z_f(K_0)=\nu^{m_0} \otimes \nu^{m_1}\otimes  \cdots \otimes \nu^{m_q}Z_f(K_0)
\eeq
viewed as an element of $\overline{\mathcal{A}}(K_0 \amalg L)$. Likewise:
\beq
\nu^{m_1}\otimes  \cdots \otimes \nu^{m_q}Z_f(L)=\nu^{m_0} \otimes \nu^{m_1}\otimes  \cdots \otimes \nu^{m_q}Z_f(L)
\eeq
viewed as an element of $\overline{\mathcal{A}}(K_0 \amalg L)$ by setting to zero all coefficients of $Z_f(L)$ for pairings not entirely lying on $L$ the way we did for $K_0$. It follows that:
\begin{align}
\nu^{m_0} \otimes \nu^{m_1}\otimes & \cdots \otimes \nu^{m_q}Z_f(K_0 \amalg L) \nonumber \\
=\nu^{m_0} \otimes \nu^{m_1}\otimes & \cdots \otimes \nu^{m_q}Z_f(K_0) + \nu^{m_0} \otimes \nu^{m_1}\otimes  \cdots \otimes \nu^{m_q}Z_f(L)
\end{align}
from which it follows that:
\beq
Z_f(K_0 \amalg L)=Z_f(K_0) + Z_f(L)
\eeq
\end{proof}
\begin{generalbareZfta} \label{generalbareZfta}
If $I$ indexes sublinks of a link $L$ that are pairwise unlinked, then denoting by $\{\text{sub}_i \;|\; i \in I \}$ the set of such sublinks, we can write:
\beq
Z_f(L)=\sum_{i \in I} Z_f(\amalg_{j \in \text{sub}_i}K_j)
\eeq
\end{generalbareZfta}
An immediate consequence of this corollary is that books for links such as the one above will have blocks that will always be empty, being an indication that some components are unlinked. We will use Proposition \ref{bareZfta} later to determine where crossings are on a link $L$. It is important to point out that if we are given $Z_f(L)$ for some yet unknown link $L$, we do not know a priori how many components this link $L$ has. Suppose the matrices for the books of $Z_f(L)$ have size $R=qN$ for some $q$ yet to be determined. To make the concept of books less rigid, we introduce an algebra $\mathcal{B}$ of books, $\mathcal{B}[R]=\oplus_{m \geq 0}\Big( \text{Mat}_{R \times R} \Big)^{\otimes m}$ being a subalgebra thereof, with completion $\overline{\mathcal{B}}[R]=\prod_{m \geq 0}\Big( \text{Mat}_{R \times R} \Big)^{\otimes m}$. Let $\overline{\mathcal{B}}$ be the completion of $\mathcal{B}$. One would like to have maps $\overline{\mathcal{B}}[R] \rightarrow \overline{\mathcal{B}}[R+S]$ for certain values of $S$. This is implemented by adding or deleting strips and the matrices $M_n^{\pm}$ do just that. Once we have such matrices, it is immediate what the value of $N$ is, from which we can find $q=R/N$. In particular in the above map we would have $S=q$ since we add one strip for each of the $q$ components, yet another way to determine $q$. The complete picture is given by a sequence of algebras of books along with embeddings that corresponds to $\overline{\mathcal{B}}$:
\beq
\overline{\mathcal{B}}[1] \rightarrow \overline{\mathcal{B}}[2] \rightarrow \cdots \rightarrow \overline{\mathcal{B}}[n] \rightarrow \cdots \nonumber
\eeq
An element such as $Z_f(L) \in \overline{\mathcal{B}}[R]$ will only map to $\overline{\mathcal{B}}[R+4mq+mN+4m^2]$ for $m \geq 0$. Indeed, adding a circle to $Z_f(L) \in \overline{\mathcal{A}}(\amalg^q S^1)$ will yield $Z_f(L)\amalg S^1 \in \overline{\mathcal{A}}(\amalg^{q+1}S^1)$, which does not contain any more information than what we already had with the original $Z_f(L) \in \overline{\mathcal{A}}(\amalg^q S^1)$. Further $Z_f(S^1)=S^1$ is trivial. Using Proposition \ref{bareZfta}, we can write:
\beq
Z_f(L)\amalg S^1=Z_f(L) \amalg Z_f(S^1)=Z_f(L \amalg S^1)
\eeq
If we now consider $Z_f(L\amalg S^1)$ in book notation, inserting an $S^1$ to the left of $L$ adds an additional 4 strips. Thus $Z_f(L\amalg S^1) \in \overline{\mathcal{B}}[R+4q+N+4]$. Adding another circle, we get $Z_f(L\amalg S^1 \amalg S^1) \in \overline{\mathcal{B}}[R+8q+2N+16]$. Continuing in this fashion determines what we call a thread for $Z_f(L)$:
\beq
\overline{\mathcal{B}}[R] \rightarrow \overline{\mathcal{B}}[R+4q+N+4] \rightarrow \cdots \rightarrow \overline{\mathcal{B}}[R+4mq+mN+4m^2] \rightarrow \cdots \nonumber
\eeq
a representative being given by $Z_f(L) \in \overline{\mathcal{B}}[R]$. Further adding a circle at each step does not modifiy the value of $Z_f(L)$. $L$ being fixed, the additional circles are not fixed however and each $Z_f(L \amalg ^m S^1)$ will be a functional of the position of the $m$ additional circles. An upshot of this formalism is that a thread immediately tells us how many components the link $L$ has, since upon adding trivial circles, matrices have an empty block for each circle, and non trivial blocks being pushed by $m(N+4m)$. Therefore we know that ultimately $Z_f(L) \in \overline{\mathcal{A}}(\amalg ^q S^1)$ after projection.\\

In a first time, if we want to recover $Z_f(L)$ in terms of chord diagrams, we essentially need to locate all the local extrema of the link $L$ as well as where its crossings are. From this knowledge, we can determine what chord diagram each book is corresponding to, and thus we have a map from $Z_f(L)$ in book notation to its projection in the completed algebra of chord diagrams. We first determine where the crossings are. We start by first recalling some particularity of the Kontsevich integral of crossings that will allow us to detect them in the expression for $Z_f(L)$.
\begin{Xpm} \label{Xpm}
For all $m \geq 0$, the coefficient of the degree $m$ part of the Kontsevich integral of a crossing is a sum of terms, one of which is $\pm 1/(m! 2^m)$.
\end{Xpm}
\begin{proof}
For a crossing of the form:
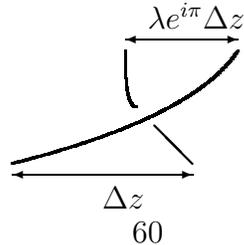
\beq
\setlength{\unitlength}{0.3cm}
\begin{picture}(12,8)
\put(0,-0.5){\vector(1,0){8}}
\put(8,-0.5){\vector(-1,0){8}}
\put(4,-2){$\Delta z$}
\put(5,5.5){\vector(1,0){5}}
\put(10,5.5){\vector(-1,0){5}}
\put(6,6){$\lambda e^{i \pi} \Delta z$}
\thicklines
\qbezier(0,0)(8,2)(10,5)
\qbezier(5.5,2.5)(5,2.5)(5,5)
\put(8,0){\line(-1,1){1.7}}
\end{picture} \nonumber
\eeq
\\
where $\lambda$ is some non zero real number and $\Delta z$ is the separation between the two strands at time $t=0$, then denoting such a crossing by $X^+$, we have:
\beq
Z(X^+)=\sum_{m \geq 0}Z_m(X^+)
\eeq
with:
\beq
Z_m(X^+)=\frac{1}{(2 \pi i)^m}\int_{t_1 < \cdots < t_m}(\pm 1)^m X^+_m \prod_{1 \leq i \leq m} \dlog \vartriangle \!\! z(t_i)
\eeq
where the plus sign corresponds to the two strands having the same orientation, the minus the case of opposite orientations, and $X^+ _m$ is the tangle chord diagram with skeleton $X^+$ with $m$ chords on it.\\
This equals:
\begin{align}
Z_m(X^+)&=\frac{(\pm 1)^m}{(2 \pi i)^m} X^+_m \int_{t_2 < \cdots < t_m}\dlog \vartriangle \!\!z(t_2)  \cdots \dlog \vartriangle \!\!z(t_m) \int_{t_1 <t_2} \dlog \vartriangle \!\!z(t_1) \nonumber \\
&=\frac{(\pm 1)^m}{(2 \pi i)^m} X^+_m \int_{t_2 < \cdots < t_m}\dlog \vartriangle \!\!z(t_2) \cdots \dlog \vartriangle \!\!z(t_m) \log(\frac{\vartriangle \!\!z(t_2)}{\vartriangle \!\!z}) \nonumber\\
&=\frac{(\pm 1)^m}{(2 \pi i)^m} X^+_m \int_{t_3 < \cdots < t_m}\dlog \vartriangle \!\!z(t_3) \cdots \dlog \vartriangle \!\!z(t_m) \int_{t_2 <t_3} \dlog \vartriangle \!\!z(t_2) \log(\frac{\vartriangle \!\!z(t_2)}{\vartriangle \!\!z}) \nonumber \\
&=\frac{(\pm 1)^m}{(2 \pi i)^m} X^+_m \int_{t_3 < \cdots < t_m}\dlog \vartriangle \!\!z(t_3) \cdots \dlog \vartriangle \!\!z(t_m) \int_{t_2 <t_3} \dlog \frac{\vartriangle \!\!z(t_2)}{\vartriangle \!\!z} \log(\frac{\vartriangle \!\!z(t_2)}{\vartriangle \!\!z}) \nonumber \\
&=\frac{(\pm 1)^m}{(2 \pi i)^m} X^+_m \int_{t_3 < \cdots < t_m}\dlog \vartriangle \!\!z(t_3) \cdots \dlog \vartriangle \!\!z(t_m) \frac{1}{2}\log^2(\frac{\vartriangle \!\!z(t_3)}{\vartriangle \!\!z}) \nonumber \\
&= \cdots = \frac{(\pm 1)^m}{(2 \pi i)^m} X^+_m \frac{1}{m!}\log ^m (\frac{\lambda e^{i \pi}\vartriangle \!\!z}{\vartriangle \!\! z})\nonumber \\
&=\frac{1}{m!}\frac{(\pm 1)^m}{(2 \pi i)^m} X^+_m \Big( log \lambda + i \pi \Big)^m \nonumber \\
&=\frac{1}{m!}\frac{(\pm 1)^m}{(2 \pi i)^m} X^+_m \Big( (i \pi)^m + \cdots \Big) \nonumber\\
&=\frac{1}{m!}\frac{(\pm 1)^m}{2^m}X^+ _m + \cdots \nonumber
\end{align}
For a crossing of the form:
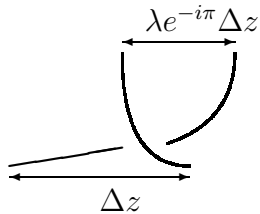
\beq
\setlength{\unitlength}{0.3cm}
\begin{picture}(12,8)
\put(0,-0.5){\vector(1,0){8}}
\put(8,-0.5){\vector(-1,0){8}}
\put(4,-2){$\Delta z$}
\put(5,5.5){\vector(1,0){5}}
\put(10,5.5){\vector(-1,0){5}}
\put(6,6){$\lambda e^{-i \pi} \Delta z$}
\thicklines
\qbezier(8,0)(5,0)(5,5)
\put(0,0){\line(6,1){5}}
\qbezier(7,1)(10,2)(10,5)
\end{picture} \nonumber
\eeq
\\
that we denote $X^-$ but otherwise with the same notations as above, then we would find, for all $m \geq 0$:
\beq
Z_m(X^-)=\frac{1}{m!}\frac{(\mp 1)^m}{2^m}X^+ _m + \cdots
\eeq
\end{proof}
\begin{!book} \label{!book}
For any given link $L$, there is a unique book in the expression for $Z_f(L)$ characterized by being the thinnest book whose coefficient is a sum of terms, one of which is real and is proportional to $\pm 1/2^M$, $M$ being the number of pages of that book, no two subsequent pages of which are identical, and such that we cannot add a page to that book different from its neighboring pages at the place of insertion to yield a book with coefficient a sum of terms, one of which is real and is proportional to $1/2^{M+1}$. This book represents all groups of crossings once, ordered from the top down, and $M-1$ is the number of changes in the crossing types of $L$.
\end{!book}
\begin{proof}
Consider a link $L$, $P=(P_1,\cdots,P_m)$ a pairing of order $m$ giving rise to a tangle chord diagram $L_P$ of degree $m$, corresponding to the following integral in $Z_f(L)$:
\beq
\frac{1}{(2 \pi i)^m}
\int_{0 < t_1 < \cdots < t_m <1}(-1)^{\varepsilon(P)}\dlog \vartriangle \!\! z[P_1](t_1) \cdots \dlog \vartriangle \!\! z[P_m](t_m) \label{kof}
\eeq
As $T=(t_1,\cdots , t_m)$ covers $\Delta_m$, the corresponding $m$ chords of $L_P$ located at times $t_1,\cdots,t_m$ along $L$ will sweep the part of the link $L$ covered by the $m$ chords. Suppose $L=T_1 \times \cdots T_l$ is the concatenation of $l$ elementary tangles $T_1,\cdots, T_l$, where in this proof only we consider that a group of crossings whose strands have the same winding constitute a single elementary tangle. By multiplicativity of $Z_f[\text{small Q}]$ for $Q>0$, we can write $Z_f(L)=Z_f[\text{\small Q}](L)=Z_f[\text{\small Q}](T_1) \times \cdots \times Z_f[\text{\small Q}](T_l)$. Then we can write \eqref{kof} as the coefficient of the part of the sum:
\beq
\sum_{m_1 + \cdots +m_l=m}Z_{f,m_1}[\text{\small Q}](T_1) \times \cdots \times Z_{f,m_l}[\text{\small Q}](T_l) \label{timedZ}
\eeq
for which:
\beq
\times_{1 \leq i \leq l}(T_i)_{|P'_i|=m_i}=L_P \label{cond}
\eeq
One should point out that writing \eqref{kof} as \eqref{timedZ} does not by any means indicate that \eqref{kof} was computed by using the multiplicativity of $Z_f[\text{\small Q}]$, but since \eqref{kof} is an iterated integral, its computation amounts to doing just \eqref{timedZ} in practice. By the preceding Lemma, contributions from crossings, be they positive or negative, always have coefficients a sum of terms one of which is real, and is not a power of $1/\pi$, which is not true of local extrema and associators in odd degrees. In even degrees however, those elementary tangles can have contributions to \eqref{kof} that are real. Consider the collection of summands of \eqref{timedZ} subject to \eqref{cond} that have coefficients a sum of terms, at least one of which is real. By the previous two remarks, each such summand represents a domain of integration in \eqref{kof} where either local extrema and associators are being swept by an even number of time variables, or by an odd number provided we simultaneously have an odd number of purely imaginary contributions from crossings to make the overall coefficient real. Consider the subcollection consisting of summands whose coefficients are a sum of terms, one of which is proportional to $1/2^m$. For $m \geq 2$, we could still get contributions from local extrema and/or associators. From this subcollection, eliminate all those summands such that upon adding a page to them different from neighboring pages at the place of insertion, we get a coefficient, a summand of which is real and is proportional to $1/2^{m+1}$. This indicates that such a summand represented a region of integration in \eqref{kof} missing a group of crossings. If all the summands have been eliminated, repeat the above procedure for a pairing $Q$, $|Q|>|P|$, the additional pairing containing components that were not present in $P$, until we have at least one summand left. Such a pairing $Q$ we call proper. Consider the collection of all summands surviving the elimination process for a proper pairing $Q$. Among the summands left there exists a thinnest book of $M$ pages with coefficient a summand of which is real and is proportional to $1/2^M$. No two of its pages are identical, and we cannot add a page to this book different from neighboring pages at the place of insertion to yield another book with coefficient with a summand that is real and is proportional to $1/2^{M+1}$. This book corresponds to a region of integration in \eqref{kof} where each group of crossings is being swept by only one time variable, and no associator or local extremum is being integrated over. This book is unique by construction.
\end{proof}
The preceding Proposition gives us the existence of such a book in the expression for $Z_f(L)$ for some given link $L$. It does not tell us how to recover such a book. We are now given $Z_f(L)$ and we are asked to determine where the crossings of $L$ are in a first time, or equivalently we have to find this book of $M$ pages. This we do by using the thread of $Z_f(L)$.
\begin{detect} \label{detect}
Using the first embedding of $Z_f(L)$ in its thread, we can determine where all the crossings of $L$ are, and whether they are positive or negative. Further we can determine within each group of crossings of a same type how many half-twists there are, and whether the involved strands wind clockwise or counterclockwise.
\end{detect}
\begin{proof}
Map $Z_f(L)$ in $\overline{\mathcal{B}}[qN]$ to its second element in its thread, that is $Z_f(L \amalg S^1)$, an element of $\overline{\mathcal{B}}[(q+1)(N+4)]$. For each chord stretching between the additional circle to the link $L$, we have a contribution $\dlog (z-z')$ to coefficients of the framed Kontsevich integral, where $z$ is a local coordinate on the circle, and $z'$ is a local coordinate on the link $L$. Since we consider equivalence classes of circles, $z$ is a variable. Now consider link components indexed by $i$ and $j$, $1 \leq i,j \leq q$ for which we have in degree one non zero contributions to $Z_f(L)$ coming from a chord stretching between a strand of the $i$-th component and a strand of the $j$-th component. Let $t_{min}$ (resp. $t_{max}$) be the time at which either of those two strands forms a local minimum (resp. a local maximum). Suppose those two strands are linked. Write $\varz(t)$ for the length of the chord between them as a function of time. We can write $\varz(t)=\lambda(t)\cdot \varz \cdot \exp(i \theta(t))$ where $\lambda$ is a real function with $\lambda(t_{min})=1$, $\varz=\varz(t_{min})$, and $\theta$ is the argument of $\varz(t)$. If we have some linking between the two strands, we have at least one group of crossings between them. Suppose there are $m$ such groups, and that the $k$-th group, $1 \leq k \leq m$ has $n_k$ half-twists and introduce a number $\eps_k=\pm 1$, a plus sign for counterclockwise rotation, a minus sign for a clockwise rotation. Let $\mu=\pm 1$ the relative orientation of the two strands, a plus sign indicating that both strands have the same orientation. Then we can define:
\beq
\theta(t)=\eps_1 n_1 \pi \frac{t-t_{min}}{t_1-t_{min}}
\eeq
if $t_{min}=t_0 \geq t \geq t_1$, and for $k \geq 2$:
\beq
\theta(t)=\sum_{1 \leq l < k}\eps_l n_l \pi + \eps_k n_k \pi \frac{(t-t_{k-1})}{(t_k - t_{k-1})}
\eeq
if $t_{k-1} \geq t \geq t_k$ for some times $t_1 ,\cdots , t_{m-1}, t_m=t_{max}$ in between each group of crossings. Now consider the degree two contribution to $Z_f(L \amalg S^1)$ coming from the tangle chord diagram having a chord stretching between the circle and one of the strands above, be it the $i$-th or the $j$-th, with length $\varz'(t)$ and relative orientation $\pm 1$, and right below the chord of length $\varz(t)$ stretching between the $i$-th and the $j$-th components. We are looking at:
\begin{align}
\pm \mu \frac{1}{(2 \pi i)^2}& \int_{t_{min}<t<t'<t_{max}} \dlog \varz'(t') \dlog \varz(t)\nonumber \\
&=\pm \mu\frac{1}{(2 \pi i)^2}\int_{t_{min}<t<t'<t_{max}} \dlog \varz'(t') \dlog (\lambda(t) \varz e^{i \theta(t)}) \\
&=\pm \mu\frac{1}{(2 \pi i)^2}\int_{t_{min}<t'<t_{max}} \dlog \varz'(t') \log \frac{\lambda(t') \varz e^{i \theta(t')}}{\varz}\\
&=\cdots  \pm \mu\frac{1}{(2 \pi)^2 i}\int_{t_{min}<t<t_{max}} \dlog \varz'(t) \cdot \theta(t)
\end{align}
where in the last line we have exhibited only the purely imaginary part of the iterated integral. If we denote by $I$ this summand, we can write:
\begin{align}
I&=\pm \mu\frac{1}{(2 \pi)^2 i}\sum_{1 \leq k \leq m} \int_{t_{k-1}}^{t_k} \dlog \varz'(t)\cdot \Big(\sum_{ 1 \leq l < k}\eps_l n_l \pi + \eps_k n_k \pi \frac{t-t_{k-1}}{t_k-t_{k-1}}\Big)\\
&=\pm \mu\frac{1}{(2 \pi)^2 i}\sum_{1 \leq k \leq m} \Big( \sum_{1 \leq l <k} \eps_l n_l \pi \Big)\log(\frac{\varz'(t_k)}{\varz'(t_{k-1})}) \nonumber \\
& \pm\mu\frac{1}{(2 \pi)^2 i}\sum_{ 1 \leq k \leq m} \eps_k n_k \pi \int_{t_{k-1}}^{t_k} \dlog \varz'(t) \cdot \frac{t-t_{k-1}}{t_k-t_{k-1}}\\
&=\pm \frac{1}{(2 \pi)^2 i}\sum_{1 \leq k \leq m} \Big[ \sum_{1 \leq l <k} \mu\eps_l n_l \pi \log(\frac{\varz'(t_k)}{\varz'(t_{k-1})})+ \mu\eps_k n_k \pi \log \varz'(t_k) \Big]\nonumber \\
&\mp\frac{1}{(2 \pi)^2 i}\sum_{1 \leq k \leq m} \frac{\mu\eps_k n_k \pi}{t_k -t_{k-1}}\int_{t_{k-1}}^{t_k} \log \varz'(t)dt
\end{align}
From this expression we have all the groups of crossings, their number of half-twists as well as windings up to the relative orientation of the $i$-th and $j$-th strands given by $\mu$. Keeping in mind that the overall $\pm$ sign in the front is the relative orientation of the circle and the portion of the link on which the top chord is ending, consider now the degree two contribution to $Z_f(L \amalg S^1)$ with the top chord stretching between the circle and the other component forming the group of crossings. The new $\varz'$ will be our old $\varz'(t) \pm \varz(t)$. Repeat the above computation. If the overall sign is $\mp \mu$, this is a sign that the two components indexed by $i$ and $j$ had opposite orientations, same orientations otherwise. We repeat this procedure for all pairs of components for which we have a non zero contribution to $Z_{f,1}(L)$. In the event that there is no linking between two strands on two components, there will simply be no imaginary part such as the one above.
\end{proof}

\begin{Rmkdetect}
By Lemma \ref{Xpm}:
\beq
Z(
\setlength{\unitlength}{0.3cm}
\begin{picture}(12,8)(-1,0)
\thicklines
\qbezier(0,0)(8,2)(10,5)
\put(8.8,4){\vector(2,1){1}}
\qbezier(5.5,2.5)(5,2.5)(5,5)
\put(5,5){\vector(0,-1){1}}
\put(8,0){\line(-1,1){1.7}}
\end{picture})=
Z(
\setlength{\unitlength}{0.3cm}
\begin{picture}(12,8)(-1,0)
\thicklines
\qbezier(8,0)(5,0)(5,5)
\put(5,4){\vector(0,1){1}}
\put(0,0){\line(6,1){5}}
\qbezier(7,1)(10,2)(10,5)
\put(10,4){\vector(0,1){1}}
\end{picture})
\eeq
\\
as well as:
\beq
Z(
\setlength{\unitlength}{0.3cm}
\begin{picture}(12,8)(-1,0)
\thicklines
\qbezier(0,0)(8,2)(10,5)
\put(8.8,4){\vector(2,1){1}}
\qbezier(5.5,2.5)(5,2.5)(5,5)
\put(5,4){\vector(0,1){1}}
\put(8,0){\line(-1,1){1.7}}
\end{picture})=
Z(
\setlength{\unitlength}{0.3cm}
\begin{picture}(12,8)(-1,0)
\thicklines
\qbezier(8,0)(5,0)(5,5)
\put(5,5){\vector(0,-1){1}}
\put(0,0){\line(6,1){5}}
\qbezier(7,1)(10,2)(10,5)
\put(10,4){\vector(0,1){1}}
\end{picture})
\eeq
\\
which means for example that a negative contribution to $Z_1$ of a crossing comes from a negative crossing, but we cannot determine whether it is two strands winding counterclockwise with opposite orientations, or two strands winding clockwise, with same orientations, unless we use threading, which immediately makes the distinction between one case and the other, as in Lemma \ref{detect} above. This is the first instance where threading is simplifying computations considerably.
\end{Rmkdetect}
We now seek to find local maxima, and then to position them relative to the crossings on $L$.
\begin{findlocmax} \label{findlocmax}
From $Z_f(L)$ we can determine the relative position of all local maxima of $L$.
\end{findlocmax}
\begin{proof}
Since chords are moved up so that at least one of their feet reaches a local maximum, chords both of whose feet are on the same local maximum are moved up to the extremity of that local maximum. Further each chord is located at some definite time, or height. Thus local maxima at different heights will have chords on their extremities at different times corresponding to those heights, and therefore we have corresponding books whose ordering of pages indicates the position of each local maximum relative to the others. For two local maxima at the same height, by non-simultaneity of chords, we will have two books, each one displaying the position of a chord on each of the two local maxima relative to the others.
\end{proof}
\begin{maxX} \label{maxX}
From $Z_f(L)$, the knowledge of the ordered sequence of local maxima from the preceding Lemma as well as the knowledge of the ordered sequence of crossings from Lemma \ref{detect}, we can position all crossings relative to the local maxima along $L$.
\end{maxX}
\begin{proof}
A chord has a pair of indices for each foot. One index for the component, the other for the strip. The component indices are clearly known from a book. The strip indices however are more difficult to work with. Thus when we talk about indices for a chord, we will mean the strip indices. Recall that chords are moved up until at least one of their feet reaches a local maximum, whenever possible. It is difficult however to associate indexes from chords emanating from groups of crossings to those of local maxima, since those latter can reorganize themselves. This leads us to first consider what we call ``shifting". Shifting means that prior to having a crossing, strands can reorganize themselves. At the level of elementary tangles, this corresponds to having a concatenation of associators and local extrema. We allow crossings on strands connected to a same local max. To determine the shifting, we go down the list of local maxima. Starting from the second highest local maximum, we determine all chords between this local max, and the highest local max. For one such chord, the strip index for the foot on a strand going up to the highest local max tells us in which strip that strand is located. In this manner we can tell the ``shifting" of the highest local max at the level of the second highest local max. Next we consider the third highest local max. Since we still assume that there is no known crossing other than simple crossings between strands connected to a same local max, by the same considerations we can determine the shifting of the highest and second highest local maxima at the level of the third highest local max. Proceeding in this fashion we can determine at each height at which a local max is situated prior to having crossings other than simple ones, the relative shifting of all other higher local maxima. Now going down the list of crossings, a chord in the first group of crossings has either one or two indices in common with local maxima indices, we can determine the location of that chord relative to higher local maxima using shifting. To illustrate this procedure, suppose such a chord has only one index in common with local maxima indexes. Using shifting we can determine what local maximum that strand whose index is unknown is connecting to. Suppose the chord index for the foot on the lowest local max is $n$, and the other foot index is $p$, and it turns out $p<n$. This means the strand whose corresponding higher local max it is connecting to is located to the left of the lower local max. Suppose the component involved is the $j$-th. Using shifting, we can determine which local max on the $j$-th component has moved down to be the closest to the lower local max in the $n$-th strip. This positions this particular crossing relative to this particular local max on the $j$-th component, and a lower local max ending in the $n$-th strip. Going next to the second highest crossing we now have three options. Either it has one or two indices in common with those of local maxima in which case it can be placed relative to local maxima as just discussed above. Using the knowledge about the positioning of the first crossing and possible shifting between the first and second crossing, we can position that second crossing relative to the other local maxima. A novelty starting at the second highest crossing is that a chord on such a crossing can possibly have no indices in common with any of the local maxima indices. Since chords are moved up however this particular chord is necessarily right below the first chord which in this case must be between one higher local max and one lower. The second chord must end on the strand connecting to the higher local max, and thus we already know what local max it is connecting to. Using shifting again we can determine what local max the strand on which the other foot is resting is connected to. Proceeding inductively by going down the list of crossings, we can position all crossings relative to the local maxima.
\end{proof}
\begin{findlocmin} \label{findlocmin}
From $Z_f(L)$, we can determine all local minima on $L$.
\end{findlocmin}
\begin{proof}
By the definition of $Z_f=Z_f[\text{\small Q}]$ for $Q>0$, for local extrema a coefficient depending on $Q$ for a tangle chord diagram of order one for which the chord is stuck between some local maximum and an ascending strand is an immediate sign that at the very bottom we have a local minimum. Using the knowledge of all crossings down the two strands involved in such a tangle chord diagram and possible shifting, we can determine where the local minimum is located. This we do for any chord between a local max and another strand. For a self chord if we do not have an $Q$ dependency, then we have a trivial knot and we can position its local minimum.
\end{proof}
\begin{Recovery} \label{Recovery}
From $Z_f(L) \in \overline{\mathcal{B}}[qN]$ we can recover the full link, and consequently we can determine $Z_f(L)$ written in terms of chord diagrams in $\overline{\mathcal{A}}(\amalg^q S^1)$.
\end{Recovery}
\begin{proof}
From Lemma \ref{detect} we determine the unique book of $M$ pages that at some definite time during the integration process over the $M$ simplex will represent $M$ chords, each one moving over one group of crossings on $L$. By the same Lemma we know the winding of strands within each group of crossings, and the corresponding number of half twists. Local maxima are located following Lemma \ref{findlocmax}, and then Proposition \ref{maxX} can place all groups of crossings relative to local maxima. The last thing to be determined is the position of local minima and this we have from Lemma \ref{findlocmin}. Thus given $Z_f(L)$, we know $L$.
In other terms, we have a map:
\beq
\setlength{\unitlength}{0.3cm}
\begin{picture}(15,3)
\thicklines
\put(0,0){$Z_f(L) \in \overline{\mathcal{B}}$}
\put(7,0.5){\vector(1,0){5}}
\put(9,1.5){$\mathfrak{z}^{-1}$}
\put(13,0){$L \in \mathbb{Z}Links$}
\end{picture}
\eeq
which induces a map:
\beq
\setlength{\unitlength}{0.3cm}
\begin{picture}(15,3)
\thicklines
\put(0,0){$Z_f(L) \in \overline{\mathcal{B}}$}
\put(7,0.5){\vector(1,0){5}}
\put(8,1.5){$\text{id} \otimes \mathfrak{z}^{-1}$}
\put(13,0){$(Z_f(L),L) \in \overline{\mathcal{B}} \otimes \mathbb{Z}Links$}
\end{picture} \nonumber
\eeq
Once we have $Z_f(L)$ in book notation, whose coefficients are those of the original framed Kontsevich integral, and we have $L$ as well, then we can place chords from books on $L$ to get corresponding tangle chord diagrams. Along with the coefficients this allows us to rewrite $Z_f(L)$ as an element of $\overline{\mathcal{A}}(L)$, from which we can recover $Z_f(L) \in \overline{\mathcal{A}}(\amalg S^1)$. In other terms we have the horizontal and vertical sides of the following triangle, and by closing it by composition, we obtain the desired map:
\beq
\setlength{\unitlength}{0.3cm}
\begin{picture}(15,10)
\thicklines
\put(-7,6.5){$Z_f(L) \in \overline{\mathcal{B}}$}
\put(1,7){\vector(1,0){5}}
\put(2,8){$\text{id} \otimes \mathfrak{z}^{-1}$}
\put(7,6.5){$(Z_f(L),L) \in \overline{\mathcal{B}} \otimes \mathbb{Z}Links$}
\put(10,6){\vector(0,-1){3}}
\put(7,1){$Z_f(L) \in \overline{\mathcal{A}}(L)$}
\put(10,0.5){\vector(0,-1){3}}
\put(7,-4.5){$Z_f(L) \in \overline{\mathcal{A}}(\amalg S^1)$}
\put(0,6){\vector(2,-3){6}}
\end{picture}
\eeq
\\
\end{proof}
\begin{Rmkidz}
The map from $\overline{\mathcal{B}}$ to $\overline{\mathcal{A}}(\amalg S^1)$ is not one-to-one. Indeed, taking the U shaped trivial circle that we called $U$, we have $\nu^{-1}=Z_f(U)$. We have also showed that flipping this knot upside down, we still get $Z_f$ of that flipped knot to be $\nu$. Those two knots however have very different books. Also, one may argue that $Z_f$ is not an isotopy invariant but $\widetilde{Z}_f$ is. Is there a map from that object to $Z_f$, and thus to the space of chord diagrams? The answer is almost certainly no, for the simple reason that $\widetilde{Z}_f(L)$ is obtained from $Z_f(L)$ by applying powers of $\nu$ to its components, thereby changing the book representation from one expression to the other. Further, having $\widetilde{Z}_f$, in order to recover $Z_f(L)$ one would need to know exactly what the coefficients of $\nu$ are, and work by hands to extract them from $\widetilde{Z}_f$ to recover the coefficients of $Z_f(L)$, and shrink the matrices by deleting all the hooks at local maxima.
\end{Rmkidz}
We would like now to illustrate the importance of using threading. Suppose such a notion is not introduced and we just consider $Z_f(L) \in \mathcal{B}[R]$. The first consequence of not using threading is that we cannot determine how many components $L$ has, and thus we cannot tell how many circles are needed to express $Z_f(L)$ in terms of chord diagrams. Suppose we know the number of link components for argument's sake. The next difficulty is to determine the book of $M$ pages given by Proposition \ref{!book}. Lemma \ref{detect} proves difficult to prove. Conjecturally, we can find the number of half twists within each group of crossings, and whether each such group is positive or negative. The exact winding remains elusive. If we conjecture that we can determine the number of half twists and the sign of each crossing, one could determine the expression of $Z_f(L)$ in terms of chord diagrams. To emphasize the power of threading however, we ask whether it is possible to determine the winding within each group of crossings from $Z_f(L)$ only. Recall that such information was obtained rather easily from Lemma \ref{detect}. We now set to determine how to obtain such windings from $Z_f(L)$ only.\\

\begin{detectXpm} \label{detectXpm}
Starting from the book of $M$ pages from Proposition \ref{!book}, we can determine the winding of the strands in each group of crossings.
\end{detectXpm}
\begin{proof}
We conjecture that it is possible to determine whether each group of crossings is positive or negative. To determine the winding on each group of crossings, it suffices to know the relative orientations of linked components: suppose two components $K_i$ and $K_j$ are linked, $1 \leq i,j \leq q$. Consider the highest group of crossings between the $i$-th and $j$-th components in the list of crossings. From this group, two strands are emanating on top, one for the $i$-th component, the other for the $j$-th component. Name those strands $K_{i,out}$ and $K_{j,out}$ locally. Let $t_{\Lambda}$ be the time at which a chord emanating from the group of crossings is located. Since both $K_{i,out}$ and $K_{j,out}$ are coming out of the group of crossings, both strands reach a local max at some point, and then are bound to go down below the group of crossings to close a loop. In particular both strands will cross the line $t=t_{\Lambda}$. Call $K_{i,down}$ and $K_{j,down}$ the portions of $K_i$ and $K_j$ respectively when each is going below $t=t_{\Lambda}$ for the first time as we trace them with our finger starting from $K_{i,out}$ and $K_{j,out}$ respectively. Since we have selected the highest group of crossings among the crossings between the $i$-th and $j$-th components, above that group there are no further crossings, but there may be some shifting. In particular if we follow a chord emanating from the group of chords, it will reach a first local max, at which point if we keep the chord moving along $K_i \amalg K_j$ it will go down to the first local min it finds. If we keep this chord moving along, it goes up to the next highest local max. At some point the chord will cross the time line $t=t_{\Lambda}$. We now work with the book of $M$ pages from Proposition \ref{!book} which we conjecture can be determined from $Z_f(L)$ alone. Its coefficient has a summand proportional to $1/2^M$, corresponding to each chord being integrated over each group of crossings. Thus if we add one page to this book, the hope is that the coefficient of this book of $M+1$ pages has a contribution $Z_{f,1}$ from the part of $K_i \amalg K_j$ strictly above $t_{\Lambda}$ that we can determine. The page we add however cannot correspond to a chord emanating from the highest group of crossings, lest the coefficient be intractable. We choose the chord as follows. We limit our number of choices for chords to three since this will be sufficient to determine the relative orientation of both components. Right above $t_{\Lambda}$, we select chords of type $(K_{i,out},K_{j,down})$, $(K_{i,down},K_{j,out})$ and $(K_{i,down},K_{j,down})$. Fix one type in what follows. Let $\Lambda$ be the separation between the two strands at the time $t_{\Lambda}$. The case $(K_{i,out},K_{j,out})$ is ruled out since it is intractable. We denote by $\pm$ the relative orientation between the portions of $K_i$ and $K_j$ of our choice, the plus sign indicating that both strands are oriented up or down, the minus sign indicating that they have mixed orientations. We have one book of $M+1$ pages corresponding to adding an additional page on top of the $M$ pages book corresponding to our initial choice for a chord, going up to the first local max it finds, at which point the chord is stuck between one local max on one component and a strand still going up on the other, and the length of the chord is denoted by $b_1$. The local max we call a turnaround. Then we consider another book of $M+1$ pages, with the additional page on top corresponding to the chord going down the other side of the turnaround until it reaches the first local min it finds, at which point the length of the chord is some number $c_1$. This point is another turnaround. The chord then moves up on the other ascending side of the turnaround until it hits the next local max. This will correspond to a further book of $M+1$ pages, the additional page representing this ascending chord. Let $b_2$ be the length of that chord when it reaches a local max. We proceed in this fashion, considering as many books of $M+1$ pages constructed as above as there are turnarounds, until one of the chords get stuck at the time $t_{\Lambda}$ since it is blocked from going lower by the already existing chord in the group of crossings. Now the question arises as to whether we can determine the contribution of those additional chords in the expression for $Z_f(L)$. We can, and we prove it in the following lemma since the proof is rather lengthy. We first present what the problem is when it comes to determining those coefficients. We have first selected a type of chord other than the problematic choice $(K_{i,out},K_{j,out})$, so the summand of the coefficient of the $M$ pages book proportional to $1/2^M$ will become proportional to $\lambda \cdot 1/2^M$ for some $\lambda$, a summand of the coefficient for the book of $M+1$ pages whose $M+1$-st page is corresponding to one of the chords just discussed. $\lambda$ is of the form $\pm (1/2 \pi i)\log b/c$, thus making the overall summand imaginary. However we cannot pinpoint this particular term among the other summands of $Z_{f,M+1}(L)$ that are purely imaginary. Assuming that we know whether each group of crossings is positive or negative, and what is its number of half twists, for the $i$-th group of crossings from the top (or bottom), $1 \leq i \leq M$, we know its contribution $\eps_i n_i \Omega /2$ to the summand of the $M$ pages book proportional to $1/2^M$, $\eps_i=1$ for a positive crossing, minus one for a negative crossing, $n_i$ the number of half twists. It follows that this particular summand has for coefficient:
\beq
\Big(\prod_{1 \leq i \leq M}\eps_i\Big)\cdot \Big(\prod_{1 \leq i \leq M}n_i \Big)\frac{\Omega^M}{2^M}:=\eps n \frac{\Omega^M}{2^M}
\eeq
with $\eps=\pm 1$.
\begin{detectkof}
We can determine the contribution $\lambda \eps n \Omega^{M+1}/2^M$ of those additional $M+1$-st pages to the coefficient of the book of $M+1$ pages.
\end{detectkof}
\begin{proof}
Among the summands of the coefficient of the $M+1$ pages book whose $M+1$-st page is of one of the three types selected, or a chord obtained by moving one such chord along $K_i$ and $K_j$ simultaneously, collect all those summands that are imaginary, one of them being $\lambda \eps n \Omega^{M+1} /2^M$, the additional power of $\Omega$ corresponding to the addition of the $M+1$-st page. Since we add a page contributing some imaginary power, it means that it is added to a region of integration in \eqref{kof} that yielded a real coefficient. We would like to isolate the term for which the contribution is exactly $\lambda \eps n \Omega^M /2^M$. However from the proof of Lemma \ref{Xpm}, the contribution in degree one of the $i$-th group of crossings is:
\beq
\pm\frac{\Omega}{2 \pi i}\log \lambda_i + \eps_i n_i \frac{\Omega}{2}
\eeq
It follows from this observation that for some values of $\lambda_i$, $1 \leq i \leq M$ and an even number of such $\log \lambda_i /2 \pi i$ being multiplied, we could get an overall power of $\eps n \Omega^M/2^M$. The same is true of the product of an even number of contributions from associators and/or local extrema. Thus in our computations it is not enough to consider only the product of the real parts of the framed Kontsevich integrals of crossings in degree one, but we also have to consider terms with a coefficient $\mu \Omega^M$ where $\mu$ gets a contribution from an even number of purely imaginary degree one contributions from crossings, associators and/or local extrema. To distinguish the desired term $\eps n \Omega^M/2^M$ from all those, we duplicate each of the $i$ pages of the $M$ pages book, $1 \leq i \leq M$. Starting from the summand whose coefficient is $\lambda\eps n \Omega^{M+1}/2^M$, by duplicating the $i$-th page, we get a book of $M+2$ pages one of whose coefficients is:
\beq
\lambda \frac{\eps n}{\eps_i n_i}\frac{\Omega^M}{2^{M-1}}\cdot \frac{1}{2}\eps_i ^2 n_i ^2 \frac{\Omega^2}{2^2}
\eeq
Continuing in this fashion, by duplicating $i$-th pages, $1 \leq i \leq M$, we end up having a term in $Z_f(L)$ of the form:
\beq
\lambda \Omega \prod_{1 \leq i \leq M}e^{\eps_i n_i \Omega /2}
\eeq
Still working with groups of crossings, if we have contributions from an even number of terms of the form $\pm\log \lambda_i /2 \pi i$ for the $i$-th page, $1 \leq i \leq M$, then by duplicating the $i$-th page we get a coefficient:
\beq
\frac{1}{2}\Big(\pm\frac{\Omega}{2 \pi i}\log \lambda_i \cdot \eps_i n_i \frac{\Omega}{2} + \frac{\Omega}{2}\eps_i n_i \cdot\pm\frac{\Omega}{2 \pi i}\log \lambda_i )=\pm\frac{\Omega}{2 pi i}\log\lambda_i\cdot \frac{\Omega}{2}\eps_i n_i
\eeq
Duplicating this $i$-th page once more, in order to get a real coefficient we need:
\beq
\frac{1}{3!}\Big(\pm\frac{\Omega}{2 \pi i}\log \lambda_i \cdot 3 \cdot \eps_i ^2 n_i ^2 \frac{\Omega^2}{2^2})=\pm\frac{\Omega}{2 \pi i}\log \lambda_i \cdot \frac{\Omega^2}{2}\eps_i ^2 n_i ^2 \frac{1}{2^2}
\eeq
Duplicating all pages, keeping in mind that we can only have an even number of contributions from $\log \lambda_i$ terms such as the one above, we would get a summand of $Z_f(L)$ of the form:
\beq
\lambda \Omega \cdot \Big(\prod_{1 \leq k \leq 2p} \pm\frac{1}{2 \pi i}\log \lambda_{i_k} \Omega )\cdot \prod_{ 1 \leq i \leq M}e^{\eps_i n_i \Omega/2}
\eeq
For contributions from associators and/or local extrema, if we consider a purely imaginary term of the form $\mu \Omega^{M+1}$, where the top $M+1$-st chord as well as an even number of contributions $\pm\log\rho_i/2 \pi i$ from associators and/or local extrema contribute to the overall factor of $\mu$, then duplicating the $i$-th page, we get a coefficient of the form:
\beq
\pm\frac{\Omega}{2 \pi i}\log \rho_i \cdot \eps_i n_i \frac{\Omega}{2}
\eeq
Duplicating this $i$-th page once more, we get:
\beq
\pm\frac{\Omega}{2 \pi i}\log \rho_i \cdot \frac{\Omega^2}{2}\eps_i^2 n_i^2 \frac{1}{2^2}
\eeq
It follows that upon doing this for all chords we end up getting a contribution to $Z_f(L)$ of the form:
\beq
\mu \Omega \cdot \frac{1}{\prod_{j \neq i_k}\eps_j n_j \frac{\Omega}{2}}\cdot \prod_{ 1 \leq k \leq 2p}\frac{1}{2 \pi i}\log \rho_{i_k} \Omega \cdot \prod_{ 1 \leq i \leq M} e^{\eps_i n_i \Omega /2}
\eeq
Next we consider contributions to $Z_{f,M+1}(L)$ that result from mixing an even number of contributions from associators and/or local extrema and/or purely imaginary contributions from crossings. The two previous computations show that we would get terms of the form:
\beq
\xi \Omega \cdot \varsigma \Omega^{2p} \cdot \prod_{ 1 \leq i \leq M} e^{\eps_i n_i \Omega /2}
\eeq
with $p \geq 1$. Finally notice that so far we have assumed that each chord is staying in a certain strip in which there is no other kind of chord, thus making the coefficients multiplicative. For those terms in $Z_{f,M+1}(L)$ corresponding to a region of integration where different chords end up in a same horizontal strip, the integration over the first time variable yields at least another iterated integral with one less integrand, times some number times $\Omega$. Repeating this procedure we end up with a number times a power of $\Omega$ times some term of the form discussed above, and upon duplication we get a term of the form:
\beq
\varphi \Omega^Q \cdot \xi \Omega \cdot \varsigma \Omega^{2p} \cdot \prod_{ 1 \leq i \leq M} e^{\eps_i n_i \Omega /2}
\eeq
where $\varphi \Omega^Q$ comes from $Q$ integrations and what remains of the above expression comes from duplicating all pages, and $p \geq 0$. It transpires of these considerations that the only term of lowest degree of the form:
\beq
\varphi \Omega^Q \cdot \xi \Omega \cdot \varsigma \Omega^{2p} \cdot \prod_{ 1 \leq i \leq M} e^{\eps_i n_i \Omega /2}
\eeq
is
\beq
\lambda \Omega \prod_{ 1 \leq i \leq M} e^{\eps_i n_i \Omega /2}
\eeq
and thus we can determine the desired contribution $\lambda \Omega$ from the additional $M+1$-st page.
\end{proof}
Armed with this lemma, we can therefore determine all the contributions from each individual $M+1$-st page added to the book of $M$ pages. Modulo $1 /(2 \pi i)$, the coefficients are:
\begin{align}
\pm \log \frac{b_1}{\Lambda}&=\alpha_1 \\
\mp \log \frac{b_1}{c_1}&=\alpha_2 \\
\pm \log \frac{b_2}{c_1}&=\alpha_3 \\
\mp \log \frac{b_2}{c_2}&=\alpha_4 \\
& \cdots \nonumber \\
\pm \log \frac{b_n}{c_{n-1}}&=\alpha_{2n-1} \\
\mp \log \frac{b_n}{c_n}&=\alpha_{2n}
\end{align}
At the point when the last chord of interest crosses the time line $t=t_{\Lambda}$, we consider the $2n$-th book of $M+1$ pages, whose coefficient is $\mp \log \frac{b_n}{c_n}=\alpha_{2n}$. We have:
\begin{align}
c_n&=b_n e^{\pm \alpha_{2n}} \\
&=c_{n-1}e^{\pm(\alpha_{2n}+\alpha_{2n-1})} \\
&=\cdots = \Lambda e^{\pm (\sum_{1 \leq i \leq 2n}\alpha_i)}
\end{align}
Since during this derivation we have chosen to work right above the highest group of crossings, we can only have shifting among strands at the top, from which it follows that we know whether $c_n > \Lambda$ or $c_n < \Lambda$. Thus we must have accordingly $ e^{\pm (\sum_{1 \leq i \leq 2n}\alpha_i)}>1$ or $ e^{\pm (\sum_{1 \leq i \leq 2n}\alpha_i)}<1$. From the knowledge of the coefficients of the $2n$ books of $M+1$ pages, we can determine the $\alpha_i$, $1 \leq i \leq 2n$, and therefore their sum. Only one sign makes the relation $e^{\pm (\sum_{1 \leq i \leq 2n}\alpha_i)}>1$ (or $ e^{\pm (\sum_{1 \leq i \leq 2n}\alpha_i)}<1$) possible. This sign is exactly the relative orientation of the portions of $K_i$ and $K_j$ we initially selected. Since we assumed that we knew the number of half-twists within each group of crossings as well as whether each such group is positive or negative, then we can tell the exact winding of each group of crossing between these two components. This we do for all linked components to fully determine crossings between all components. Self crossings are trivially known because we can determine the relative orientation of strands at each crossing.
\end{proof}
\begin{exdetect}
We illustrate in a simple example the movement of one chord emanating from the highest group of crossing and moving along two components until it crosses the line $t=t_{\Lambda}$. Consider the following portion of a tangle:
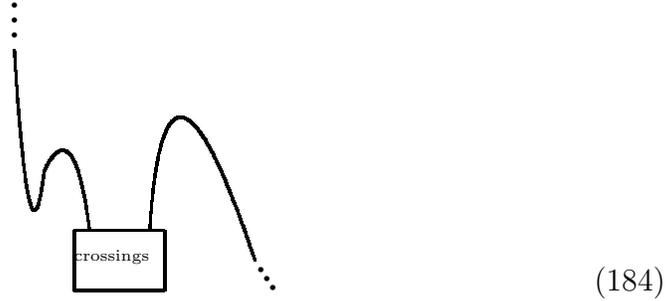
\beq
\setlength{\unitlength}{0.4cm}
\begin{picture}(9,10)
\thicklines
\put(2,0){\line(1,0){3}}
\put(2,0){\line(0,1){2}}
\put(2,2){\line(1,0){3}}
\put(5,0){\line(0,1){2}}
\put(2,1){\text{\tiny crossings}}
\qbezier(2.5,2)(2,6)(1,4)
\qbezier(1,4)(0.5,0)(0,8)
\multiput(0,8.5)(0,0.5){3}{\circle*{0.15}}
\qbezier(4.5,2)(5,10)(8,1)
\multiput(8.2,0.7)(0.2,-0.3){3}{\circle*{0.15}}
\end{picture}
\eeq
The first chord is coming out of the group of crossings at which point the separation between the strands is some number $\Gamma$, and goes up to the first local max marked by an asterisk, at which point the separation between the strands is some number $b_1$.
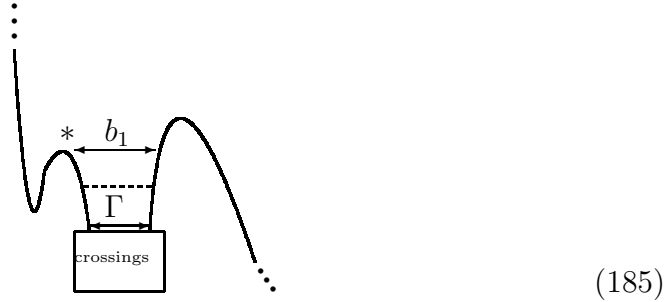
\beq
\setlength{\unitlength}{0.4cm}
\begin{picture}(9,10)
\put(2.5,2.2){\vector(1,0){2}}
\put(4.5,2.2){\vector(-1,0){2}}
\put(3,2.5){$\Gamma$}
\put(2,4.7){\vector(1,0){2.7}}
\put(4.7,4.7){\vector(-1,0){2.7}}
\put(3,5){$b_1$}
\put(1.5,5){$*$}
\thicklines
\multiput(2.3,3.5)(0.3,0){8}{\line(1,0){0.15}}
\put(2,0){\line(1,0){3}}
\put(2,0){\line(0,1){2}}
\put(2,2){\line(1,0){3}}
\put(5,0){\line(0,1){2}}
\put(2,1){\text{\tiny crossings}}
\qbezier(2.5,2)(2,6)(1,4)
\qbezier(1,4)(0.5,0)(0,8)
\multiput(0,8.5)(0,0.5){3}{\circle*{0.15}}
\qbezier(4.5,2)(5,10)(8,1)
\multiput(8.2,0.7)(0.2,-0.3){3}{\circle*{0.15}}
\end{picture}
\eeq
One cannot say that the contribution of that chord is $\pm \log b_2 /\Gamma$ since this chord can go below the group of crossings and thus becomes indistinguishable from the contribution from integrating below that group. This difficult remark aside, let us keep this chord going: the chord then reaches the turnaround marked by an asterisk and start going down the other side of the local max until it reaches the next local min at which point the separation between the strands is some number $c_1$, and the overall contribution of this tangle chord diagram is $\mp \log(b_1/c_1)$:
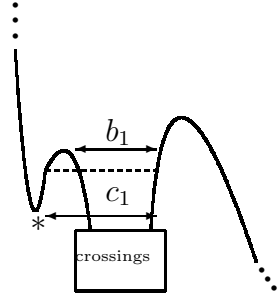
\beq
\setlength{\unitlength}{0.4cm}
\begin{picture}(9,10)
\put(0.5,2){$*$}
\put(2,4.7){\vector(1,0){2.7}}
\put(4.7,4.7){\vector(-1,0){2.7}}
\put(3,5){$b_1$}
\put(1,2.5){\vector(1,0){3.7}}
\put(4.7,2.5){\vector(-1,0){3.7}}
\put(3,3){$c_1$}
\thicklines
\multiput(1,4)(0.3,0){13}{\line(1,0){0.15}}
\put(2,0){\line(1,0){3}}
\put(2,0){\line(0,1){2}}
\put(2,2){\line(1,0){3}}
\put(5,0){\line(0,1){2}}
\put(2,1){\text{\tiny crossings}}
\qbezier(2.5,2)(2,6)(1,4)
\qbezier(1,4)(0.5,0)(0,8)
\multiput(0,8.5)(0,0.5){3}{\circle*{0.15}}
\qbezier(4.5,2)(5,10)(8,1)
\multiput(8.2,0.7)(0.2,-0.3){3}{\circle*{0.15}}
\end{picture}
\eeq
When the chord reaches the local min marked by an asterisk, it goes up until it reaches the next local max at which point the separation between the strands is some number $b_2$, for a contribution of $\pm \log (b_2/c_1)$:
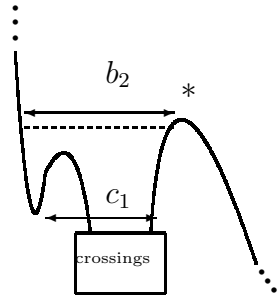
\beq
\setlength{\unitlength}{0.4cm}
\begin{picture}(9,10)
\put(5.5,6.5){$*$}
\put(1,2.5){\vector(1,0){3.7}}
\put(4.7,2.5){\vector(-1,0){3.7}}
\put(3,3){$c_1$}
\put(0.3,6){\vector(1,0){5}}
\put(5.3,6){\vector(-1,0){5}}
\put(3,7){$b_2$}
\thicklines
\multiput(0.3,5.5)(0.3,0){17}{\line(1,0){0.15}}
\put(2,0){\line(1,0){3}}
\put(2,0){\line(0,1){2}}
\put(2,2){\line(1,0){3}}
\put(5,0){\line(0,1){2}}
\put(2,1){\text{\tiny crossings}}
\qbezier(2.5,2)(2,6)(1,4)
\qbezier(1,4)(0.5,0)(0,8)
\multiput(0,8.5)(0,0.5){3}{\circle*{0.15}}
\qbezier(4.5,2)(5,10)(8,1)
\multiput(8.2,0.7)(0.2,-0.3){3}{\circle*{0.15}}
\end{picture}
\eeq
The chord then goes down the other side of the local max marked by an asterisk until it reaches a local min, at which point the separation of the strands is some number $c_2$, and we have a contribution of $\mp \log(b_2/c_2)$:
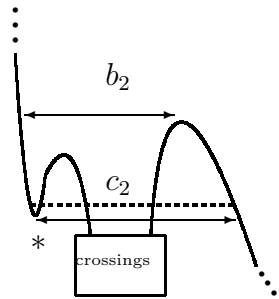
\beq
\setlength{\unitlength}{0.4cm}
\begin{picture}(9,10)
\put(0.5,1.5){$*$}
\put(0.7,2.5){\vector(1,0){6.6}}
\put(7.3,2.5){\vector(-1,0){6.6}}
\put(3,3.5){$c_2$}
\put(0.3,6){\vector(1,0){5}}
\put(5.3,6){\vector(-1,0){5}}
\put(3,7){$b_2$}
\thicklines
\multiput(0.5,3)(0.3,0){23}{\line(1,0){0.15}}
\put(2,0){\line(1,0){3}}
\put(2,0){\line(0,1){2}}
\put(2,2){\line(1,0){3}}
\put(5,0){\line(0,1){2}}
\put(2,1){\text{\tiny crossings}}
\qbezier(2.5,2)(2,6)(1,4)
\qbezier(1,4)(0.5,0)(0,8)
\multiput(0,8.5)(0,0.5){3}{\circle*{0.15}}
\qbezier(4.5,2)(5,10)(8,1)
\multiput(8.2,0.7)(0.2,-0.3){3}{\circle*{0.15}}
\end{picture}
\eeq
However the chord at this point has not crossed the line $t=t_{\Lambda}$ yet so we keep moving it along, it goes up the other side of the turnaround until it reaches the next local max marked by an asterisk, at which point the separation between the strands is some number $b_3$ and we have a contribution of $\pm \log(b_3/c_2)$:
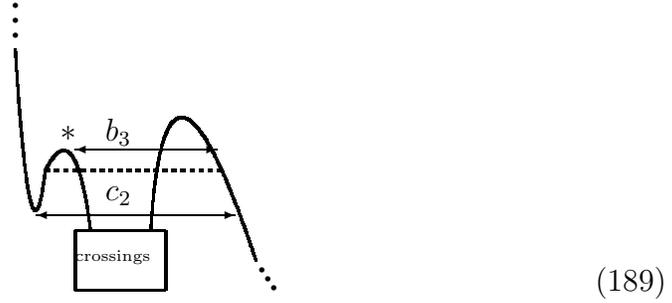
\beq
\setlength{\unitlength}{0.4cm}
\begin{picture}(9,10)
\put(1.5,5){$*$}
\put(0.7,2.5){\vector(1,0){6.6}}
\put(7.3,2.5){\vector(-1,0){6.6}}
\put(3,3){$c_2$}
\put(2,4.7){\vector(1,0){4.7}}
\put(6.7,4.7){\vector(-1,0){4.7}}
\put(3,5){$b_3$}
\thicklines
\multiput(1,4)(0.3,0){20}{\line(1,0){0.15}}
\put(2,0){\line(1,0){3}}
\put(2,0){\line(0,1){2}}
\put(2,2){\line(1,0){3}}
\put(5,0){\line(0,1){2}}
\put(2,1){\text{\tiny crossings}}
\qbezier(2.5,2)(2,6)(1,4)
\qbezier(1,4)(0.5,0)(0,8)
\multiput(0,8.5)(0,0.5){3}{\circle*{0.15}}
\qbezier(4.5,2)(5,10)(8,1)
\multiput(8.2,0.7)(0.2,-0.3){3}{\circle*{0.15}}
\end{picture}
\eeq
Finally the chord moves down the other side of the turnaround and finally crosses the time line $t=t_{\Lambda}$ at which point the strands are separated by a distance some number $c_3$ and we have a contribution of $\mp \log(b_3/c_3)$:
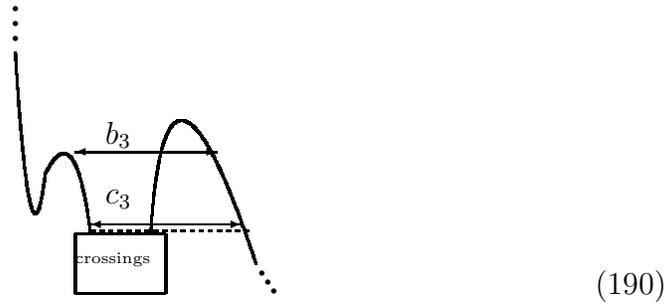
\beq
\setlength{\unitlength}{0.4cm}
\begin{picture}(9,10)
\put(2.5,2.3){\vector(1,0){5}}
\put(7.5,2.3){\vector(-1,0){5}}
\put(3,3){$c_3$}
\put(2,4.7){\vector(1,0){4.7}}
\put(6.7,4.7){\vector(-1,0){4.7}}
\put(3,5){$b_3$}
\thicklines
\multiput(2.5,2.1)(0.3,0){18}{\line(1,0){0.15}}
\put(2,0){\line(1,0){3}}
\put(2,0){\line(0,1){2}}
\put(2,2){\line(1,0){3}}
\put(5,0){\line(0,1){2}}
\put(2,1){\text{\tiny crossings}}
\qbezier(2.5,2)(2,6)(1,4)
\qbezier(1,4)(0.5,0)(0,8)
\multiput(0,8.5)(0,0.5){3}{\circle*{0.15}}
\qbezier(4.5,2)(5,10)(8,1)
\multiput(8.2,0.7)(0.2,-0.3){3}{\circle*{0.15}}
\end{picture}
\eeq
\end{exdetect}
\begin{plat}
For an unknown $q$-components link $L$ represented as a plat, from $Z_f(L)$ in book notation we can recover its expression in $\hat{\mathcal{A}}(\amalg^q S^1)$.
\end{plat}
\begin{proof}
It suffices to know the algebraic number of crossings between strands. The degree one coefficient of the following tangle (where a crossing is either positive or negative, $\lambda >0$, $\eps=\pm 1$, $n$ the number of half-twists):
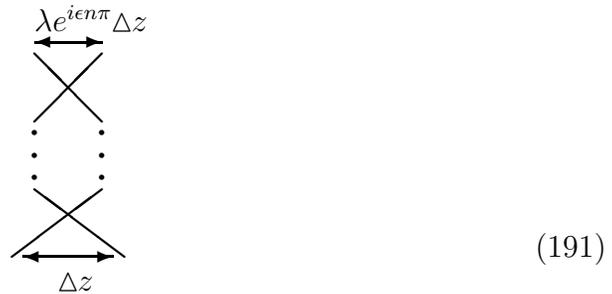
\beq
\setlength{\unitlength}{0.3cm}
\begin{picture}(7,12)
\thicklines
\put(0,0){\line(4,3){4}}
\put(5,0){\line(-4,3){4}}
\multiput(1,3.5)(0,1){3}{\circle*{0.2}}
\multiput(4,3.5)(0,1){3}{\circle*{0.2}}
\put(1,6){\line(1,1){3}}
\put(4,6){\line(-1,1){3}}
\put(0.5,0){\vector(1,0){4}}
\put(4.5,0){\vector(-1,0){4}}
\put(2,-1.5){$\vartriangle \!\!z$}
\put(1,9.5){\vector(1,0){3}}
\put(4,9.5){\vector(-1,0){3}}
\put(1,10){$\lambda e^{i \eps n \pi}\!\! \vartriangle \!\! z$}
\end{picture}
\eeq
in $Z_{f,1}(L)$ is:
\beq
\pm \frac{1}{2 \pi i}\log \frac{\lambda e^{i \eps n \pi} \vartriangle \!\! z}{\vartriangle \!\!z}
=\pm \Big(\frac{1}{2 \pi i}\log \lambda \Big) \;\pm \eps \frac{n}{2}
\eeq
Thus if we have an odd number of half-twists between 2 strands, $n/2$ is fractional, an integer otherwise. It follows that from $Z_{f,1}(L)$ we can determine for a fixed strand the collection of strands to its right near the top of $L$ it has an odd number of half-twists with, or equivalently what are those strands to its right near the top of $L$ it passes on the right at the bottom of $L$. Doing this for all strands we can determine the permutation that to any strand near the top of $L$ associates a strand near the bottom of the link. This is sufficient to associate to any book the element of $\mathcal{A}(\amalg^q S^1)$ it corresponds to.
\end{proof}

\end{document}